\documentclass[11pt]{article}

\usepackage{epsfig}
\usepackage{amssymb, latexsym}
\usepackage{amscd}
\usepackage[all]{xy}
\usepackage{epsf}
\usepackage[mathscr]{eucal}
\usepackage{tikz}
\usepackage{tikz-cd}
\usepackage{verbatim}
\usetikzlibrary{matrix,arrows}
\usepackage{mathtools}
\usepackage{xcolor}
\usepackage{stmaryrd}
\usetikzlibrary{positioning,fit,calc}

\usepackage{amsmath,amsfonts,amscd,amssymb,amsthm}

\usepackage[titletoc]{appendix}
\usepackage[sc]{titlesec}

\long\def\comment#1\endcomment{}

\theoremstyle{plain}

\newtheorem{theorem}{\sc Theorem}[section]
\newtheorem{lemma}[theorem]{\sc Lemma}

\newtheorem{prop}[theorem]{\sc Proposition}
\newtheorem{coroll}[theorem]{\sc Corollary}

\theoremstyle{plain}

\newtheorem{defn}[theorem]{\sc Definition}

\theoremstyle{exercise}
\newtheorem{remark}[theorem]{\sc Remark}

\newtheorem{example}[theorem]{\sc Example}

\makeatletter
\newcommand*{\doublerightarrow}[2]{\mathrel{
  \settowidth{\@tempdima}{$\scriptstyle#1$}
  \settowidth{\@tempdimb}{$\scriptstyle#2$}
  \ifdim\@tempdimb>\@tempdima \@tempdima=\@tempdimb\fi
  \mathop{\vcenter{
    \offinterlineskip\ialign{\hbox to\dimexpr\@tempdima+1em{##}\cr
    \rightarrowfill\cr\noalign{\kern.5ex}
    \rightarrowfill\cr}}}\limits^{\!#1}_{\!#2}}}
\makeatother

\makeatletter \@addtoreset{equation}{section} \makeatother

\def\eqref#1{\thetag{\ref{#1}}}

\let\latexref=\ref
\def\ref#1{{\normalfont{\latexref{#1}}}}

\newcommand{\ldot}{{\:\raisebox{2,3pt}{\text{\circle*{1.5}}}}}
\newcommand{\udot}{{\:\raisebox{3pt}{\text{\circle*{1.5}}}}}
\def\dlim_#1{{\displaystyle\lim_{#1}}^\hdot}

\newcommand{\id}{\operatorname{\rm id}}

\newcommand{\eqto}{\mathrel{\stackrel{\sim}{\to}}}

\newcommand{\Ob}{\mathrm{Ob}}

\newcommand{\Mod}{\mathrm{Mod}}
\newcommand{\opp}{\mathrm{op}}
\renewcommand{\Bar}{\mathrm{Bar}}

\newcommand{\GS}{\mathrm{GS}}
\newcommand{\Hom}{\mathrm{Hom}}
\newcommand{\Bimod}{\mathrm{Bimod}}

\newcommand{\free}{\mathrm{free}}

\newcommand{\RHom}{\mathrm{RHom}}
\newcommand{\Hoch}{\mathrm{Hoch}}

\newcommand{\op}{\mathrm{op}}

\newcommand{\Mon}{{\mathscr{M}on}}

\newcommand{\colim}{\mathrm{colim}}
\newcommand{\Id}{\mathrm{Id}}

\newcommand{\Sets}{\mathscr{S}ets}

\newcommand{\Vect}{\mathscr{V}ect}

\newcommand{\Cat}{{\mathscr{C}at}}
\newcommand{\Top}{{\mathscr{T}op}}

\newcommand{\Fun}{{\mathrm{Fun}}}

\newcommand{\str}{\mathrm{str}}

\renewcommand{\k}{\Bbbk}

\newcommand{\pprime}{{\prime\prime}}

\newcommand{\tot}{\mathrm{tot}}

\newcommand{\Ord}{\mathrm{Ord}}

\newcommand{\Tot}{\mathrm{Tot}}

\renewcommand{\emptyset}{\varnothing}

\newcommand{\height}{\mathrm{ht}}

\newcommand{\codim}{\mathrm{codim}}
\renewcommand{\min}{\mathrm{min}}
\renewcommand{\max}{\mathrm{max}}
\newcommand{\Lan}{\mathrm{Lan}}
\newcommand{\Ran}{\mathrm{Ran}}
\newcommand{\lk}{\mathbf{lk}}

\newcommand{\norm}{\mathrm{norm}}
\renewcommand{\colim}{\mathrm{colim}}
\newcommand{\Nat}{\mathrm{Nat}}
\newcommand{\Bicat}{\mathrm{Bicat}}

\newcommand{\sevafigc}[4]{\begin{figure}[h]\centerline{
 \epsfig{file=#1,width=#2,angle=#3}}
\bigskip\caption{#4}\end{figure}}

\title{\sc The category $\Theta_2$, derived modifications, and \\ deformation theory of monoidal categories}

\author{\sc Piergiorgio Panero and Boris Shoikhet}
\date{}

\begin{document}\maketitle
{\footnotesize
\begin{center}{\parbox{4,5in}{{\sc Abstract.}
A complex $C^\udot(C,D)(F,G)(\eta, \theta)$, generalising the Davydov-Yetter complex of a monoidal category [D], [Y], is constructed. Here $C,D$ are $\k$-linear (corresp., dg) bicategories, $F,G\colon C\to D$ are $\k$-linear (corresp., dg) strong functors, $\eta,\theta\colon F\Rightarrow G$ are strong natural transformations. Morally, it is a complex of ``derived modifications'' $\eta \Rrightarrow \theta$, likewise for the case of dg categories one has the complex of ``derived natural transformations'' $F\Rightarrow G$, given by the Hochschild cochain complex of $C$ with coefficients in $C$-bimodule $D(F-,G=)$.

The complex $C^\udot(C,D)(F,G)(\eta,\theta)$ naturally arises from a 2-cocellular dg vector space $A(C,D)(F,G)(\eta,\theta)\colon \Theta_2\to C^\udot(\k)$, as its $\Theta_2$-totalization (here $\Theta_2$ is the category dual to the category of Joyal 2-disks, [J]). 

It is shown that for a $\k$-linear monoidal category $C$ the third cohomology $H^3(C^\udot(C,C)(\Id,\Id)(\id,\id)))$ is isomorphic to the vector space of the outer (modulo twists) infinitesimal deformations of the $\k$-linear monoidal category which we call the {\it full} deformations. It means that the following data is to be deformed: (a) the underlying dg category structure, (b) the monoidal product on morphisms (the monoidal product on objects is a set-theoretical datum and is maintained under the deformation), (c) the associator. The data (a), (b), (c) is subjects to the (infinitesimal versions of) numerous monoidal compatibilities, which we interpret as the closeness of the corresponding degree 3 element. 
Similarly, $H^2(C^\udot(C,D)(F,F)(\id,\id))$ is isomorphic to the vector space of the outer infinitesimal deformations of the strong monoidal functor $F$.

A relative totalization $Rp_*A(C,D)(F,F)(\id,\id)$ along the projection $p\colon \Theta_2\to\Delta$ is defined, and it is shown to be a cosimplicial monoid, which fulfils the Batanin-Davydov 1-commutativity condition [BD]. Then it follows from loc.cit. that $C^\udot(C,D)(F,F)(\id,\id)$ is a $C_\ldot(E_2;\k)$-algebra. 
Conjecturally, $C^\udot(C,C)(\Id,\Id)(\id,\id)$ is a $C_\ldot(E_3;\k)$-algebra; however the proof requires more sophisticated methods.

}}
\end{center}
}

\newpage

\section*{\sc Introduction}
\subsection{\sc }
In formal deformation theory, with a deformation functor (defined on a suitable category of commutative differential graded (dg) coalgebras) one associates a differential graded Lie algebra (or, more generally, an $L_\infty$ algebra), whose completed chain Chevalley-Eilenberg complex pro-represents the deformation functor. In characteristic 0, such representability has been a proven statement ([L-DAGX], [P], [GLST]). Sometimes the underlying complex of the dg Lie algebra (corresp., of an $L_\infty$ algebra) is called the deformation complex. In many particular deformation problems, such complex is easier to find than the entire dg Lie algebra structure. The first cohomology of the deformation complex is isomorphic to the infinitesimal deformations mod out equivalences, whence the dg Lie algebras encodes, via the Maurer-Cartan equation, the formal (global) deformations mod out formal equivalences. When the ``base of deformation'' can be chosen as a $E_n$-coalgebra (in a suitable sense, e.g., a coalgebra over the dg operad $C_\ldot(E_n;\k)$, or a coalgebra over the Koszul resolution $\mathrm{hoe}_n$ of the operad $e_n=H_\ldot(E_n;\k)$, etc.), such that it gives rise to a deformation functor defined on a suitable category of $E_n$-coalgebras,  the shifted by $[n]$ deformation complex has a structure of $E_{n+1}$-algebra, [Tam2], [L-DAGX].

The higher structures on a deformation complex are important if we are interested by the formality phenomena. For instance, the Deligne conjecture asserts that the Gerstenhaber bracket on the Hochschild cohomological complex of a (dg) algebra $A$ (which defines a dg Lie algebra structure) can be ``extended'' to a $C_\udot(E_2;\k)$-algebra structure (it has found several proofs, see [MS1,2], [KS], [Tam1], [BB]). It was used by Tamarkin [Tam3] in his proof of (a stronger version of) the Kontsevich formality theorem [K]. The idea is roughly that the more higher structure on the deformation complex we consider, the more rigid the deformation complex with this structure becomes. Thus, the original formality theorem of Kontsevich was stated for dg Lie algebra structure on the Hochschild cohomological complex of a polynomial algebra, and was proven by methods inspired by the Topological Quantum Field Theory.
The idea of Tamarkin was to consider the entire higher structure of homotopy 2-algebra on the complex, and using the aforementioned rigidity, it can be proven by homotopy theoretical methods. The ``transcendental part'' of the proof becomes hidden in a solution to the Deligne conjecture.

\subsection{\sc }
In this paper, we are interested in deformation theory of monoidal dg categories (and more generally of dg bicategories). 
Our interest originates in (partly open) deformation theory of associative bialgebras. In this case, the deformation complex was constructed by Gerstenhaber and Schack [GS], its intrinsic interpretation in terms of abelian category of tetramodules over the bialgebra was given in [Sh1]. This interpretation made it possible to compute the Gerstenhaber-Schack cohomology for the case of $B=S(V)$, the (co)free (co)commutative bialgebra.  The answer was (as it had been conjectured by Kontsevich) 
$$
H^k_{\GS}(S(V))=\bigoplus_{a+b=k}\Hom(V^{\otimes a}, V^{\otimes b})=S^k(V^*[-1]\oplus V[-1])
$$
The symmetric algebra $H^\udot(B){\GS}=S(V^*[-1]\oplus V[-1])$ has a Poisson algebra structure of degree -2, which comes from the degree -2 pairing $V^*[-1]\oplus V^*[-1]\to\k$. It gives, along with the graded commutative product, a structure of $e_3$-algebra, where $e_3=H_\ldot(E_3;\k)$. An interesting and important open question is how one can lift this structure to the Gerstenhaber-Schack complex for $B=S(V)$, or for general $B$.

Motivated by the problem of finding the higher structures on the (Gerstenhaber-Schack) deformation complex of a bialgebra $B$, we consider in this paper the deformation theory of a (dg or $\k$-linear) monoidal category $C$. In the example associated to deformation theory of $B$, the monoidal category $C=\Mod(B)$ is the category of left modules over the underlying algebra $B$, it is known to be a monoidal category: for $M,N\in \Mod(B)$, the tensor product $M\otimes_\k N$ is a $B\otimes_\k B$-module, then the precomposition along the coproduct map $\Delta\colon B\to B\otimes_\k B$ makes $M\otimes_\k N$ a $B$-module. The counit map $\varepsilon\colon B\to \k$ makes $\k$ the unit for this monoidal structure. 

The link between the deformations of $\Mod(B)$ and the deformations of $B$ is partially established by the Tannaka-Krein duality (though in what concerns the deformation complexes this link has to be understood better). An advantage of the monoidal category approach is that the higher structures on the deformation complex of a dg monoidal category are more manageable and can be explicitly found. A recent paper [BD], where the authors deal with a truncation of our complexes, called the Davydov-Yetter complex of a monoidal category. In loc.cit. the authors constructed a $C_\udot(E_3;\k)$-algebra structure on this truncated deformation complex. The Davydov-Yetter complex controls only the (infinitesimal) deformations of the associator, whence our complex controls the (infinitesimal) deformations of all linear data (the associator, the underlying dg category, the morphisms part of the product bifunctor, see (A1)-(A4) in Section \ref{sectionlast}). 

Work in progress [Sh3] aims to find a structure of $C_\ldot(E_3;\k)$-algebra on the (non-truncated) deformation complex, employing the technique of Batanin $n$-operads [Ba1-3]. 
According to [Ba3] an action of a contractible $(n-1)$-terminal  (pruned and reduced) $n$-operad on a complex gives, via the the symmetrisation functor and the cofibrant replacement, an action of the chain operad $C_\ldot(E_n;\k)$ on it. In our opinion, the theory of $n$-operads provides, via the aforementioned result, a very flexible and powerful approach to higher generalisations of the Deligne conjecture. In [BM1] a version of Deligne conjecture for general $n$ is stated, and it is proven in [BM2] for $n=2$.

\subsection{}
Let $\k$ be a field. 
We consider $\k$-linear bicategories, see [Ke], [L], we recall the basic definitions related to (enriched) bicategories in Section \ref{sectionbicat} . A particular case of a ($\k$-linear) bicategory with a single onject is a ($\k$-linear) monoidal category.
For a $\k$-linear monoidal category $C$, we provide a complex $C^\udot(C,C)(\Id,\Id)(\id,\id)$, whose 3rd cohomology is proven to parametrise the infinitesimal deformations of $C$ mod out the infinitesimal equivalences (see Theorem \ref{infdefcat}). We also construct more general complexes. More precisely, for dg monoidal categories (resp., dg bicategories) $C,D$, two strong $\k$-linear monoidal (resp., strong bicategorical) functors $F,G\colon C\to D$, and two strong bicategorical natural transformations $\eta,\theta\colon F\Rightarrow G$, we provide a complex $C^\udot(C,D)(F,G)(\eta,\theta)$, whose 0-th cohomology is equal to the modifications $\eta\Rrightarrow\theta$ (playing the role of 3-morphisms for the tricategory of bicategories, see Section \ref{sectionbicat}).
The entire complex $C^\udot(C,D)(F,G)(\eta,\theta)$ (or rather the closed elements of it) plays the role of {\it derived modifications}.

Thus what we are dealing with here is a set-up for further theory one level (dimension) higher than the one developed in Tamarkin's paper ``What do dg categories form?'' [Tam1]. The work in progress [Sh3] aims to construct a contractible Batanin 3-operad [Ba2,3] acting on the corresponding 3-quiver (whose underlying 2-quiver is a strict 2-category), which, among the other things, would provide a homotopy 3-algebra structure on $C^\udot(C,C)(\Id,\Id)(\id,\id)$. Here we construct the 3-quiver itself. 

The complexes $C^\udot(C,D)(F,G)(\eta,\theta)$ emerge as the totalization\footnote{By $\Theta_2$-totalization we mean here the corresponding (non-normalized) cochain Moore complex, as it is defined in \eqref{tott}, \eqref{totd}, see also discussion in Section \ref{remtot}} of 2-cocellular complexes, that is, of functors 
$A(C,D)(F,G)(\eta,\theta)\colon \Theta_2\to C^\udot(\k)$. Here $\Theta_2$ is the category dual to the category of Joyal 2-disks [J], [B1,2]. 

\subsection{}\label{dyintro}
Our complex $C^\udot(C,C)(\Id,\Id)(\id,\id)$ can be thought of as a relaxed version of the Davydov-Yetter complex [D], [Y] of a monoidal category, which fits better for aims of deformation theory. . 
Recall that the Davydov-Yetter cochains in degree $n$ are {\it natural} transformations from the functor $M^n$ to itself, where $M^n(X_1,\dots,X_n)=
X_1\otimes(X_2\otimes(\dots(X_{n-1}\otimes X_n)\dots)$. It gives rise to a cosimplicial (dg) vector space. Then the Davydov-Yetter complex is defined as the totalization of this cosimplicial vector space. 

Often the Davydov-Yetter complex is said to compute the infinitesimal deformations of a monoidal category. 
In fact,  it is not quite correct, the Davydov-Yetter complex only encodes the deformations of the associator, leaving the underlying $\k$-linear category and the monoidal product {\it on morphisms} fixed.

In our set-up of $\k$-linear monoidal category, it is natural to deform all $\k$-linear data. More precisely, we assume that the monoidal product {\it on objects} remains fixed, while the underlying $\k$-linear category, the monoidal product {\it of morphisms}, and the associator are being deformed. We refer to such deformations as {\it full}. 

Theorem \ref{infdefcat} states that $H^3(C^\udot(C,C)(\Id,\Id)(\id,\id))$ parametrises the infinitesimal {\it full} deformations of $C$ mod out the infinitesimal equivalences. 

The link between $C^\udot(C,C)(\Id,\Id)(\id,\id)$ and the Davydov-Yetter complex can be described as follows. 
The natural embedding $\Delta\times\Delta\to\Theta_2$ makes it possible to consider\\ 
$C^\udot(C,C)(\Id,\Id)(\id,\id)$ as a bicomplex, with the horizontal differential $d_0$ and the vertical differential $d_1$. One can show that $C_{DY}^\udot(C)$ is the kernel of the vertical differential $d_1$ restricted to the 0-th row of this bicomplex (with the differential $d_0$). 

The naturality of the Davydov-Yetter cochains is dropped in $C^\udot(C,C)(\Id,\Id)(\id,\id)$, and is replaced by the naturality with respect to monoidal structural maps. 
The reader is advised to look directly to Section \ref{sectiondygen} for more detail on connection between the Davydov-Yetter and our complexes and on this restricted naturality for cochains. 

One of our motivations here was the recent paper [BD], where a $C_\ldot(E_3;\k)$ algebra structure on the Davydov-Yetter complex $C_{DY}^\udot(C)$ of a $\k$-linear monoidal category $C$ was constructed. In [BD], the authors consider more generally $n$-commutative cosimplicial monoids, and prove that the totalization of such cosimplicial monoid has a structure of homotopy $(n+1)$-algebra. 
On the other hand, it is shown in [BD] that $C_{DY}^\udot(C)$ is the totalization of a 2-commutative cosimplicial monoid, which implies that $C_{DY}^\udot(C)$ is a homotopy 3-algebra. 

Unfortunately, the 2-commutativity fails for our complex (even for the strict case). More precisely, there is a natural projection $p\colon \Theta_2\to\Delta$, which defines a cosimplicial complex $R^\udot p_*(A(C,D)(F,G)(\eta,\theta))$. \footnote{One can not state that $R^\udot p_*(A(C,D)(F,G)(\eta,\theta))$ is the right homotopy Kan extension, because $A(C,D)(F,G)(\eta,\theta)$ fails to be Reedy fibrant. One can alternatively define $R^\udot p_*(\dots)$ as ``relative totalization'', see Section \ref{remtot}.} 
For the case $F=G$, $\eta=\theta=\id$, this cosimplicial complex is in turn a cosimplicial (dg) monoid. 
One easily shows that the $\Delta$-totalization of the latter monoid is equal to the $\Theta_2$-totalization of $A(C,D)(F,G)(\eta,\theta)$. However, the cosimplicial monoid $Rp_*(A(C,C)(\Id,\Id)(\id,\id)$ fails to be 2-commutative. (It is, in a sense, a {\it homotopy 2-commutative monoid}, a cosimplicial monoid in which the 2-commutativity relation holds only up to homotopy, in a suitable sense, which conjecturally should be enough for its totalization to be a homotopy 3-algebra).\footnote{It would be interesting to define the concept of a cosimplicial homotopy $n$-commutative monoid; its totalization should be an algebra over $C^\udot(E_{n+1};\k)$.}

On the other hand, for a strong bicategorical $\k$-linear functor $F\colon C\to D$, the cosimplicial monoid $R^\udot p_*A(C,D)(F,F)(\id,\id)$ is 1-commutative (Proposition \ref{propneq2}). Then it follows from [BD], Corr. 2.46 that the complex $C^\udot(C,D)(F,F)(\id,\id)$ is a homotopy 2-algebra (Theorem \ref{defftheorem}).

\subsection{\sc Organisation of the paper}
The paper consists of 6 Sections and 2 Appendices. 

In Section 1, we recall definitions and well-known results on the categories $\Theta_n$, Joyal $n$-disks, and their interplay. None of the results of Section 1 is new, we basically follow [B1,2] and [J].

In Section 2, we define the totalization of a 2-cocellelular complex (that is, of a functor $\Theta_2\to C^\udot(\k)$), as well as the relative totalization along the projection $p\colon \Theta_2\to \Delta$. For $X\colon \Theta_2\to C^\udot(\k)$, the relative totalization $Rp_*(X)$ is a functor $\Delta\to C^\udot(\k)$. In Proposition \ref{prop!}, we prove, for any $X$ as above, the transitivity property for its totalizations: $\Tot_{\Theta_2}(X)=\Tot_\Delta(Rp_*(X))$. 

In Section 3, we recall the basic notions related to bicategories, and introduce a (seemingly, new) concept of a 2-bimodule over a bicategory. We introduce a ``bicategorical'' version $\hat{\Theta_2}$ of the category $\Theta_2$ and study the left Kan extension along the projection $\hat{\Theta}_2^\op\to \Theta_2^\op$. 

In Section 4, we introduce our main new construction, the 2-cocellular complexes\\ $A(C,D)(F,G)(\eta,\theta)$. The definition is more tricky than one could expect, namely, the components $A(C,D)(F,G)(\eta,\theta)_T$ are subcomplexes of the corresponding components of a ``more natural'' complex $\hat{A}(C,D)(F,G)(\eta,\theta)_T$. The passage from $\hat{A}$ to $A$ is performed by  imposing the {\it bicategorical relations} \eqref{relmain11}-\eqref{relmain33} and taking the corresponding subspaces; without that, the assignment $T\rightsquigarrow\hat{A}(C,D)(F,G)(\eta,\theta)_T$ itself fails to be 2-cocellular.

In Section 5, we employ results of [BD] for studying higher structures on the complexes $C^\udot(C,D)(F,G)(\eta,\theta)$. It is possible due to the transitivity property of Proposition \ref{prop!}. We prove that $C^\udot(C,D)(F,F)(\id,\id)$ is a homotopy 2-algebra, for any strong $\k$-linear 2-functor, in Theorem \ref{defftheorem}.

Section 6 contains an identification of\\ $H^3(C^\udot(C,C)(\Id,\Id)(\id,\id))$ with infinitesimal full deformations of $C$ mod out infinitesimal equivalences. It justifies our complexes $C^\udot(C,D)(F,G)(\eta,\theta)$ as related to deformation theory of monoidal $\k$-linear categories (and, more generally, of $\k$-linear bicategories).

In Appendix A we list the relations between (co)dimension 1 operators in $\Theta_2$, used throughout the paper. Appendix B contains a (computational) proof of Proposition \ref{propdeltaaction}.

\subsection*{\sc}
\subsubsection*{\sc Acknowledgements}
The authors are thankful to Michael Batanin for his interest and suggestions. \\
The work of Piergiorgio Panero was supported by the FWO Research Project Nr. G060118N.\\
The work of Boris Shoikhet was supported by Support grant
for International Mathematical Centres Creation and Development, by the Ministry of Higher Education and Science
of Russian Federation and PDMI RAS agreement № 075-15-2022-289 issued on April 6, 2022.

\section{\sc The categories $\Theta_n$}
Here we recall the definition of the categories $\Theta_n$, $n\ge 1$, and some related concepts. 
\subsection{\sc $n$-ordinals and $n$-leveled trees}\label{nord}
We denote by $[n]$ the ordinal $0<1<\dots<n$ having $n+1$ elements.  Recall that the simplicial category $\Delta$ has objects $[n], n\ge 0$, that is, all non-empty finite ordinals. 
Its morphisms are the mononotonous maps $f\colon [k]\to [\ell]$, that is, $f(i)\le f(j)$ if $i\le j$. 

Recall the relations between the standard elementary face operators  $\partial^i\colon [n-1]\to [n]$ and the elementary degeneracy operators $\varepsilon^i\colon [n+1]\to [n]$, $i=0,\dots,n$,  in $\Delta$:

\begin{equation}\label{deltarel}
\begin{aligned}
\ &\partial^j\partial^i=\partial^i\partial^{j-1}\text{   if   }i<j\\
&\varepsilon^j\varepsilon^i=\varepsilon^i\varepsilon^{j+1}\text{   if   }i\le j\\
&\varepsilon^j\partial^i=\begin{cases}
\partial^i\varepsilon^{j-1}&\text{   if   }i<j\\
\id&\text{   if   }i=j,j+1\\
\partial^{i-1}\varepsilon^j&\text{   if   }i>j+1
\end{cases}
\end{aligned}
\end{equation}

An ordinal as above is also called a 1-ordinal. The following definition is due to M.Batanin:
\begin{defn}\label{defnord}{\rm 
		{\it An $n$-ordinal} $S$ is a sequence of surjective maps in $\Delta$:
		\begin{equation}\label{eq1.0}
[k_n]\xrightarrow{\rho_{n-1}} [k_{n-1}]\xrightarrow{\rho_{n-2}}\dots \xrightarrow{\rho_1} [k_1]\xrightarrow{\rho_0} [0]	
	\end{equation}
	The category $\Ord_n$ of $n$-ordinals has all $n$-ordinals as its objects, and the morphisms $S\to T$ are commutative diagrams
	\begin{equation}\label{eq1}
	\xymatrix{
	[k_n]\ar[r]^{\rho_{n-1}}\ar[d]^{f_n} &[k_{n-1}]\ar[r]^{\rho_{n-2}}\ar[d]^{f_{n-1}}&\dots \ar[r]^{\rho_1}& [k_1]\ar[r]^{\rho_0}\ar[d]^{f_1}& [0]\ar[d]^{\id}	\\
	[\ell_n]\ar[r]^{\rho_{n-1}^\prime} &[\ell_{n-1}]\ar[r]^{\rho_{n-2}^\prime}&\dots \ar[r]^{\rho_1^\prime}& [\ell_1]\ar[r]^{\rho_0^\prime}& [0]
}
\end{equation}
in which $f_1,\dots,f_n$ are not necessarily maps in $\Delta$, but a {\it weaker condition} holds: for any $a\in [k_j]$ the restriction of $f_{j+1}$ to $\rho_j^{-1}(a)$ is order-preserving. {\it Remark:} note that $f_j$'s are not necessarily morphisms in $\Delta$, the above condition is weaker than order-preserving. 

An object of $\Ord_n$ is a non-empty $n$-ordinal. 

When the assumption that the maps $\{\rho_i\}$ are surjective is dropped, the object \eqref{eq1.0} is called {\it an $n$-level tree} (or, shortly, an $n$-tree). A morphism of $n$-trees is defined as in \eqref{eq1}. 
}
\end{defn}

The difference between the $n$-ordinals and the $n$-level trees is that the former have all input vertices at the top $n$-th level, whence an $n$-tree may have input vertices at all levels. Sometimes $n$-ordinals are called {\it pruned} $n$-trees. 

\comment
One can define an $n$-tree alternatively as follows:
\begin{defn}{\rm
A $n$-leveled pre-tree $T$ is a collection of finite ordered sets $\{T(i)\}_{0\le i\le n}$ endowed with a map $i_T\colon T_{\ge 1}\to T$ which lowers the level by 1, such that $T(0)$ is a 1-element set, and such that the sets $i_T^{-1}(x)$, $x\in T$, are linearly ordered.  An $n$-level tree is an isomorphism class of an $n$-level pre-trees. 
}
\end{defn}
Given an $n$-ordinal as above, the associated leveled tree $T$ has $T(i)=k_i+1$, and $i_T|_{T(i)}=\rho_{i-1}$. 
\endcomment

Introduce some terminology related to leveled trees, which is used later in the paper.

We represent an $n$-level trees as a collection of finite sets $\{T(i)\}_{0\le i\le n}$ endowed with a map $i_T\colon T_{\ge 1}\to T$ which lowers the level by 1. The map $i_T$ is defined as $\rho_{i-1}$ at level $i$.

For $x\in T(i)$ we write $\height(x)=i$. By definition, $n=\height(T)=\max_{x\in T}\height(x)$.  A vertex $x$ of a leveled tree is called an {\it input}, or a {\it leaf}, if $i_T^{-1}(x)=\emptyset$. Note that for an $n$-leveled tree, the height of an input may be smaller than $n$ (but it always exists an input of the height $n$). 

An {\it edge} is a pair $(x,y)$ with $x=i_T(y)$. The set of edges of $T$ is denoted by $e(T)$. We define the {\it dimension} $d(T)=\sharp e(T)$. A leveled tree is called {\it linear} if $d(T)=\height(T)$. 

For each vertex $x\in T$, the ordered set of incoming edges $e_x(T)$ is defined as $i_T^{-1}(x)$. 

For a leveled tree $T$ define a leveled tree $\bar{T}$ as follows. For each $x\in T$, we set $e_x(\bar{T})=e_x(T)\cup (x,x_{-})\cup (x,x_+)$ with the order in which $(x,x_-)$ is the minimal element and $(x,x_+)$ is the maximal element. Thus we add the leftmost and the rightmost element to each set $e_x(T)$. It results in $\bar{T}(i)=T(i)+2T(i-1)$, and $\height(\bar{T})=\height(T)+1$. 
A {\it $T$-sector} of height $k$ is a triple $(x;y_L,y_R)$ where $x\in T(k)$, $y_L,y_R\in\bar{T}(k+1)$, $i_{\bar{T}}(y_L)=i_{\bar{T}}(y_R)=x$, and $y_L,y_R$ are {\it consequtive} elements of $\bar{T}(k+1)$. We say that $x$ {\it supports} a sector $(x;y_L,y_R)$  It follows that each input vertex $x$ of $T$ supports a unique sector (which is $(x;x_-,x_+)$).

\subsection{\sc The wreath product definition of $\Theta_n$}
The definition of the categories $\Theta_n$, $n\ge 1$, is given inductively via the {\it wreath product} $\Delta\wr \mathscr{A}$, see below. 
\begin{defn}\label{wreathdef}{\rm
Let $\mathscr{A}$ be a category. The objects of the category $\Delta\wr\mathscr{A}$ are tuples $([\ell]; A_1,\dots,A_\ell)$, where $A_1,\dots,A_\ell\in\mathscr{A}$. A morphism $\Phi\colon ([\ell];A_1,\dots,A_\ell)\to ([m]; B_1,\dots,B_m)$ is a tuple $(\phi; \phi_1,\dots,\phi_\ell)$ where $\phi\colon [\ell]\to [m]$ is a morphism in $\Delta$, and $\phi_i=(\phi_i^{\phi(i-1)+1},\dots,\phi_i^{\phi(i)})$ is a tuple of morphisms in $\mathscr{A}$, with $\phi_i^k\colon A_i\to B_k$. The composition is defined in the natural way. 
}
\end{defn}
The reader is adviced to look at Lemma \ref{lemmanice} which explains a natural framework in which the category $\Delta\wr\mathscr{A}$ emerges.

We set: 
\begin{equation}\label{defthetawp}
\Theta_1=\Delta \text{ and }\Theta_n=\Delta\wr \Theta_{n-1}, n\ge 2
\end{equation}

\begin{remark}{\rm
The case $\ell=0$ is allowed for an object of $\Theta_n$. In this case the object $([0];\emptyset)$ is final.
}
\end{remark}

\comment
There is an isomorphism between the set of objects of the category $\Theta_n$ and the set of $n$-ordinals (possibly empty).
Indeed, for $n=1$, to the object $[m]\in \Theta_1$ one associated the 1-ordinal $[m-1]$. We write $t([m])=[m-1]$. 
Assume we have already identified the objects of $\Theta_{n-1}$ with (possibly empty) $(n-1)$-ordinals.
Let $A=([m], A_1,\dots,A_m)$ be an object of $\Theta_n=\Delta\wr\Theta_{n-1}$. Denote by $t(A_1),\dots, t(A_m)$ the corresponding $(n-1)$-ordinals. Then the $n$-ordinal $t(A)$ is obtained by the ``grafting'' operation. The corresponding $n$-leveled tree has $[m-1]$ in level 1, and then the $(n-1)$-leveled tree $t(A_i)$ is ``attached'' to $i-1\in [m-1]$, $i=1,\dots,m$.
\endcomment

\subsection{\sc $n$-globular sets and strict $n$-categories}
There is another category equivalent to  the category $\Theta_n$.

Recall that an {\it $n$-globular set} is the data one has on the undelying sets of objects, 1-morphisms, ..., $n$-morphisms of a strict $n$-category. In this sense, it is a ``pre-$n$-category''. For $n=1$, it is a quiver. 

The general definition is as follows. 
\begin{defn}{\rm
{\it An $n$-globular set} is a collection of sets $X_0, X_1,\dots, X_n$ and maps 
\begin{equation}
X_n \doublerightarrow{s_{n-1}}{t_{n-1}}X_{n-1}\doublerightarrow{s_{n-2}}{t_{n-2}}\dots X_1\doublerightarrow{s_0}{t_0}X_0
\end{equation}
(here $s_k$ are {\it source maps} and $t_k$ are {\it target maps}), such that 
$s_ks_{k+1}=s_kt_{k+1}$, $t_ks_{k+1}=t_kt_{k+1}$, $0\le k\le n-1$.

For two $n$-globular sets $X,Y$, a morphism $f\colon X\to Y$ is defined as a sequence of maps $f_i\colon X_i\to Y_i$, $0\le i\le n$, which commute with the source and the target maps $s$ and $t$. 

The category of $n$-globular sets is denoted by $\mathrm{Glob}_n$. The reader easily interprets the category $\mathrm{Glob}_n$ as some presheaf category. 
}
\end{defn}

The following question arises: {\it how one can define the free strict $n$-category generated by an $n$-globular set?} 
More precisely, the question is in defining the left adjoint functor $\omega_n$ to the forgetful functor $R\colon \mathrm{Cat}_n\to \mathrm{Glob}_n$. (The $n=1$ case is the well-known construction of the free category generated by a quiver). 

The construction of M.Batanin [Ba1], Sect. 4, which associates an $n$-globular set $T^*$ to an $n$-ordinal $T$, is served to solve this problem. 

We recall this construction, following a more explicit treatment given in [B1, Lemma 1.2]: 
\begin{lemma}
Let $T$ be an $n$-leveled tree, denote by $T^*_k$ the set of all sectors of $T$ of height $k$, $0\le k\le n$. Then $T^*$ is an $n$-globular set.
\end{lemma}
\begin{proof}
Let $(x; y_L,y_R)\in T^*_k$, we have to define $s_{k-1}(x;y_L,y_R)$ and $t_{k-1}(x;y_L,y_R)$. Let $x_L,x,x_R$ be the three consequtive elements in $\bar{T}(k)$. Define $$s_{k-1}(x; y_L,y_R)=(i_T(x); x_L,x) \text{ and }t_{k-1}(x;y_L,y_R)=(i_T(x);x,x_R)$$
One easily sees that the globular identities hold, see [B1, Lemma 1.2] for more detail.
\end{proof}

\begin{example}{\rm
Some examples for 2-level trees are shown in Figure \ref{fig1} below.
\sevafigc{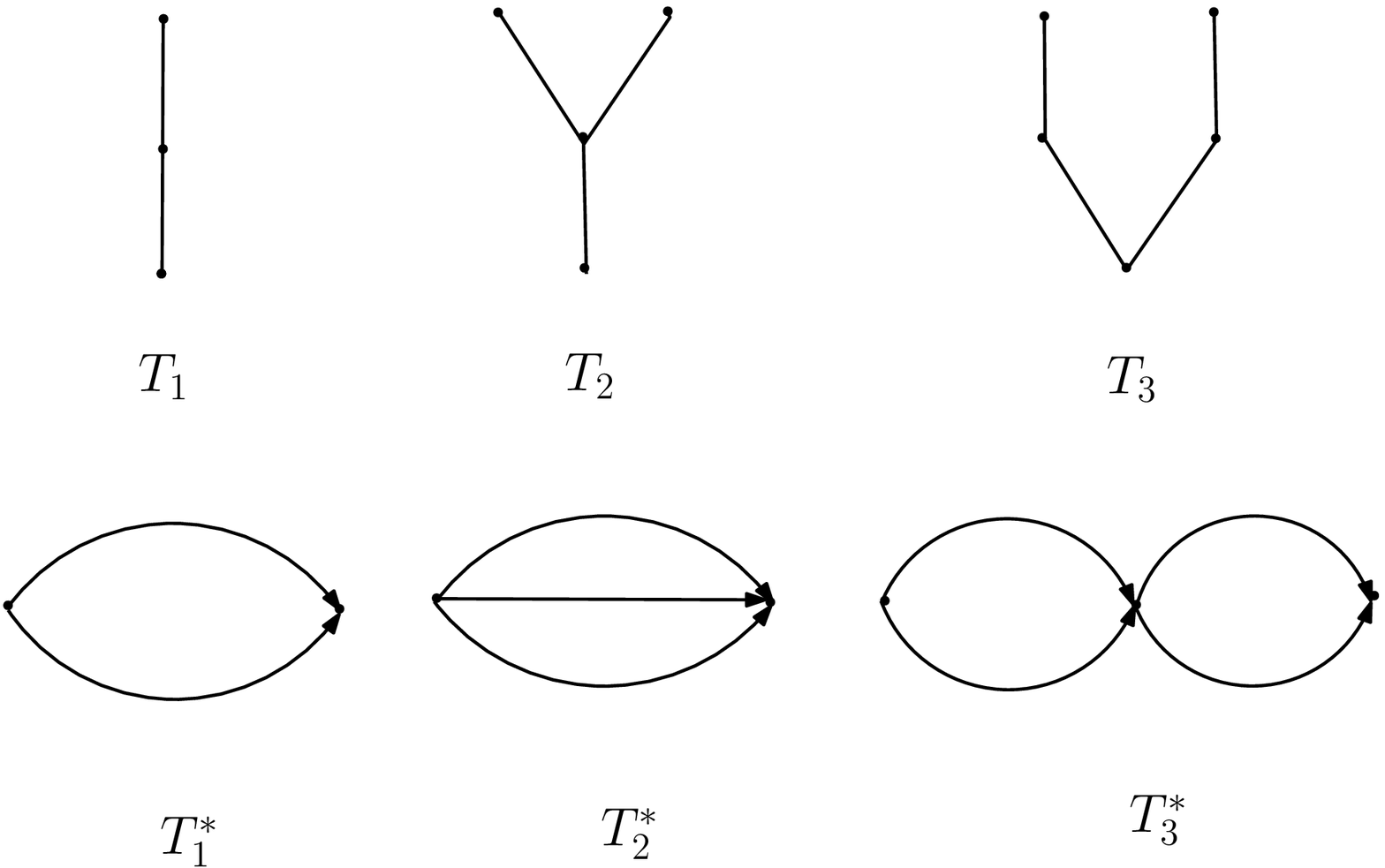}{80mm}{0}{\label{fig1}}
}
\end{example}

\begin{example}{\rm
For the 2-level tree $T$ having $n$ vertices at level 1 with the preimages having $\ell_1,\dots,\ell_n$ vertices at level 2, the 2-globular set $T^*$ is schematically shown in Figure \ref{example4tree}.
}
\end{example}

With the help of the $T^*$ construction, one can define the left adjoint $\omega_n\colon\mathrm{Glob}_n\to\mathrm{Cat}_n$ to the forgetful functor $R\colon\mathrm{Cat}_n\to\mathrm{Glob}_n$, as follows. 

Let $X$ be an $n$-globular set. We define an $n$-globular set $\omega_n(X)$ and prove that it is a strict $n$-category. 

Set 
\begin{equation}\label{ncatfree}
(\omega_n(X))_k=\coprod_{T:\height(T)\le k}\Hom_{\mathrm{Glob}_n}(T^*,X)
\end{equation}
(one often uses the notation $\Hom_{\mathrm{Glob}_n}(T^*,X)=X^T$).

First of all, we show that $\omega_n(X)$ is an $n$-globular set.

Denote by $\partial_k T$ the $(k-1)$-leveled tree, obtaing by removing all vertices of height higher than $k-1$. There are two maps of $n$-globular sets $s^*_{k-1},t^*_{k-1}\colon (\partial_kT)^*\to (\partial_{k+1}T)^*$. In general, a map of globular sets $S^*\to T^*$ is determined by its restriction to the input sectors of $S^*$, see [B1, Lemma 1.3]. The map $s^*_{k-1}$ (corresp., $t^*_{k-1}$) is obtained by assigning to each input vertex $x$ of $\partial_kT$ (which uniquely defined its input sector) the leftmost (resp., rightmost) input sector in $\partial_{k+1}T$ supported by $x$. One shows that the maps $s_{k-1}^*,t_{k-1}^*$ satisfy the identities dual to the globular identities. 
Thus, for any $n$-globular set $X$, and for a $k$-leveled tree $T$, $k\le n$, the precompositions with the maps $s^*_{k-1},t^*_{k-1}$ define maps 
$$
s_{k-1},t_{k-1}\colon X^T\to X^{\partial_k T}
$$
It follows that these maps satisfy the globular identities. Thus, $\omega_n(X)$ is a globular set. 

Next, prove that $\omega_n(X)$ is a strict $n$-category. 
\\
The following statement is proven in [B2, Th. 3.7].
\begin{prop}\label{propnice}
For any $n$-ordinals $S,T$, one has $\Theta_n(S,T)=\mathrm{Cat}_n(\omega_n(\overline{S}^*), \omega_n(\overline{T}^*))$.
\end{prop}
\qed

The proof is obtained, by induction, from the following nice interpretation of the wreath product, [B2, Prop. 3.5]:
\begin{lemma}\label{lemmanice}
Assume that a small category $\mathscr{A}$ is a full subcategory of a cocomplete cartesian monoidal category $\mathscr{V}$.
Then $\Delta\wr \mathscr{A}$ is a full subcategory of $\mathscr{V}{-}\Cat$.
\end{lemma}
\begin{proof}
Any 1-ordinal $[n]$ can be considered as a linear category $\mathbf{n}$ with $n+1$ objects $0,\dots,n$, with a single morphism in $\mathbf{n}(i,j)$ for $i\le j$ and with empty set of morphisms otherwise. Having $n$ objects $A_1,\dots,A_n$ of $\mathscr{A}$, we regard them as objects of $\mathscr{V}$, and consider the linear $\mathscr{V}$-quiver:
$$
0\xrightarrow{A_1}1\xrightarrow{A_2}2\xrightarrow{A_3}\dots\xrightarrow{A_n}n
$$
Consider the $\mathscr{V}$-category generated by this quiver, denote it by $F_\mathscr{V}(A_1,\dots,A_n)$ (here we use cocompleteness of $\mathscr{V}$ to show that the forgetful functor from $\mathscr{V}$-categories to $\mathscr{V}$-quivers has a left adjoint). 

For $B_1,\dots,B_m\in \mathscr{A}$, a $\mathscr{V}$-functor $\Phi\colon F_\mathscr{V}(A_1,\dots,A_n)\to F_\mathscr{V}(B_1,\dots,B_m)$ is defined by its restriction to ``generators'', that is, by a map $\phi\colon [n]\to [m]$, and, for any $1\le i\le n$, a morphism $A_i\to F_\mathscr{V}(\phi(i-1),\phi(i))=B_{\phi(i-1)+1}\times\dots \times B_{\phi(i)}$.
We conclude that these $\mathscr{V}$-functors are the same as the morphisms $([n], A_1,\dots,A_n)\to ([m], B_1,\dots,B_m)$ in $\Delta\wr \mathscr{A}$. 
\end{proof}

\begin{example}{\rm
For the case $\Theta_2=\Delta\wr\Delta$, we set $\mathscr{V}=\Cat$, using the imbeding $\Delta\to \Cat$, $[n]\rightsquigarrow\mathbf{n}$. Thus to the element $([n], [\ell_1],\dots,[\ell_n])$ is associated the 2-category generated by the following 2-globular set:

\sevafigc{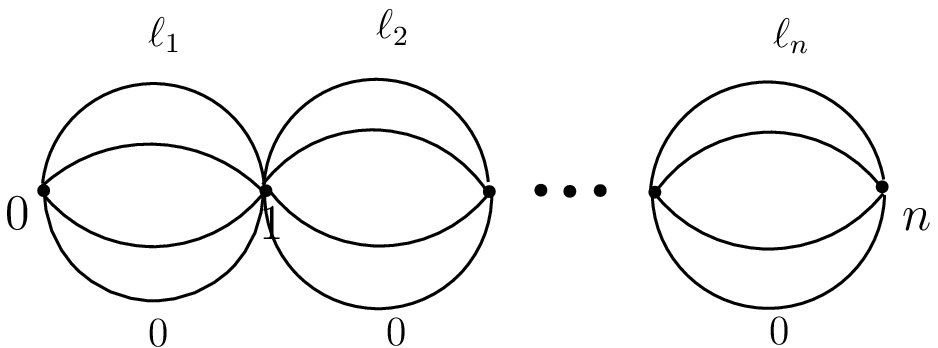}{70mm}{0}{\label{example4tree}}
\comment
\begin{equation}
\begin{tikzpicture}[baseline=(current bounding box.center),scale=1]
    \node (1) at (0,0){$\udot_{0}$}; \node(2) at (2,0){$\udot_{1}$}; \node(3) at (4,0){$\udot_{2}$}; \node (4) at (6.5,0){$\udot_{n-1}$}; \node (5) at (8.5,0) {$\udot_{n}$}; \node (6) at (5.25,0) {$\dots$};
    \node (A) at (1,0.15) {$\vdots$}; \node (B) at (3,0.15) {$\vdots$}; \node (C) at (7.5,0) {$\vdots$};
    \path[->, font=\scriptsize]
    (1)edge[bend left=60, "$\ell_1$"](2)(1)edge[bend left = 30](2)(1)edge[bend right=30,"$0$"'](2)
    (2)edge[bend left=60, "$\ell_2$"](3)(2)edge[bend left = 30](3)(2)edge[bend right=30](3)
    (2)edge[bend right=60,"$0$"'](3)(4)edge[bend left=60, "$\ell_n$"](5)(4)edge[bend left = 30](5)(4)edge[bend right=60,"$0$"'](5);
    \end{tikzpicture}
\end{equation}
\endcomment
Of course, this globular set is $T^*$, where $T$ is the corresponding 2-ordinal $[\ell_1+\dots+\ell_n+n-1]\to [n-1]$.
}
\end{example}

\comment

How to define a map of free $n$-categories $\omega_n(\overline{S}^*)\to \omega_n({T}^*)$? As the left-hand side $n$-category is {\it freely} defined by the corresponding globular set, the intuition says that it is enough to fix the images of the generators, which can be arbitrary (but should respect the relations in $\overline{S}^*$). By {\it generators} we understand the elements of the globular set ${S}^*$, corresponded to the {\it input} sectors. These images may be arbitrary morphisms of the same degree in $\omega_n(\overline{T}^*)$, which are $S$-compatible (that is, should respect the relations coming from the operators $s_k,t_k$ on $\overline{S}^*$). More precisely, there is the following statement, [B1, Lemma1.3, Def. 1.8]:

\begin{lemma}\label{lemmamaps}
Let $S,T$ be $n$-leveled trees.
\begin{itemize}
\item[(a)] Any map of $n$-globular sets $S^*\to T^*$ is injective, and such maps are in 1-to-1 correspondence with connected subtrees $S\subset T$, with a choice of a $T$-sector to each input sector of $S$;
\item[(b)] any map of strict $n$-categories $\omega_n(\overline{S}^*)\to\omega_n(\overline{T}^*)$ is in 1-to-1 correspondence with the following data: for each input sector $\sigma$ of $\overline{S}$ an imbedding $\phi_\sigma\colon \overline{T}^*_\sigma\to \overline{T}$, such that $\{\phi_\sigma\}$ for all input sectors of $\overline{S}$ are {\it $S$-compatible}. 
\end{itemize}
\end{lemma}
\qed

See [B1, Remark 1.9] on different interpretations of $S$-compatibility. 
\endcomment

\subsection{\sc Disks}
The category of disks was introduced in [J]. 
An {\it interval} is a finite ordinal, a {\it map of intervals} is a map in $\Delta$ preserving the leftmost and the rightmost elements. 
The category of the non-empty intervals is denoted by $\Delta_f$. Joyal (loc.cit.) showed that $\Delta_f^\opp\simeq \Delta_+$ where $\Delta_+$ is the category of {\it all} finite ordinals (including the empty ordinal which is the initial object, we denote it [-1]). The functor $F\colon \Delta_+^\opp\to \Delta_f$ is $[n]\mapsto \Delta_+([n],[1])$, $F([n])=[n+1]$. The dual functor $G\colon \Delta_f^\opp\to \Delta_+$ is $[n]\mapsto \Delta_f([n],[1])$, then the initial object $[-1]$ is $\Delta_f([0],[1])$, and in general $G([n])=[n-1]$. 
\begin{defn}{\rm
A {\it disk of finite sets} $D_\ldot$ is a sequence $D_1,D_2,\dots$ of finite sets, equiped with the following data: 
\begin{itemize}
\item[(a)] a map $p\colon D_k\to D_{k-1}$ such that for any $x\in D_{k-1}$ the pre-image $p^{-1}(x)$ has an interval structure, $k\ge 1$
\item[(b)] two maps $d_0,d_1\colon D_{k-1}\to D_k$ sending $x\in D_{k-1}$ to the leftmost and the righmost elements of the interval $p^{-1}(x)$, $k\ge 1$
\item[(c)] 
for $k\ge 1$, the diagram 
$$d_0(D_{k-1})\cup d_1(D_{k-1})\to D_k\underset{d_1}{\overset{d_0}{\rightrightarrows}} D_{k+1}
$$
is an equalizer, where the first arrow is the canonical embedding,
\item[(c)] $D_0$ is a single point.
\end{itemize}

A map of two disks $F\colon D_\ldot\to D^\prime_\ldot$ is a collection of maps $\{F_k\colon D_k\to D_k^\prime\}_{k\ge 0}$ compatible with $p, d_0,d_1$, such that for any $x\in D_k$ the map $p^{-1}(x)\to p^{-1}(F_k(x))$ is a map of intervals, $k\ge 0$.\\
The category of disks is denoted by $\mathrm{Disk}$.\\
For a disk $D_\ldot$ the {\it interior} $i(D_k)$ is defined as $D_k\setminus\{d_0(D_{k-1}\cup d_1(D_{k-1})\}$. It is an ordinal, and the sequence of maps of ordinals $p\colon i(D_k)\to i(D_{k-1}), k\ge 1$ makes $i(D_\ldot)=\{i(D_k)\}_{k\ge 0}$ a {\it leveled tree}. The height $\height(D_\ldot)$ is defined as the height of the level tree $i(D_\ldot)$. The category of disks of height $\le n$ is denoted by $\mathrm{Disk}_n$. 
}
\end{defn}
The functor $i$ sends disks to leveled trees. The functor $T\mapsto\overline{T}$ is a left adjoint to it. For any leveled tree $T$, the leveled tree $\overline{T}$ is a disk of finite sets. The elements of $\overline{T}$ in the image of $i$ are {\it internal}, and the elements in $\overline{T}\setminus T$ are {\it boundary}.\\ A map of disks $\overline{S}\to\overline{T}$ is ``more general'' than a map of leveled trees $S\to T$. The reason is that a map of disks $\overline{S}\to \overline{T}$ may map an internal point to a boundary point in $\overline{T}$. Thus, the category $\mathrm{Ord}_n$ is identified with a not full subcategory of $\mathrm{Disk}_n$. 

The following Proposition is [B1, Prop. 2.2]:
\begin{prop}\label{berdisks}
For any $n$-leveled trees $S,T$ one has
$$\Cat_n(\omega_n(\overline{S}^*),\omega_n(\overline{T}^*))=\mathrm{Disk}_n(\overline{T},\overline{S})$$
Thus, the assignment $T\mapsto \overline{T}$ provides an equivalence of $\Theta_n^\opp$ and $\mathrm{Disk}_n$.
\end{prop}

\qed
\comment
\begin{proof}
We outline the main idea, refering the reader to loc.cit. for more detail.\\
According to Lemma \ref{lemmamaps}, to define a map of $n$-categories $\omega_n(\overline{S}^*)\to\omega_n(\overline{T}^*)$, we have to associate a map (an imbedding) $\phi_\sigma\colon T^*_\sigma\to \overline{T}^*$, for any input sector $\sigma\in \overline{S}^*$, in an $S$-compatible way.

Assume a map in $f\in \mathrm{Disk}_n(\overline{T},\overline{S})$ is given. Then each element of $T=i(\overline{T})$ is mapped to an element of $\overline{S}$. Let $\sigma$ be an input sector $(x; y_L,y_R)$ of $\overline{S}$. We can talk on subtrees $S_L, S_R, S_\sigma$ in $\overline{S}$ which are formed by all descendants (=iterative images of the map $i_{\overline{S}}$) of the points $y_L,y_R, x$ correspondingly. Clearly they are {\it linear} subtrees. Now take $T_L:=f^{-1}(S_L)$, $T_R:=f^{-1}(S_R)$, $T_\sigma=f^{-1}(S_\sigma)$. They are all {\it connected} level subtrees of $\overline{T}$ (but possibly having the height lower then the height of $S_L,S_R,S_\sigma$, correspondingly. \\
Any input vertex of $T_\sigma$ supports a unique input sector in $\overline{T}_\sigma^*$, and all input sectors come in this way from input vertices. Take an input vertex $v$ of $T_\sigma$, consider the corresponding input sector $\kappa(v)$ of $\overline{T}^*_\sigma$.
According to Lemma \ref{lemmamaps}, we have to associate to $\kappa(v)$ an input sector of $\overline{T}^*$ having the same height. To this end, note that, whenever $x\in S$ is in the image of $f$, $v\in \overline{T}$  {\it supports a unique sector which is placed in between of $T_L$ and $T_R$}. For the case when $x\notin \mathrm{Im}(f)$, and $f(v)=i_S^k(x)$, $k\ge 1$, we send $\kappa(v)$ to the unique $\overline{T}$-sector {\it whose image in $\overline{S}$ contains $i_S^{k-1}(x)$}.\\
 One checks that the conditions of Lemma \ref{lemmamaps}(b) are fulfilled, and that such correspondence between $\mathrm{Disk}_n(\overline{T},\overline{S})$ and $\Cat_n(\omega_n(\overline{S}^*),\omega_n(\overline{T}^*))$ is 1-to-1.

\end{proof}
\endcomment

\begin{remark}\label{tam}{\rm
We can restrict the assignment from the proof [B1, Prop. 2.2] to the maps of disks $\overline{S}\to\overline{T}$ which come from maps of leveled trees $S\to T$ (that is, which map internal points to internal). 
The corresponding sub-category $C$ of $\mathrm{Cat}_n$ has objects $\omega_n(\overline{T}^*)$, $T\in \Ord_n$, and has the set of morphisms 
$C(\omega_n(\overline{S}^*),\omega_n(\overline{T}^*))$ which is the subset of $\Cat_n(\omega_n(\overline{S}^*),\omega_n(\overline{T}^*))$ formed by maps of $n$-categories, preserving minima and maxima, in an appropriate sense. For $n=2$, this equivalence is used by Tamarkin [Tam1].

In fact, this equivalence (rather than the equivalence of Proposition \ref{berdisks}) can be thought of as a proper analogue of the Joyal equivalence $\Delta_f\simeq\Delta_+^\opp$, for $n\ge 2$.
}
\end{remark}

\subsection{\sc The categories $\Theta_n$ as higher analogs of the category $\Delta$; inner and outer face maps}
We have three equivalent descriptions of the category $\Theta_n$ which are
\begin{itemize}
\item[(a)] the definition via the wreath product \eqref{defthetawp},
\item[(b)] the definition via morphisms of free strict $n$ categories $\omega_n(\overline{T}^*)$, Proposition \ref{propnice},
\item[(c)] as the dual of the category $\mathrm{Disk}_n$, Proposition \ref{berdisks}.
\end{itemize}
We will take advantage of all three equivalences. In particular, (c) is used to naturally define the realization/totalization, (b) is used to see that any strict $n$ category $C$ has a nerve which is a cocellular set $N(C)\colon \Theta_n^\opp\to \Sets$, and (a) is the most combinatorially explicit and manageable. 

Existence of the nerve was the main motivation in [Joy], where the disk categories were defined. 
It also makes it possible to consider $\Theta_n$ as an analogue of $\Delta$, for $n\ge 2$.

Note that the nerve $N(C)$ of the ordinary category $C$ is a simplicial set, whose components can be defined as 
$$
N(C)_k=\mathrm{Cat}(\mathbf{k},C)
$$
(where $k$ is the linear category with $k+1$ objects). We see directly that it gives rise to a simplicial set, because a map $[k]\to [m]$ in $\Delta$ amounts to the same thing as a map of the linear categories $\mathbf{k}\to\mathbf{m}$.

Let now $C$ be a strict $n$-category. Define its {\it $n$-nerve} as a cellular set $N(C)\colon \Theta_n^\opp\to\Sets$, for which
\begin{equation}\label{nerve2cat}
N(C)_T=\mathrm{Cat}_n(\omega_n(\overline{T}^*),C)
\end{equation}
It gives rise to an $n$-cellular set because, by Proposition \ref{propnice}, 
$$
\Theta_n(S,T)=\mathrm{Cat}_n(\omega_n(\overline{S}^*),\omega_n(\overline{T}^*))
$$

For any strict $n$-category $C$, the $n$-cellular set $N(T)$ has a property which is a higher $n$ counterpart of the Boardman-Vogt inner horns filling property for $n=1$, called the ``weak Kan complexes''. The simplicial sets with inner horns filling condition 
were further studied by Joyal (under the name quasi-categories) and Lurie  (under the name $(\infty,1)$-categories, as a model for weak analogues of ordinary categories. The aim in [J] was to define a model for weak analogues of strict $n$-categories, and it was the motivation for introducing the categories $\Theta_n$. 

In $\Theta_n$, there are two classes of maps, {\it face maps} and {\it degeneracy maps}, and face maps are further subdivided to {\it outer face maps} and {\it inner face maps}. 

The most direct way to define them is by using the category of disks $\mathrm{Disk}_n$, see Proposition \ref{berdisks}. 

Let $S,T$ be two $n$-leveled trees. \\
A map $\overline{S}\to \overline{T}$ in $\mathrm{Disk}_n$ is called a {\it degeneracy} if it is an embedding on each interval $p^{-1}(x)$, and $|S|+1=|T|$.\\
A map $f\colon\overline{S}\to \overline{T}$ is called an  {\it inner face map} 
if $|S|=|T|+1$, $f$ contracts two neighbour  {\it inner} points $a,b$ at some interval $p^{-1}(x)$, and is a {\it shuffling} map on the interiors of the intervals $p^{-1}(a)$ and $p^{-1}(b)$ (when such shuffling is fixed, the map of disks with these conditions is uniquely defined).\\
Let $x\in S$, $a\in p^{-1}(x)$, $a\in i(S)$ is an extreme (=leftmost or rightmost) vertex of $p^{-1}(x)$ in $i(S)$. We call $a$ {\it a special extreme vertex} if $a$ is also an input vertex of $i(S)$.
Let $f\colon\overline{S}\to\overline{T}$ be a map in $\mathrm{Disk}_n$, $|S|=|T|+1$,  $a$ is a special extreme vertex in $S$, $b$ a vertex in $\partial p^{-1}(x)\subset\overline{S}$ (left or right) neighbour to $a$ (the vertex $b$ is unique except for the case when $p^{-1}(x)\subset \overline{S}$ consists of 3 vertices, two of which are boundary). We call $f$ the {\it outer face map} associated with $(a,b)$ as above if $T=S\setminus\{a\}$ and $f$ maps $a$ to $b$. \\
It is clear that any surjective map of codimension 1 is either inner or outer face map. 

\begin{remark}{\rm
The inner (resp. all) face maps are used to define {\it inner horns} (resp., {\it all horns}) in [Joy, Def.2], and to define {\it weak $n$-categories} (resp., {\it weak $n$-groupoinds} as cellular sets $X\colon \Theta_n\to\Sets$ with inner (resp., all) horns filling property.
This idea was further elaborated in [B1,2] and [Ar]. 
}
\end{remark}

\subsection{\sc The Reedy structure on $\Theta_2$, description of elementary coface and codegeneracy maps}\label{facetheta}
One can translate the above definition of coface and codegeneracy maps to the wreath product definition of $\Theta_2$. For $D=([k]; [n_1],\dots,[n_k])\in \Theta_2$, define {\it dimension} of $D$ as 
\begin{equation}\label{eqdim}
|D|=k+n_1+\dots+n_k
\end{equation}
It was proven in [B1, Lemma 2.4(a)] that $\Theta_2$ is a Reedy category, in which {\it degree} is equal to the dimension \eqref{eqdim}, and there are two classes of morphisms, coface maps and codegeneracy maps, which raise (corresp., lower) the degree. The construction was clarified in [BR], using the wreath product definition. (In [B1], [BR] a Reedy structure on $\Theta_n$ for general $n\ge 1$ is defined).

Recall following [BR] the Reedy category structure on $\Theta_2$.

Recall that an object of $\Theta_2$ is given by a tuple $([k]; [n_1],\dots,[n_k])$, a morphism
$\Phi\colon ([n];[\ell_1],\dots,[\ell_n])$ $\to ([m]; [k_1],\dots,[k_m])$ is $(\phi;\phi_1,\dots,\phi_n)$, where $\phi\colon [n]\to [m]$ is a morphism in $\Delta$, and $\phi_i=(\phi_i^{\phi(i-1)+1},\dots, \phi_i^{\phi(i)})$, $\phi_i^s\colon [\ell_i]\to [k_s]$ is a tuple of morphisms in $\Delta$. 

Define two subcategories $\Theta_2^-,\Theta_2^+\subset\Theta_2$, such that $\Ob\Theta_2^-=\Ob\Theta_2^+=\Ob\Theta_2$. 
We say that $\Phi\in \Theta_2^-$ if $\phi$ is surjective, and for $\phi(i-1)<\phi(i)$ the map $\phi_i^{\phi(i)}\colon [k_i]\to [\ell_{\phi(i)}]$ is surjective. We say that a map $\Phi\in\Theta_2^+$ if $\phi$ is injective, and for any $i$ the family of maps 
$\{\phi_i^j\colon [k_i]\to [\ell_j]\}_{\phi(i-1)+1\le j\le \phi(i)}$ is {\it jointly} injective, that is, for any $a,a+1\in [k_i]$ there is $j$ such that $\phi_i^j(a)\ne \phi_i^j(a+1)$ (note that individual $\phi_i^j$ may not be injective for all $j$). 

The following statement is a particular case of [BR], Prop. 2.11:

Any $\Phi\in\Theta_2$ can be uniquely decomposed as $\Phi=\alpha^+\circ\alpha^-$ with $\alpha_+\in \Theta_2^+$, $\alpha_-\in\Theta_2^-$. One has $\Phi\in\Theta_2^+\cap\Theta_2^-$ if $\Phi=\id$, the morphisms in $\Theta_2^-$ decrease $|D|$ and the morphisms in $\Theta_2^+$ raise $|D|$.

Below we list the $\codim=\pm 1$ (with respect to $|-|$) coface and codegeneracy maps in the wreath product model of $\Theta_2$.

We denote by $\partial^j$ the $j$-th coface maps $\partial^j\colon [n]\to [n+1]$ in $\Delta$, $0\le j \le n+1$.

\vspace{2mm}

{\it Inner coface maps of codimension 1:}

\vspace{2mm}

\begin{itemize}
\item[(F1)] $n=m$, $\ell_i=k_i$ for $i\ne p$, $k_p=\ell_p+1$, all $\phi_i^s=\id$ except for $\phi_p^{\phi(p)}$ equal to the $j$-th coface map $\partial^j\colon [\ell_p]\to[\ell_p+1]$, $j\ne 0, \ell_p+1$ (that is, $\partial^j$ is an inner coface map in $\Delta$). We denote this coface map $\partial_p^j$,
\item[(F2)] $m=n+1$, the morphism $\phi\colon [m]=[n]\to [n+1]$ is $\partial^j$, $j\ne 0, n+1$.
Next, $k_s=\ell_s$ except for $s=j,j+1$, $k_j+k_{j+1}=\ell_j$, and all $\phi_s=\id$ except for $s=j$. 
Let $\sigma$ be a $(k_j,k_{j+1})$-shuffle permutation in $\Sigma_{\ell_j}$. The permutation $\sigma$ defines two maps $p\colon [k_j-1]\to [\ell_{j}-1]$ and $q\colon [k_{j+1}-1]\to [\ell_j-1]$ in $\Delta$. They define the Joyal dual maps $p^*\colon [\ell_j]\to [k_j]$ and $q^*\colon [\ell_j]\to [k_{j+1}]$ in $\Delta$ preserving the end-points. Then $(p^*,q^*)\colon [\ell_j]\to [k_j]\times [k_{j+1}]$, extended by the identity maps of the ordinals $[\ell_i], i\ne j$, defines a map in $\Theta_2$. It is the $\codim=1$ coface map associated with a shuffle permutation $\sigma$.

We denote this coface map by $D_{j,\sigma}$.

Let us define $p^*,q^*$ explicitly, unwinding the definition. 
We think of the sets $\{1,\dots,k_j\}$, $\{1,\dots, k_{j+1}\}$, $\{1,\dots,\ell_j\}$ as the elementary {\it arrows} in the interval categories $I_{k_j}, I_{k_{j+1}}$, and $I_{\ell_j}$, correspondingly.
Then $p^*$ and $q^*$ are defined as follows. Both $p^*$ and $q^*$ preserve the end-points. 
If $\sigma^{-1}(\overrightarrow{i,i+1})=\overrightarrow{a,a+1}\in I_{k_j}$, then $p^*(i)=a,p^*(i+1)=a+1, q^*(i)=q^*(i+1)$. If $\sigma^{-1}(\overrightarrow{i,i+1})=\overrightarrow{b,b+1}\in I_{k_{j+1}}$, then $q^*(i)=b,q^*(i+1)=b+1,p^*(i)=p^*(i+1)$.

\end{itemize}

\vspace{2mm}

{\it Outer coface maps of codimension 1:}

\vspace{2mm}

\begin{itemize}
\item[(F3)] $n=m$, $\ell_i=k_i$ for $i\ne p$, $k_p=\ell_p+1$, all $\phi_i^s=\id$ except for $\phi_p^{\phi(p)}$ equal to the $j$-th coface map $\partial^j\colon [\ell_p]\to[\ell_p+1]$, $j= 0, \ell_p+1$ (that is, $\partial^j$ is an outer coface map in $\Delta$). We denote this coface map $\partial_p^j$
\item[(F4)] the two remaining codim 1 coface maps are $D_\min$ and $D_\max$. In both cases, $m=n+1$. For the case of $D_\min$, $\phi=\partial^0$, and $k_1=0$, $k_s=\ell_{s-1}$ for $s\ge 1$, the maps $\phi_i=(\phi_i^{i+1})=(\id)$. For the case of $D_\max$, $k_{m+1}=0$, $\phi=\partial^{n+1}$, $\phi_i=(\phi_i^i)=(\id)$. 
\end{itemize}

More generally, we call a map $\Phi\colon ([n];[\ell_1],\dots,[\ell_n])\to ([m]; [k_1],\dots,[k_m])$, $\Phi=(\phi;\ \phi_1,\dots,\phi_n)$, a {\it coface map} if $\phi\colon [n]\to [m]$ is injective, and each $\phi_i\colon [\ell_i]\to [k_{\phi(i-1)+1}]\times\dots\times [k_{\phi(i)}]$ is a (jointly) injective map (the latter means that for any $a,b\in [\ell_i], a\ne b$, for at least one $\phi_i^s, \phi(i-1)+1\le s\le \phi(i)$, one has $\phi_i^s(a)\ne \phi_i^s(b)$).\\
\\

Here is the list of elementary codegeneracy maps in $\Theta_2$:
\begin{itemize}
\item[(D1)] $n=m$, $\ell_i=k_i$ for $i\ne p$, $k_p=\ell_p-1$, all $\phi_i^s=\id$ except for $\phi_p^{\phi(p)}$ equal to the $j$-th codegeneracy map $\varepsilon^j\colon [\ell_p]\to[\ell_p-1]$.\\ We denote this codegeneracy map by $\varepsilon_p^j$,
\item[(D2)] $n-1=m$, the first component $p(\Phi)$ is $\varepsilon^p\colon [n]\to [n-1]$. For any $[\ell_{p+1}]$, it extends uniquely to a morphism $$\Phi\colon ([n]; [\ell_1],\dots, [\ell_p],[\ell_{p+1}],[\ell_{p+2}],\dots,[\ell_n])\to ([n-1]; [\ell_1],\dots,[\ell_p], [\ell_{p+2}],\dots,[\ell_n])$$
for which $\phi_1,\dots,\phi_p,\phi_{p+2},\dots,\phi_{n}$ are identity maps. \\
We denote this operator $\Upsilon^p_{\ell_{p+1}}$. Note that the morphism $\Upsilon^p_{\ell_{p+1}}$ is of codimension 1 iff $\ell_{p+1}=0$. We define $\Upsilon^p_\ell=0$ if $\ell\ne\ell_{p+1}$.

\end{itemize}

One can show that the morphisms in $\Theta_2$ listed above form a set of generators for $\Theta_2$. 
For relations between these generators, see Appendix \ref{relationstheta}.

\section{\sc The totalization of a 2-cocellular vector space}
Here we define the non-normalized $\Theta_2$-totalization of a 2-cocellular complex. Also we define a relative $p$-totalization $Rp_*$, for the functor $p\colon \Theta_2\to \Delta$, and prove the transitivity property saying that $\Tot_\Delta\circ  Rp_*=\Tot_{\Theta_2}$.

\subsection{\sc Generalities on realizations and totalizations}\label{remtot}
Recall that for a general category $\Xi$ and a functor $C\colon \Xi\to C^\udot(\k)$ the corresponding {\it realization} in $C^\udot(\k)$ is a functor $\Sets^{\Xi^\opp}\to C^\udot(\k)$ defined as the left Kan extension of $C$ along the Yoneda embedding $\Xi\to \Sets^{\Xi^\opp}$. 
That is, the realization depends on a functor $C\colon \Xi\to C^\udot(\k)$. For $X\in \Sets^{\Xi^\opp}$, we denote by $|X|_C$ or just $|C|$ its realization with respect to the functor $C$.
Dually, for the same $\Xi$ and $C$, define the {\it totalization} functor $\Sets^\Xi\to C^\udot(\k)$ as the right Kan extension of $C^\opp\colon \Xi^\opp\to C^\udot(\k)^\opp$ by the dual Yoneda functor $\Xi^\op\to (\Sets^\Xi)^\op$. The result is a functor 
$\Sets^\Xi\to C^\udot(\k)$. For $Y\in \Sets^\Xi$, we denote by $\Tot_C(Y)$ ot $\Tot(Y)$ its totalization. 
\begin{remark}{\rm
One similarly defines the realization (resp., the totalization) with values in any cocomplete (resp., complete) category $E$ (replacing the category $C^\udot(\k)$ in the above definition) out of a functor $C\colon \Xi\to E$. For example, for $\Xi=\Delta$ one can take $E=\Top$, $\Xi([n])=\Delta^n$, or $E=\Cat, \Xi([n])=I_n$ (where $I_n$ is the linearly ordered poset with $n+1$ objects), etc.
In the case $Xi=\Delta, E=C^\udot(\k)$, one can takes $C([n])=N(\k\Delta(=,[n])$, or $C([n])=C_{\mathrm{Moore}}(\k\Delta(=,[n]))$ for a realization/totalization in $E=C^\udot(\k)$ (here $C_{\mathrm{Moore}}$ and $N$ denote the Moore chain complex and the normalized Moore chain complex of a simplicial vector space).
}
\end{remark}

It follows from the definition that 
the realization commutes with all small colimits and the totalization commutes with all small limits.
\comment
for $X\in \Sets^{\Xi^\opp}$ and for $Y\in \Sets^{\Xi^\opp}$, one has
\begin{equation}\label{realcolim}
|\colim_{i\in I} X_i|=\colim_{i\in I} |X_i|
\end{equation}
\begin{equation}\label{totlim}
\Tot(\lim_{i\in I} Y_i)=\lim_{i\in I}\Tot(X_i)
\end{equation}
\endcomment

Using this observation, one proves that the realization of a simplicial set $X$ in $C^\udot(\k)$, for $C([n])$ equal to the non-normalized Moore complex of $\k\Delta(=,[n])$, is equal to the non-normalized Moore complex of $X$. 
This fact is fairly standard, but as we employ a similar argument for $\Xi=\Theta_2$ later in the paper, we recall the argument for completeness. 

Represent $X=\colim_{h_{[j]}\to X}h_{[j]}$, where $h_{[j]}([i])=\Delta([i],[j])$.
Then
$$
|X|=|\colim_{h_{[j]}\to X}h_{[j]}|=\colim_{h_{[j]}\to X}|h_{[j]}|\overset{*}{=}\colim_{h_{[j]}\to X}C([j])=\colim_{h_{[j]}\to X}C^\udot_{\mathrm{Moore}}(h_{[j]})
$$
where in the equality marked by $*$ we use that the Yoneda functor is fully faithful and that the unit of the left Kan extension adjunction along a fully faithful functor is the identity [R, Cor. 1.4.5].
The rightmost complex has its $i$-th term equal to $\colim_{h_{[j]}\to X}\k\Delta([i],[j])$. 
Then
$$
(\colim_{h_{[j]}\to X} C^\udot([j]))_i=\colim_{h_{[j]}\to X}(\k h_{[j]}([i]))=(\colim_{h_{[j]}\to X}\k h_{[j]})[i]=X([i])
$$
Similarly we show, for $C([n])=N(\k\Delta(=,[n]))$, the totalization counterparts of these statements.  

In this paper, we consider the realization and the totalization for $C([n])=C^\udot_{\mathrm{Moore}}(\k\Delta(=,[n]))$, for the case of $\Delta$, and $C(T)=C^\udot(\k\Theta_2(=,T))$, where $C^\udot(=)$ is the complex defined in \eqref{tott}, \eqref{totd}. 

We also define ``relative totalization'' with respect to the projection $p\colon \Theta_2\to\Delta$. For the case of realization, it can be defined as the following left Kan extension of $\Phi$ along the Yoneda embedding
$$
\begin{tikzpicture}[baseline=(current bounding box.center), scale=1]
\node (1) at (0,0) {$\Theta_2$};
\node (2) at (2,0) {$\Sets^{\Theta_2^\op}$};
\node (3) at (0,-2) {$C^\udot(\k)^{\Delta^\op}$};
\node (A) at (0.8,-0.5) {$\Downarrow$};
\path[->,font=\scriptsize]
(1)edge(2)
(1)edge["$\Phi$"](3)
(2)edge["$\Lan$"](3);
\end{tikzpicture}
$$
where the functor $\Phi$ sends an object $T\in \Theta_2$ to the functor $\Delta^\opp\to C^\udot(\k)$, $[n]\mapsto \mathfrak{R}_{[n]}(T)$, see \eqref{relphi}. 

Note that the functor $\Delta\to C^\udot(\k)$, $[\ell]\mapsto C^\udot_{\mathrm{Moore}}(\k\Delta(=,[\ell]))$ is a resolution of the constant functor $[\ell]\mapsto \k$. The functor $T\mapsto C(T)$ is a projective resolution by Yonedas of the constant functor $T\mapsto \k$ (as is proven in Section \ref{sectionthetares}),  and the functor $\Phi$ above is a resolution of the functor $T\mapsto\{[n]\mapsto \k\Delta([n],p(T))$ (as it is proven in Lemma \ref{lemmarelphi} below).

Morally, we wanted to talk about the homotopy right Kan extension along $p$, instead of the $p$-relative totalization. We want to use the transitivity of the right homotopy Kan extensions  for the composition $\Theta_2\xrightarrow{p} \Delta\to *$, which is clear from the derived functor of composition theorem, to prove Proposition \ref{prop!}. However, the totalization agrees with the homotopy right Kan extension for the projection to $*$ only for {\it Reedy fibrant cosimplicial (resp., 2-cocellular) complexes of vector spaces} (see [Hir], Section 18.7). Typically, the cosimplicial space whose Moore complex is the Hochschild complex is {\it not} Reedy fibrant (despite the fact that the corresponding bar-complex is Reedy cofibrant). To overcome this difficulty, we talk about the ``relative totalization'' along $p\colon \Theta_2\to \Delta$, choosing the functor $\Phi$ exactly as a resolution of $\k\Delta([n],p(T))$ by the (projective) Yoneda functors. Then we prove the transitivity comparing the explicit formulas for the totalizations.

\subsection{\sc The (absolute) totalization of a $2$-cocellular vector space}\label{sectionthetares}
Let $X\colon \Theta_2\to C(\k)$ be a cocellular complex. First of all, we define its non-normalized Moore complex explictly, as the complex whose degree $\ell$ component is:
\begin{equation}\label{tott}
\Tot_{\Theta_2}(X)^\ell=\bigoplus_{{\substack{{T\in\Theta_2}\\{\dim T=\ell}}}}X_T
\end{equation}
and the differential of degree +1 is equal to the sum of (taken with appropriate signs) all codimension 1 face maps:
\begin{equation}\label{totd}
\begin{aligned}
d|_{X_T}=&\sum_{\substack{{\text{coface maps $\partial_p^j$}}\\{ (F1), (F3)}}}(-1)^{k_1+\dots+k_{p-1}+p-1+i-1} \partial_p^i+\sum_{\substack{{\text{coface maps}}\\{\text{$D_{p,\sigma}$ (F2)}}}}(-1)^{k_1+\dots+k_{p-1}+p-1+\sharp(\sigma)} D_{p,\sigma}+\\
& D_\min+(-1)^{k_1+\dots+k_n+n} D_\max
\end{aligned}
\end{equation}
where $T=([n]; [k_1],\dots,[k_n])$, $\dim T=k_1+\dots+k_n+n$, and for the notations for the face maps see Section \ref{facetheta}.

\begin{lemma}\label{lemmatotd}
One has $d^2=0$.
\end{lemma}
\begin{proof}
It follows from relations \eqref{reltheta1}-\eqref{reltheta7} that the summands in $d^2$ come in pairs, in which the two operators are equal one to another. One checks by hand that for each pair the two terms have opposite signs, which makes them mutually cancelled. 
\end{proof}

This formula can be ``explained'' in the following way. We compute $\underline{\Hom}_{\Theta_2}(K^\udot, X)$ where $K^\udot$ is a (projective) resolution of the constant 2-cocellular vector space $\k$, formed by the (linearised) Yoneda modules $\k\Theta_2(-,T)$, for $T\in \Theta_2$ (here $\underline{\Hom}_{\Theta_2}$ stands for $\Hom$ taking values in $C^\udot(\k)$, it can be equally defined as enriched natural transformations [R, Sect. 7.3]) . In this way, it is ``semi-derived'' functor of $\colim_{\Theta_2}X$, because $\Hom_{\Theta_2}(\k,X)=\colim_{\Theta_2} X$, but it is not properly derived because $X$ is not, in general, Reedy fibrant. We can not claim that for another choice of resolution of $\k$ we get a quasi-isomorphic complex. 

Here is an explicit resolution. 

Fix $T\in \Theta_2$. Denote by $K^\udot(T)$ a complex whose components are 
\begin{equation}
K^{-\ell}(T)=\bigoplus_{\substack{{T^\prime\in \Theta_2}\\ {|T^\prime|=\ell}}}\k\Theta_2(T^\prime,T)
\end{equation}
The differential $d\colon K^{-\ell}(T)\to K^{-\ell+1}(T)$ is defined likewise the differential in \eqref{totd}, so that $K^\udot(-,T)$ for a given $T$ is the realization of the cellular vector space $?\mapsto \k\Theta_2(?,T)$. We check $d^2=0$ as in Lemma \ref{lemmatotd}. 

Clearly the post-composition makes $K^\udot(T)$ functorial in $T$. One remains to show that, for given $T$, $(K^\udot(T),d)$ is quasi-isomorphic to $\k$.\\
Consider $(K^\udot(T),d)$ as the total complex of a bicomplex, with ``vertical'' component of the differential $$d_1=\sum_{\substack{{\text{coface maps $\partial_p^j$}}\\{ (F1), (F3)}}}(-1)^{k_1+\dots+k_{p-1}+p-1+i-1} \partial_p^i
$$
and the ``horizontal'' component 
$$
d_0=\sum_{\substack{{\text{coface maps}}\\{\text{$D_{p,\sigma}$ (F2)}}}}(-1)^{k_1+\dots+k_{p-1}+p-1+\sharp(\sigma)} D_{p,\sigma}+
 D_\min+(-1)^{k_1+\dots+k_n+n} D_\max
$$
We claim that $d_0d_1+d_1d_0=0$, we check it later in the proof of Proposition \ref{propdeltaaction}.

\subsection{\sc The $p$-relative totalization $R p_*(X)$: an explicit description}
\subsubsection{\sc }
The ordinary enriched right Kan extension of $X$ along $p\colon \Theta_2\to\Delta$ is taken equal to 
\begin{equation}
p_*(X)([n])=\underline{\Hom}_{\Theta_2}\big(\k\Delta([n],p(-)),X(-)\big)
\end{equation}
where, as above, $\underline{\Hom}_{\Theta_2}$ taking values in $C^\udot(\k)$ is the enriched natural transformations. 
We replace $\k\Delta([n], p(T))$ by its resolution $\mathfrak{R}_{[n]}$ by Yoneda modules.
Note that it is not the homotopy right Kan extension along $p$, for the case $X$ is not Reedy fibrant. Therefore, we can not say that for any other choice of resolution we get a quasi-isomorphic complex. The resulting $\Delta$-complex $\Hom_{\Theta_2}(\mathfrak{R}, X)$ is ``closely related'' to our relative realization. 

We have to resolve the functor $T\mapsto \k\Delta([n],p(T))$ (for a given $[n]\in\Delta$) by the Yoneda functors $h_{T^\prime}(T)=\k\Theta_2(T^\prime,T)$. Below we provide an explicit resolution $\mathfrak{R}_{[n]}^\udot$ (which is a complex of vector spaces over $\k$).

The degree $\ell$ component is
\begin{equation}\label{relphi}
\mathfrak{R}_{[n]}^{\ell}(T)=\bigoplus_{\substack{{T^\prime\in \Theta_2,\  p(T^\prime)=[n]}\\{\dim T^\prime=n-\ell}}}\k\Theta_2(T^\prime,T)
\end{equation}
Thus, the complex $\mathfrak{R}_{[n]}^\udot$ has non-zero components in degrees $\le 0$. 

The differential $d\colon \mathfrak{R}_{[n]}^\ell\to\mathfrak{R}_{[n]}^{\ell+1}$ is defined as the alternated sum of the ``vertical'' face operators (acting on the first argument $T^\prime$), that is, of face operators (F1) and (F3) from  the list in Section \ref{facetheta}. More precisely, for $T=([q]; [t_1],\dots,[t_q])$,  $T^\prime=([n]; [k_1],\dots,[k_n]), \Phi=(\phi;\phi_1,\dots,\phi_n)\colon T^\prime\to T$, one has
\begin{equation}\label{eqdred}
d(\Phi)=\sum_{s=1}^n\sum_{i=0}^{k_s}(-1)^{k_1+\dots+k_{s-1}+s-1+i}\Phi_{s,i}
\end{equation}
where $\Phi_{s,i}\colon T_{s,i}^\prime\to T$ is defined as the pre-composition $ \Phi \circ \partial_s^i$ (see Section \ref{facetheta}, (F1) and (F3)), and $T^\prime_{s,i}=([n]; [m_1],\dots,[m_n])$, where $m_j=k_j$ for $j\ne s$, $m_s=k_s-1$, and $\partial_s^i\colon T^\prime_{s,i}\to T^\prime$ is the corresponding ``vertical'' face operator. Note that this pre-composition does not affect the ``horisontal'' map $\phi$. It is clear that $d^2=0$. 

\begin{lemma}\label{lemmarelphi}
The following statements are true:
\begin{itemize}
\item[(1)] degree 0 cohomology of $\mathfrak{R}^\udot_{[n]}$ is isomorpic to $\k\Delta([n],p(T))$,
\item[(2)] the higher cohomology (in the negative degrees $\le -1$) vanish.
\end{itemize}
\end{lemma}
\begin{proof}
(1): the degree 0 component $\mathfrak{R}^0_{[n]}$ is a direct sum $\oplus\k\Phi$, where $\Phi\colon ([n];[0],\dots,[0])\to T$, which is the same as $\Phi=(\phi\colon [n]\to [q]; T_1,\dots, T_n)$ where $T_i\in [t_{\phi(i-1)+1}]\times\dots\times [t_{\phi(i)}]$ an element
(recall that $[q]=p(T)$). Degree 0 cohomology is equal to the quotient-space by the image of $\oplus\k\Phi^\prime$, with $\Phi^\prime\colon ([n]; [0],\dots,[0],[1],[0],\dots,[0])\to T$. IFor a given $\phi\colon [n]\to [q]$, all choices of $T_i$ become equal in the quotient-space $H^0(\mathfrak{R}_{[n]}^\udot)=\mathfrak{R}_{[n]}^0/d(\mathfrak{R}_{[n]}^{-1})$. It shows that $H^0(\mathfrak{R}_{[n]}^\udot)\simeq \k\Delta([n],[q])=\k\Delta([n],p(T))$.\\
(2): we construct a contracting homotopy operator $H$ of degree -1, that is an operator such that $(dH+Hd)|_{\mathfrak{R}_{[n]}^\ell}=c_\ell$, where $c_\ell$ is the multiplication by an integer $c_\ell$, non-zero for $\ell\ne 0$. This $H$ is constructed in a standard way as the alternated sum of the ``vertical'' degeneracy operators. 
\end{proof}
\begin{remark}{\rm
It is clear that the complex $\mathfrak{R}_{[n]}^\udot$ is a direct sum $\oplus_\phi\mathfrak{R}^\udot_{[n],\phi}$ over $\phi\colon [n]\to p(T)$, because the differential does not affect $\phi$. Each complex $\mathfrak{R}_{[n],\phi}^\udot$ is a resolution of $\k$ (where $\k$ denotes the complex-object $\k$ in degree 0). 
}
\end{remark}

One checks that $\mathfrak{R}_{[n]}^\udot$ is a functor $\Theta_2\to C^\udot(\k)$, where the action of $\Theta_2$ is given by the {\it post-composition}. It commutes with the differential as the general post-composition and pre-composition do. 

\subsubsection{\sc The action of $\Delta$}
Our next task is to endow our resolution $\mathfrak{R}_{[n]}^\udot$ with a structure of a functor $\Delta^\opp\to C^\udot(\k)$, when $[n]$ varies. Note, that unlike for the cohomology $\k\Delta([n],p(T))$ of $\mathfrak{R}^\udot_{[n]}$, the ``lifted'' action of $\Delta$ on $\mathfrak{R}^\udot_{[n]}$ does not come automatically.  

We need to define the actions of the elementary face operators $\partial^i$ and the elementary degeneracy operators $\varepsilon^j$ in $\Delta$, which we denote, in this context, by $\Omega^i_\Delta$ and $\Upsilon^j_\Delta$, correspondingly. 

Here are the definitions.\\
Let $\Phi\colon T^\prime\to T$ be an element in $\mathfrak{R}^\ell_{[n]}$, $p(T^\prime)=[n]$. 

\begin{equation}
\Omega^i_\Delta(\Phi)=\sum_\sigma\pm \Phi\circ D_{i,\sigma}\pm \Phi\circ D_\min\pm\Phi\circ D_\max
\end{equation} 

\begin{equation}
\Upsilon^j_\Delta(\Phi)=\pm \Phi\circ \Upsilon^j_{0}
\end{equation}
where $\Upsilon^j_{0}\colon ([n]; [\ell_1],\dots,[\ell_j],[0],[\ell_{j+2}],\dots,[\ell_n])\to T$. Note that we take only $T^\prime$ with $[\ell_{j+1}]=[0]$. \\
(See Section \ref{facetheta} for the notations $D_{i,\sigma}$ and $\Upsilon^j$). 

\begin{prop}\label{propdeltaaction}
\begin{itemize}
\item[(1)]
The operators $\Omega^i_\Delta$ and $\Upsilon^j_\Delta$ define maps of complexes $\Omega^i_\Delta\colon \mathfrak{R}^\udot_{[n]}\to \mathfrak{R}^\udot_{[n-1]}$ and $\Upsilon^j_\Delta\colon \mathfrak{R}^\udot_{[n]}\to \mathfrak{R}^\udot_{[n+1]}$, preserving the cohomological degree.
\item[(2)]
The operators $\Omega_\Delta^i$ and $\Upsilon^j_\Delta$ fulfil the simplicial relations, defining a simplicial object $\mathfrak{R}^\udot_?\colon \Delta^\opp\to C^\udot(\k)$, functorial in $T$. The cohomology $H^\udot(\mathfrak{R}^\udot_{[n]})$ with respect to the differential \eqref{eqdred}, with its simplicial action, is isomorphic to $\Delta([n],p(T))$ with its natural simplicial action.
\end{itemize}
\end{prop}

The proof of Proposition \ref{propdeltaaction} is a straightforward computation performed in Appendix \ref{appendixb} using relations between the generators of $\Theta_2$ listed in Appendix \ref{appendixa}.

\subsubsection{\sc Finally, here is $Rp_*(X)$}
For $X\in \Fun(\Theta_2, C^\udot(\k))$, define $Rp_*(X)[n]\in C^\udot(\k)$
as the following complex:
\begin{equation}
0\to Rp_*(X)[n]^0\xrightarrow{d} Rp_*(X)[n]^1\xrightarrow{d}Rp_*(X)[n]^2\xrightarrow{d}\dots
\end{equation}
where
\begin{equation}
Rp_*(X)[n]^\ell=\bigoplus_{\substack{{T\in \Theta_2,\ p(T)=[n]}\\{\dim T=n+\ell}}}X(T)
\end{equation}
and the differential $d$ is the alternated sum of ``vertical'' face operators (of type (F1) and (F3) in Section \ref{facetheta}):
\begin{equation}
d|_{X(T)}=\sum_{i=1}^{p(T)}\sum_{j=0}^{T_i}(-1)^{T_1+\dots+T_{i-1}+j}\partial_i^{j}
\end{equation}
where we write $T=(p(T);\ [T_1],\dots, [T_{p(T)}])$.

According to Proposition \ref{propdeltaaction}, when $[n]\in \Delta$ varies, it gives rise to a functor $\Delta\to C^\udot(\k)$. 

More precisely, we have similar to the stated in Proposition \ref{propdeltaaction} formulas for the coface maps $\Omega_\Delta^i\colon Rp_*(X)[n]\to Rp_*(X)[n+1]$, $i=0,\dots, n+2$, and the codegeneracy maps $\Upsilon_\Delta^j\colon Rp_*(X)[n]\to Rp_*(X)[n-1]$, $j=0,\dots,n$:
\begin{equation}\label{deltaaction1x}
\Omega_\Delta^i=\sum_\sigma\pm D_{i,\sigma}\pm D_\min\pm D_\max
\end{equation}
and
\begin{equation}\label{deltaaction2x}
\Upsilon_\Delta^j=\pm \Upsilon_0^j
\end{equation}
in the notations of Section \ref{facetheta}. The signs are the same as in Proposition \ref{propdeltaaction}.

\begin{prop}\label{propdeltaactionbis}
The operators \eqref{deltaaction1x} and \eqref{deltaaction2x} commute with the differentials on $Rp_*(X)[n]$ for a given $n$, and satisfy the standard cosimplicial identities.
\end{prop}

For a proof, we have two options. We can either note that the proof of Proposition \ref{propdeltaaction} is literally applied to a proof of Proposition \ref{propdeltaactionbis}, with the same computations, or, alternatively, we can deduce it from Proposition \ref{propdeltaaction} and the definition of $Rp_*(X)[=]$ as the right Kan extension
\begin{equation}
\begin{tikzpicture}[baseline=(current bounding box.center), scale=1]
\node (1) at (0,0) {$\Theta_2^\op$};
\node (2) at (2,0) {$(\Sets^{\Theta_2^\op})^\op$};
\node (3) at (0,-2) {$(C^\udot(\k)^{\Delta^\op})^\op$};
\node (A) at (0.8,-0.5) {$\Uparrow$};
\path[->,font=\scriptsize]
(1)edge(2)
(1)edge["$\Phi$"](3)
(2)edge["$\Ran=Rp_*(-)$"](3);
\end{tikzpicture}
\end{equation}
Here $\Phi$ is defined as above, $T\mapsto \{[n]\mapsto \mathfrak{R}_{[n]}(T)\}$. In this way, the $\Delta$-action on $\mathfrak{R}_{[?]}(T)$ given by Proposition \ref{propdeltaaction} is translated to the action of $\Delta$ on $Rp_*(X)[?]$. 

\qed

\subsubsection{\sc $\Tot_\Delta(Rp_*(X))\simeq \Tot_{\Theta_2}(X)$}
 Here we prove the following
\begin{prop}\label{prop!}
Let $X\colon \Theta_2\to C(\k)$ be a 2-cocellular vector space. Then 
the $\Delta$-totalization $\Tot_\Delta(Rp_*(X))$ of $Rp_*(X)$ is a complex isomorphic to  the $\Theta_2$-totalization $\Tot_{\Theta_2}(X)$. 
 \end{prop}
 \comment
 \begin{remark}{\rm
 We provide two proofs, among which the first one is more conceptual, and the second one is more straightforward. 
 In fact, the two complexes in the statement are even {\it isomorphic} in the non-normalized setting, as follows from the second proof. 
 }
 \end{remark}
 \endcomment
 \comment
\noindent

{\it First proof:}\\
Let $C,D$ be small categories, $K\colon C\to D$ a functor. For a small-complete and small-cocomplete category $\mathcal{E}$, the restriction functor $K^*\colon \Fun(D,\mathcal{E})\to\Fun(C,\mathcal{E})$ has a left adjoint (which is $\Lan_K(-)$) and a right adjoint (which is $\Ran_K(-)$) functors. \\
Assume $\mathcal{E}=C(\k)$. The functor $K^*$ is exact, and the left homotopy Kan extension $\mathbf{L}\Lan_K(-)$ is left adjoint and the right homotopy Kan extension $\mathbf{R}\Ran_K(-)$ is right adjoint to it, as functors between the homotopy (in our case, derived) categories. 

Turn back to the proof of Proposition. The totalization $\Tot_{\Theta_2}(-)$ is quasi-isomorphic to the right homotopy Kan extension $\mathbf{R}\Ran_{\Theta_2\to*}(-)$ (see ???), and similarly $\Tot_\Delta(-)=\mathbf{R}\Ran_{\Delta\to *}(-)$. 
Denote the functors 
$$
\xymatrix{
\Theta_2\ar@/_1pc/[rr]_{P_2}\ar[r]^{p}&\Delta\ar[r]^{P_1}&{*}\\
}
$$
Clearly we have for the restriction functors $P_2^*=p^*\circ P_1^*$, and therefore similarly we have 
$\mathbf{R}\Ran_{P_1}\circ \mathbf{R}\Ran_{p}(X)\simeq\mathbf{R}\Ran_{P_2}(X)$ for their derived right adjoints. \\
\\
\endcomment

\begin{proof}
When one applies the usual non-normalized cochain complex functor to the cosimplicial vector space $Rp_*(X)$, we get exactly the formula \eqref{tott} for the (non-normalized) $\Theta_2$-totalization.
\end{proof}

\comment
\begin{remark}{\rm
If one imagines that one has a homotopy limit/ homotopy right Kan extension interpretation of our complexes (that is, if $A(C,D)(F,G)(\eta,\theta)$ were Reedy fibrant, which is not the case) one could give a more conceptual and less computational proof of the transitivity in the statement of Proposition \ref{prop!}, using general derived functor of composition theorem. 
}
\end{remark}
\endcomment

\section{\sc Bicategories and 2-bimodules over bicategories}\label{sectionbicat}

\subsection{\sc Reminder on bicategories}
Here we recall the basic definitions related to bicategories, basically aiming to fix our terminology and notations. The reference are [ML], [Ke], [Be], [L]. 

Shortly, a bicategory is a lax category enriched in $\Cat$. Here ``lax'' indicates that the associativity of composition holds up to a 2-arrow which is an isomorphism. Below is a detailed definition (of a small bicategory).

\begin{defn}{\rm
A small bicategory $C$ consists of the following data: 
\begin{itemize}
\item[(1)] a set $C_0$ whose elements are called the objects of $C$,
\item[(2)] for $x,y\in C_0$, a set $C_1(x,y)$ whose elements are called 1-morphisms or 1-arrows,
\item[(3)] for any $x,y\in C_0$ and $f,g\in C_1(x,y)$, a set $C_2(x,y)(f,g)$ called 2-morphisms or 2-arrows, 
\item[(4)] for any $x,y\in C_0$, one has a small category $C(x,y)$ whose objects are $C_1(x,y)$ and whose arrows $f\to g$ are $C_2(x,y)(f,g)$,
\item[(5)] for any $x,y,z\in C_0$, there is a composition $m_{x,y,z}\colon C(y,z)\times C(x,y)\to C(x,z)$, which is a functor of categories,
\item[(6)] for any $x,y,z,w\in C_0$, there is a natural transformation
\begin{equation}
\begin{tikzpicture}[baseline=(current bounding box.center), scale=1]
\node (1) at (0,0){$C(z,w)\times C(y,z)\times C(x,y)$};
\node (2) at (0,-2){$C(z,w)\times C(x,z)$};
\node (3) at (6,0){$C(y,w)\times C(x,y)$};
\node (4) at (6,-2){$C(x,w)$};
\node (B) at (3.7,-1){$\Uparrow_{\alpha_{x,y,z,w}}$};
 \path[->, font=\scriptsize]
 (1) edge["$\id\times m_{x,y,z}$"] (2)
 (1) edge["$m_{y,z,w}\times \id$"] (3)
 (2) edge["$m_{x,z,w}$"] (4)
 (3) edge["$m_{x,y,w}$"] (4);
 \end{tikzpicture}
\end{equation}
whose components are given by isomorphisms (for which we often use the notation $\alpha_{h,g,f}$ assuming $\alpha_{x,y,z,w}(h\times g\times f)$),
\item[(7)] 
for any $x,y\in C_0$, there are 2-arrows $\rho_{x,y}$ and $\lambda_{x,y}$ defined as 
\begin{equation}
\begin{tikzpicture}[baseline=(current bounding box.center), scale=1]
 \node (1) at (0,0){${C(x,y)\times \mathbf{1}}$};
 \node (2) at (0,-2){${C(x,y)\times C(x,x)}$};
 \node (3) at (4,-2){$C(x,y)$};
 \node (A) at (2,-1.5){$\Uparrow_{\rho_{x,y}}$};
 \node (4) at (8,0){${\mathbf{1}\times C(x,y)}$};
\node (5) at (8,-2){${C(y,y)\times C(x,y)}$};
\node (6) at (12,-2){$C(x,y)$};
\node (B) at (10,-1.5){$\Uparrow_{\lambda_{x,y}}$};
 \path[->, font=\scriptsize]
 (1)edge["${\id\times\id_x}$"'](2)
 (1)edge["$\simeq$"](3)
 (2)edge["${m_{x,x,y}}$"'](3) 
 (4)edge["${\id_y\times\id}$"'](5)
 (4)edge["$\simeq$"](6)
 (5)edge["${m_{x,y,y}}$"'](6);
 \end{tikzpicture}
\end{equation}
whose components are isomorphisms
(we often denote $\rho_{x,y}(f\times\id)$ and $\lambda_{x,y}(\id\times g)$ by $\rho(f)$ and $\lambda(g)$, correspondingly)
\end{itemize}

which are subject to the following identities:
\begin{itemize}
\item[(i)] The associator $\alpha_{x,y,z,w}$ is subject to the usual pentagon diagram:
\begin{equation}
    \begin{tikzpicture}[commutative diagrams/every diagram]
  \node (P0) at (90:2.3cm) {$(t\circ h)\circ (g\circ f)$};
  \node (P1) at (90+72:2cm) {$t\circ (h\circ (g\circ f))$} ;
  \node (P2) at (90+2*72:2cm) {\makebox[5ex][r]{$t\circ ((h\circ g)\circ f)$}};
  \node (P3) at (90+3*72:2cm) {\makebox[5ex][l]{$(t\circ (h\circ g))\circ f$}};
  \node (P4) at (90+4*72:2cm) {$((t\circ h)\circ g)\circ f$};
  \path[commutative diagrams/.cd, every arrow, every label]
    (P0) edge node {$\alpha_{t\circ h,g,f}$} (P4)
    (P1) edge node {$\alpha_{t,h,g\circ f}$} (P0)
    (P2) edge node[swap] {$\alpha_{t,h\circ g,f}$} (P3)
    (P1) edge node[swap] {$m_{t,\alpha_{h,g,f}}$} (P2)
    (P3) edge node[swap] {$m_{\alpha_{t,h,g},f}$} (P4);
\end{tikzpicture}
\end{equation}
\item[(ii)] for any $x,y\in C_0$, the unit 2-arrows
$$
\lambda(f) \colon m_{x,y,y}(\id_y, f)\Rightarrow f
$$
and 
$$
\rho(f)\colon m_{x,x,y}(f,\id_x)\Rightarrow f
$$
are subject to the following unit identity, for $f\in C_1(x,y), g\in C_1(y,z)$:
$$
\xymatrix{
g\circ (\id_y\circ f)\ar[rr]^{\alpha_{x,y,z,,w}}\ar[rd]_{\id(g)\circ\lambda(f)}&&(g\circ \id_y)\circ f\ar[dl]^{\rho(g)\circ\id(f)}\\
&g\circ f
}
$$
\end{itemize}
}
\end{defn}

The composition  of 2-arrows coming from the composition in $C(x,y)$ is called {\it vertical} (notation: $f\circ_v g$, where $f,g\in C(x,y)$), while the composition of 2-arrows 
coming from $m_{x,y,z}\colon C(y,z)\times C(x,z)\to C(x,z)$ is called {\it horizontal} (notation: $g\circ_h f$, where $f\in C(x,y), g\in C(y,z)$). 
\\

{\it Examples.}

(i) When $\alpha_{x,y,z,w}, \lambda(f), \rho(f)$ are the identity 2-arrows, one gets the concept of a 2-category (which is then a ``strict'' bicategory).

(ii) When $C_0=\{*\}$, the bicategory is a monoidal category. 

\vspace{0.5cm}

Next we recall lax-functors of bicategories and lax-natural transformations thereof. Also we recall the concept of a {\it modification} between two lax natural transformations, playing the role of 3-arrows in the (weak) 3-category of bicategories (they form a tricategory, see [GPS]). 

\begin{defn}\label{deflaxfunctor}{\rm
Let $C,D$ be bicategories. A {\it lax-functor} $F\colon C\to D$ is given by the following data:
\begin{itemize}
\item[(1)] an assignment $F\colon C_0\to D_0$ of objects, 
\item[(2)] for any $x,y\in C_0$, a functor $F_{x,y}\colon C(x,y)\to D(Fx,Fy)$,
\item[(3)] for any $x,y,z\in C_0$ there is a natural transformation $\phi_{x,y,z}$ making the diagram below commutative:
\begin{equation}
\begin{tikzpicture}[baseline=(current bounding box.center), scale=1]
\node (1) at (0,0){{$C(y,z)\times C(x,y)$}};
\node (2) at (0,-2){{$D(Fy,Fz)\times D(Fx,Fy)$}};
\node (3) at (6,0){{$C(x,z)$}};
\node (4) at (6,-2){{$D(Fx,Fz)$}};
\node (A) at (3,-1){$\Rightarrow_{\phi_{x,y,z}}$};
\path[->,font=\scriptsize]
(1) edge["$F_{y,z}\times F_{x,y}$"'] (2)
(1) edge["$m^C_{x,y,z}$"] (3)
(2) edge["$m^D_{Fx,Fy,Fz}$"] (4)
(3) edge["$F_{x,y}$"] (4);
\end{tikzpicture}
\end{equation}
(we use notations $\phi_{g,f}\colon Fg\circ Ff\Rightarrow F(gf)$),
\item[(4)] for any $x\in C_0$ there is a natural transformation $\phi_x$ making the diagram below commutative:
\begin{equation}
\begin{tikzpicture}[baseline=(current bounding box.center), scale=1]
\node (1) at (0,0){$\mathbf{1}$};
\node (3) at (0,-2){$\mathbf{1}$};
\node (2) at (4,0){$C(x,x)$};
\node (4) at (4,-2){$D(Fx,Fx)$};
\node (A) at (2,-1){$\Rightarrow_{\phi_x}$};
\path[->,font=\scriptsize]
(1) edge["$I_x$"] (2)
(1) edge["$=$"] (3)
(2) edge ["$F_{x,x}$"](4)
(3) edge ["$I^\prime_{Fx}$"] (4);
\end{tikzpicture}
\end{equation}
\end{itemize}
 which are subject to the following properties
\begin{itemize}
\item[(i)]
the hexagon diagram
\begin{equation}
\begin{tikzpicture}[baseline=(current bounding box.center), scale=1]
\node (1) at (0,0){$(Fh\circ Fg)\circ Ff$};
\node (2) at (4,0){$F(h\circ g)\circ Ff$};
\node (3) at (8,0){$F((h\circ g)\circ f)$};
\node (4) at (0,-2){$Fh\circ (Fg\circ Ff)$};
\node (5) at (4,-2){$Fh\circ F(g\circ f)$};
\node (6) at (8,-2){$F(h\circ (g\circ f))$};
\path[->,font=\scriptsize]
(1) edge ["$\phi\circ\id$"] (2)
(2) edge ["$\phi$"] (3)
(3) edge ["$F(\alpha^C)$"] (6)
(1) edge ["$\alpha^D$"] (4)
(4) edge ["$\id\circ\phi$"] (5)
(5) edge ["$\phi$"] (6);
\end{tikzpicture}
\end{equation}
commutes,
\item[(2)] the unit commutative diagrams for $\phi_x$, see e.g. [L], 1.1.
\end{itemize}
When $\{\phi_{x,y,z}\}$ and $\{\phi_x\}$ are isomorphisms, the lax functor $F$ is called {\it strong}, when they are identity natural transformations the lax functor is called {\it strict}.
}
\end{defn}

\begin{defn}{\rm
Let $C,D$ be bicategories, $F=(F,\phi),G=(G,\psi)\colon C\to D$ lax-functors. A {\it lax natural transformation} $\eta\colon F\Rightarrow G$ is given by the following data
\begin{itemize}
\item[(1)] a collection of 1-arrows $\eta_x\colon Fx\to Gx$, $x\in C_0$,
\item[(2)] a collection of natural transformations
\begin{equation}
\begin{tikzpicture}[baseline=(current bounding box.center), scale=1]
\node (1) at (0,0){$C(x,y)$};
\node (2) at (4,0){$D(Fx,Fy)$};
\node (3) at (0,-2){$D(Gx,Gy)$};
\node (4) at (4,-2){$D(Fx,Gy)$};
\node (A) at (2,-1){$\Rightarrow_{\eta_{x,y}}$};
\path[->,font=\scriptsize]
(1) edge ["$F_{x,y}$"] (2)
(1) edge ["$G_{x,y}$"] (3)
(2) edge ["$(\eta_y)_*$"] (4)
(3) edge ["$(\eta_x)^*$"] (4);
\end{tikzpicture}
\end{equation}
which are subject to commutative diagrams expressing the compatibility of $\{\eta_f\}$ with the composition $f_2\circ f_1$ and of $\{\eta_x\}$ with unit morphisms, see e.g. [L], 1.2. 
The 2-arrows $\eta_{x,y}$ are given in components by 2-arrows $\eta_f\colon G(f)\circ \eta_x\Rightarrow\eta_y\circ F(f)\colon F(x)\to G(y)$, for any $f\in C_1(x,y)$. 
\end{itemize}
A lax-monoidal transformation of bicategories is called {\it strong} if the 2-arrows $\{\eta_f\}$ are isomorphisms, and it is called {\it strict} if they are identity 2-arrows. 
}
\end{defn}

\begin{defn}\label{defmodif}{\rm
Let $C,D$ be bicategories, $F,G\colon C\to D$ lax-functors, $\eta,\theta\colon F\Rightarrow G$ lax natural transformations. 
A {\it modification} $\kappa\colon \eta \Rrightarrow \theta$ is given by a collection of 2-arrows $\{\Gamma_x\}$ for $x\in C_0$:
\begin{equation}
\begin{tikzpicture}[baseline=(current bounding box.center), scale=1]
\node (1) at (0,0){$Fx$};
\node (2) at (4,0){$Gx$};
\node (A) at (2.1,0){$\Downarrow^{\Gamma_x}$};
\path[->,font=\scriptsize]
(1) edge[bend right, "$\eta_x$"'] (2)
(1) edge[bend left, "$\theta_x$"] (2);
\end{tikzpicture}
\end{equation}
(denoted as $\Gamma\colon \eta\Rrightarrow\theta$) which are subject to the following axiom: for any $f\colon x\to y\in C_1$, the diagram below commutes:
\begin{equation}
\begin{tikzpicture}[baseline=(current bounding box.center), scale=1]
\node (1) at (0,0){$Gf\circ \eta_x$};
\node (2) at (4,0){$Gf\circ \theta_x$};
\node (3) at (0,-2){$\eta_y\circ Ff$};
\node (4) at (4,-2){$\theta_y\circ Ff$};
\path[->,  font=\scriptsize]
(1) edge[double, "$\id\circ \Gamma_x$"] (2)
(1) edge[double, "$\eta_f$"'] (3)
(2) edge[double, "$\theta_f$"] (4)
(3) edge[double, "$\Gamma_y\circ \id$"] (4);
\end{tikzpicture}
\end{equation}
 
}
\end{defn}

\begin{defn}\label{enrichbicat}{\rm
For a symmetric monoidal category $\mathscr{V}$, by $\mathscr{V}$-enriched bicategory we mean a modification of the previous definition, in which $C_2(x,y)(f,g)\in\mathscr{V}$ (while $C_0, C_1(x,y)\in \Sets$), so that each category $C(x,y)$ is a $\mathscr{V}$-enriched category, for any $x,y\in C_0$, and the compositions $m_{x,y,z}$ are given by $\mathscr{V}$-functors. 
Similarly one defines $\mathscr{V}$-lax functors, $\mathscr{V}$-lax natural transformations, and $\mathscr{V}$-modifications. 
Namely, one adjusts axioms by requiring that the 2-arrows $\phi_x,\phi_{x,y,z}, \eta_x,\eta_{x,y},\Gamma_x$ are objects of $\mathscr{V}$. The reader will find details in [Ke]. 

A particular case we deal with in this paper is $\mathscr{V}=C^\udot(\k)$, the category of complexes over a field $\k$. In this case, we call a $\mathscr{V}$-enriched bicategory a dg bicategory. 
}
\end{defn}

\subsubsection{\sc Coherence for bicategories}
We'll need the coherence theorem for bicategories [Be], [ML], [L]. 

Let $C,D$ be two bicategories. They are said to be {\it biequivalent} if there are {\it strong} functors $F\colon C\to D$ and $G\colon D\to C$ and {\it strong} natural transformations $\id_C\to G\circ F$ and $F\circ G\to\Id_D$ (that is, we use only the underlying 1-category structure on the bicategory $[C,D]$, the modifications are irrelevant for this definition). The coherence theorem for bicategories is the following result:

\begin{theorem}\label{theorcoherence}
Every bicategory is biequivalent to a 2-category. 
\end{theorem}
See e.g. [L, Sect. 2.3] for a short proof, based on the Yoneda embedding. 
\qed

In a 2-category, the associator and the identity maps are equal to identity. The following statement, which we primarily will use in the paper, is a direct consequence of Theorem \ref{theorcoherence}:

\begin{coroll}
Let $C$ be a bicategory, $x,y\in C_0$, $f,g\colon C_1(x,y)$. Let $\eta,\theta\colon f\Rightarrow g$ be two natural transformations each of which is a composition of the associators and the unit maps. Then $\eta=\theta$.
\end{coroll}
\qed

A similar statement holds for enriched bicategories, in particular, for dg bicategories. 

\subsection{\sc 2-bimodules over a bicategory}\label{section2bimod}
Let $\mathscr{V}$ be a symmetric monoidal category, $C$ a $\mathscr{V}$-enriched bicategory (see Definition \ref{enrichbicat}).
A $C$-2-bimodule consists of the following data:
\begin{itemize}
\item[(i)] A 2-globular set $M$ whose second component $M_2$ is $\mathscr{V}$-enriched, and $M_{\le 1}=C_{\le 1}$,
\item[(ii)] The upper and the lower vertical compositions:
$$
\bar{\circ}^v_-\colon M_2(g,h)\otimes C_2(f,g)\to M_2(f,h),\ \ \bar{\circ}^v_+\colon C_2(g,h)\otimes M_2(f,g)\to M_2(f,h)
$$
where $f,g,h\in C_1(x,y)$,
\item[(iii)] The left and the right horizontal compositions 
$$
\bar{\circ}^h_-\colon M_2(g_1,g_2)\otimes C_2(f_1,f_2)\to M_2(g_1f_1,g_2f_2),\ \ \bar{\circ}^h_+\colon C_2(g_1,g_2)\otimes M_2(f_1,f_2)\to M_2(g_1f_1,g_2f_2)
$$
where $f_1,f_2\in C_1(x,y), g_1,g_2\in C_1(y,z)$
\end{itemize}

which are subject to the following properties:
\begin{itemize}
\item[(1)] the vertical compositions are strictly associative, in the sense that for $f_1,f_2,f_3,f_4\in C_1(x,y)$ the following three identities hold:
$$
m\bar{\circ}^v_-(\psi\circ^v\phi)=(m\bar{\circ}^v_- \psi)\bar{\circ}^v_-\phi
$$
where $m\in M_2(f_3,f_4), \phi\in C_2(f_1,f_2), \psi\in C_2(f_2,f_3)$,
$$
\psi\bar{\circ}^v_+(m\bar{\circ}^v_-\phi)=(\psi\bar{\circ}^v_+ m)\bar{\circ}^v_-\phi
$$
where $\psi\in C_2(f_3,f_4), m\in M_2(f_2,f_3), \phi\in C_2(f_1,f_2)$,
$$
\psi\bar{\circ}^v_+(\phi\bar{\circ}^v_+m)=(\psi\circ \phi)\bar{\circ}^v_+m
$$
where $\psi\in C_2(f_3,f_4), \phi\in C_2(f_2,f_3), m\in M_2(f_1,f_2)$,
\item[(2)] the horizontal compositions are associative up to the associator $\alpha$ in $C$, in the sense that for
$x,y,z,w\in C_0$, $f_1,f_2\in C_1(x,y)$, $g_1,g_2\in C_1(y,z)$, $h_1,h_2\in C_1(z,w)$, one has:
\begin{equation}\label{modass1}
\alpha_{h_2,g_2,f_2}\bar{\circ}^v_+\big(m\bar{\circ}^h_-(\psi\circ^h\phi))=\big((m\bar{\circ}^h_-\psi)\bar{\circ}^h_-\phi\big)\bar{\circ}^v_-\alpha_{f_1,g_1,h_1}
\end{equation}
where $\phi\in C_2(f_1,f_2), \psi\in C_2(g_1,g_2), m\in M_2(h_1,h_2)$,
\begin{equation}\label{modass2}
\alpha_{f_2,g_2,h_2}\bar{\circ}^v_+\big(\psi\bar{\circ}^h_+(m\bar{\circ}^h_- \phi)\big)=\big((\psi\bar{\circ}^h_+ m)\bar{\circ}^h_-\phi\big)\bar{\circ}^v_-\alpha_{f_1,g_1,h_1}
\end{equation}
where $\phi\in C_2(f_1,f_2), m\in M_2(g_1,g_2), \psi\in C_2(h_1,h_2)$,
\begin{equation}\label{modass3}
\alpha_{f_2,g_2,h_2}\bar{\circ}^v_+\big(\psi\bar{\circ}^h_+(\phi\bar{\circ}^h_+m)\big)=
\big((\psi\circ\phi)\bar{\circ}_h^+m\big)\bar{\circ}^v_-\alpha_{f_1,g_1,h_1}
\end{equation}
where $\phi\in C_2(g_1,g_2), \psi\in C_2(h_1,h_2), m\in C_2(f_1,f_2)$,
\item[(3)] four Eckmann-Hilton identities (depending on the place among the four arguments at which an element of $M_2$ is, then there are elements of $C_2$ on the three others). One among these 4 identities reads:
$$
(\phi_3\bar{\circ}^v_+m)\bar{\circ}^h_-(\phi_2\circ^v \phi_1)=(\phi_3\circ^h \phi_2)\bar{\circ}^v_+(m\bar{\circ}^v_- \phi_1)
$$
where $f_1,f_2,f_3\in C_1(x,y)$, $g_1,g_2,g_3\in C_1(y,z)$, $\phi_1\in C_2(f_1,f_2)$, $\phi_2\in C_2(f_2,f_3)$, $\phi_3\in C_2(g_1,g_2)$, $m\in M_2(g_2,g_3)$;
the three other identities are similar and are left to the reader,
\item[(4)] let $x,y\in C_0$, $f,g\in C_1(x,y)$, $m\in M_2(f,g)$, then:
\begin{equation}\label{modunit}
\rho(g)\bar{\circ}^v_+(m\bar{\circ}^h_-\id(x))\bar{\circ}^v_-\rho(f)^{-1}=m,\ \ \lambda(g)\bar{\circ}^v_+(\id(y)\bar{\circ}^h_+m)
\bar{\circ}^v_-\lambda(f)^{-1}=m
\end{equation}

\end{itemize}

The category of 2-bimodules over a bicategory $C$ is denoted $\Bimod_2(C)$. 

As a trivial example, $M=C$ satisfies the axioms; it is called the {\it tautological} 2-bimodule over $C$. 

\vspace{1cm}

\subsection{\sc Examples}
\subsubsection{\sc The free 2-bimodule}\label{free2bimod}
There is a forgetful functor $U$ from $C$-2-bimodules to 2-globular sets whose second component is enriched over $\mathscr{V}$. Assume for simplicity that $\mathscr{V}=\Vect(\k)$ or $C^\udot(\k)$.
The functor $U$ has a left adjoint $L$; the corresponding $C$-2-bimodule $L(M_0)$ is called the {\it free $C$-2-bimodule generated by the $\mathscr{V}$-enriched 2-globular set $M_0$}. We set $L(M_0)_{\le 1}=C_{\le 1}$, let $x,y\in C_0$ and $f,g\in C_1(x,y)$. Define
\begin{equation}\label{eqfree2bimod}
M(f,g)=\bigoplus_{\substack{{z_0,z_1\in C_0}\\
{\alpha\in C_1(x,z_0), \beta\in C_1(z_1,y)}\\
{f_0,g_0\in C_1(z_0,z_1)}\\
}}C(\beta\circ g_0\circ \alpha,g)\otimes M_0(f_0,g_0)\otimes C(f,\beta\circ f_0\circ \alpha)/\sim
\end{equation}
where the equivalence relation identifies the following elements, in notations of diagram \eqref{free2modpic} below:
\begin{equation}\label{free2bimodrel}
(\phi_2\circ^v(\sigma_2\circ^h\id_{g_0}\circ^h\sigma_1))\otimes m\otimes \phi_1\sim
\phi_2\otimes m\otimes((\sigma_2\circ^h\id_{f_0}\circ^h\sigma_1)\circ^v\phi_1)
\end{equation} 
\begin{equation}\label{free2modpic}
    \begin{tikzpicture}[baseline=(current bounding box.center),scale=1]
    \node (1) at (0,0){\scriptsize{$x$}}; \node(2) at (2,0){\scriptsize{$z_0$}}; \node(3) at (5,0){\scriptsize{$z_1$}}; \node (4) at (7,0){\scriptsize{$y$}}; 
    \node (A) at (1,0) {$\Uparrow_{\sigma_1}$}; \node (B) at (3.5,0) {$\Uparrow_m$}; \node (C) at (6,0) {$\Uparrow_{\sigma_2}$}; \node (D) at (3.5,1.35) {$\Uparrow_{\phi_2}$}; \node (E) at (3.5,-1.35) {$\Uparrow_{\phi_1}$};
    \path[->, font=\scriptsize]
    (1)edge[bend left, "$\alpha_+$"](2)(1)edge[bend right, "$\alpha_-$"'](2)(3)edge[bend left, "$\beta_+$"](4)(3)edge[bend right, "$\beta_-$"'](4)(2)edge[bend left, "$g_0$"](3)(2)edge[bend right, "$f_0$"'](3)(1)edge[bend left=60,"$g$"](4)(1)edge[bend right=60,"$f$"'](4);
    \end{tikzpicture}
\end{equation}
One easily checks that indeed 
\begin{equation}
\Hom_{\Bimod_2}(L(M_0),N)\simeq \Hom_{\mathrm{Glob}_2}(M_0, U(N))
\end{equation}

\subsubsection{\sc The 2-bimodule $M(C,D)(F,G)(\eta,\theta)$}\label{section2bimodtheta}
Let $C,D$ be $\k$-linear bicategories, $F,G\colon C\to D$ {\it strong} lax functors, $\eta,\theta\colon F\Rightarrow G$ two {\it strong} lax natural transformations. 

We associate with this data the following 2-$C$-bimodule $M(C,D)(F,G)(\eta,\theta)$: 
\begin{equation}\label{thebimodule}
 M(C,D)(F,G)(\eta,\theta)(f,g)=D(\eta(y)\circ F(f), G(g)\circ \theta(x))
\end{equation}
where $f,g\in C(x,y)$. 

\begin{lemma}
In notations as above, $M(C,D)(F,G)(\eta,\theta)$ is a 2-bimodule over $C$. 
\end{lemma}

\begin{proof}
Let $m\in M(C,D)(F,G)(\eta,\theta)(f,g)$. 

For a 2-morphism $\alpha\colon f^\prime\Rightarrow f$, the vertical composition $m\bar{\circ}_{-}^v \alpha$ is defined as the vertical composition $m{\circ}^v (\eta(y)\circ F(\alpha))$ in $D$. Similarly, for $\beta\colon g\Rightarrow g^\prime$, the vertical composition $\beta\bar{\circ}_+^v m$ is defined as the vertical composition $G(\beta)\circ \theta(x)){\circ}^v m$. 

For $f_0,g_0\in C(w,x), \alpha\in C(f_0,g_0)$, define the 2-morphism $(m\circ^h_{-}\alpha)_0$ in $D$ as the horizontal composition in 
$m\circ^h F(\alpha)$ in $D$ post-composed vertically with the 2-morphism $\theta(x)\circ F(g_0)\Rightarrow G(g_0)\circ \theta(x_0)$ (whiskered by $G(g)$).
Define the horizontal composition $m\bar{\circ}^h_{-}\alpha=\psi \circ^v (m\circ^h_-\alpha)_0)\circ^v \phi$
where $\phi$ is the vertical composition of 2-arrows $\eta(y)\circ F(f\circ f_0)\eqto \eta(y)\circ (F(f)\circ F(f_0))\eqto 
(\eta(y)\circ F(f))\circ F(f_0)$, and $\psi$ is the composition $G(g)\circ (G(g_0)\circ \eta(w))\eqto (G(g)\circ G(g_0))\circ \eta(w)\eqto G(g\circ g_0)\circ \eta(w)$.

For $f_1,g_1\in C(y,z)$, $\beta\in C(f_1,g_1)$, define the 2-morphism  $(\beta\circ^v_{+} m)_0$ in $D$ as the horizontal composition $m{\circ}^h G(\beta)$ in $D$ pre-composed vertically with the 2-morphism 
$\eta(z)\circ F(f_1)\Rightarrow G(f_1)\circ \eta(y)$ (whiskered by $F(f)$). Define the horizontal composition $\beta\bar{\circ}^v_+ m$ as $\psi\circ^v (\beta\circ^v_+ m)_0\circ^v \phi$, where $\phi$ is the vertical composition of 2-arrows
$\eta(z)\circ F(f_1\circ f)\eqto \eta(z)\circ (F(f_1)\circ F(f))\eqto (\eta(z)\circ F(f_1))\circ F(f)$, and $\psi$ is the vertical composition $G(g_1)\circ (G(g)\circ \theta(x))\eqto (G(g_1)\circ G(g))\circ \theta(x)\eqto G(g_1\circ g)\circ \theta(x)$. 
\begin{equation}
    \begin{tikzpicture}[baseline=(current bounding box.center),scale=1]
    \node (1) at (0,0){\scriptsize{$F(w)$}}; \node(2) at (0,2){\scriptsize{$G(w)$}}; \node(3) at (4,0){\scriptsize{$F(x)$}}; \node (4) at (4,2){\scriptsize{$G(x)$}}; \node(5) at (8,0){\scriptsize{$F(y)$}}; \node (6) at (8,2){\scriptsize{$G(y)$}};
    \node (A) at (2,0) {$\Uparrow_{F(\alpha)}$}; \node (B) at (6,1) {$\Uparrow_{m}$}; \node (C) at (2,1.75) {$\Uparrow_{\theta(g_0)}$}; \node (D) at (4,-1.25) {$\Uparrow_{\varphi}$}; \node (E) at (4,3.25) {$\Uparrow_{\psi}$};
    \path[->, font=\scriptsize]
    (1) edge["$\theta(w)$"](2)(3)edge["$\theta(x)$"'](4)(5)edge["$\eta(y)$"'](6)(1)edge[bend right, "$F(f_0)$"'](3)(1)edge[bend left, "$F(g_0)$"](3)(2)edge[bend left, "$G(g_0)$"](4)(4)edge[bend left, "$G(g)$"](6)(3)edge[bend right, "$F(f)$"](5)(1)edge[bend right=60, "$F(f\circ f_0)$"'](5)(2)edge[bend left=60, "$G(g\circ g_0)$"](6);
    \end{tikzpicture}
\end{equation}

One has to check the compatibilities with the associator \eqref{modass1}-\eqref{modass3} and the compatibility with the unit \eqref{modunit}. It is left to the reader.

\end{proof}

Note that the particular case $C=D$, $F=G=\Id$, $\eta=\theta=\id$ recovers the tautological 2-bimodule over $C$ (for which $M_2(x,y)=C_2(x,y)$). 

\subsection{\sc The category $\hat{\Theta}_2$}\label{sectionthetahat}
In analogy with the category $\Theta_2$, we define here its ``bicategorical extension'' $\hat{\Theta}_2$, playing an important role in definition of the complex $A(C,D)(F,G)(\eta,\theta)$ in the next Section. 

The objects of $\hat{\Theta}_2$ are the same as the objects of $\Theta_2$, that is, the 2-level trees $T$. The morphisms are defined in terms of 2-globular sets ${T}^*$, namely
\begin{equation}
\hat{\Theta}_2(S,T)=\Hom_\Bicat(\hat{\omega}_2(S^*), \hat{\omega}_2(T^*))
\end{equation}
where $\hat{\omega}_2(D)$ denotes the free {\it bicategory} generated by the 2-globular set $D$, and $\Hom_\Bicat(-,-)$ stands for the set of {\it strict} functors of bicategories. 

In sequel we need somewhat more explicit description of morphisms of $\hat{\Theta}_2$. 

First of all, we have the following explicit description of the bicategory $\hat{\omega}_2(T^*)$ (which follows from the Coherence Theorem \ref{theorcoherence}). The vertices are the 0-cells of $T^*$ which we denote by $\{0,1,\dots,k\}$. Consider an ordered set $S$ with $n$ elements. We consider an {\it extended} set $S(m)$, adding some $m$ marked new elements to the $n$ elements of $S$, along with a total order on $S(m)$ compatible with the total order on $S$. Define an {\it extended parenthesising} of $S$ as a full parenthesising of $S(m)$, $m\ge 0$. The 1-morphism set $(\hat{\omega}_2)_1(i,j)$ is empty if $i<j$, and is the set of {\it extended parenthesisings} of the ordered set of elementary 1-morphisms in the paths $\lambda(i,j)$ from $i$ to $j$ in $(T^*)_1$ (the marked elements indicate the positions of identity 1-morphisms). The two such paths are considered different also if they differ only by parenthesising or number and position of marked elements. The composition of 1-morphisms is defined naturally. 

The set of 2-morphisms $\hat{\omega}_2(\lambda(i,j), \lambda^\prime(i,j))$ consists of a single element if $\lambda^\prime(i,j)$ dominates $\lambda(i,j)$ (that is, in each column between $i$ and $j$ of $T^*$ the element of $\lambda^\prime(i,j)$ in this column is $\ge$ than the element of $\lambda(i,j)$ in this column, the domination relation does not depend on the new marked elements), and is empty set otherwise. The vertical and the horizontal compositions of 2-arrows are defined by the unique possible way. 

It gives rise to the following wreath-product-like description of morphisms in $\hat{\Theta}_2$.

A morphism $\Phi$ from $([k]; [n_1],\dots,[n_k])$  to $([\ell]; [m_1],\dots,[m_\ell])$ in $\hat{\Theta}_2$ is given by $(\phi,\{\phi_i\})$
\begin{itemize}
\item[(i)] a morphism $\phi\colon [k]\to [\ell]$ in $\Delta$,
\item[(ii)] for each $1\le i\le k$ maps $\phi_i^j\colon [n_i]\to [n_j]$, $\phi(i-1)+1\le j\le \phi(i)$,
\item[(iii)] for each $a\in [n_i]$ an extended parenthesising of the $(\phi(i)-\phi(i-1))$-element set (considered as the set $\{\phi_i^j(a)\}_{j}$) (for given $i$, the number of added marked elements and the extended parenthesising may depend on $a\in [n_i]$).
\end{itemize}
The composition of morphism is defined naturally. The only new feature compared with the case of $\Theta_2$ is that here we also compose parethesisings, in the operadic way. 

\vspace{5mm}

We are interested in explicit form of the left Kan extension $\Lan_{p^\opp}(F)$ of a functor $F\colon \hat{\Theta}_2^\opp\to \mathscr{E}$ along the functor $p^\opp\colon \hat{\Theta}_2^\opp\to\Theta_2^\opp$, where in our main examples $\mathscr{E}$ is an abelian category. 

Let $\phi\colon D\to D^\prime$ be a map in $\Theta_2$, we are interested in a lift $\hat{\phi}\colon D\to D^\prime$ in $\hat{\Theta}_2$, so that $p(\hat{\phi})=\phi$. It is clear that each map $\phi$ admits a non-empty (in fact, a finite) set of lifts. They depend on (multiple) insertions of the associator 2-arrows and the unit 2-arrows.

\begin{lemma}\label{lemmakanexttilde}
Let $F\colon \hat{\Theta}_2^\opp\to\mathscr{E}$ be a functor to an abelian category $\mathscr{E}$. Then the left Kan extension $\Lan_{p^\opp}(F)$
along the projection $p^\opp\colon \hat{\Theta}_2^\opp \to\Theta_2^\op$ is given by 
\begin{equation}
\Lan_{p^\opp}(F)(D)=F(D)/V_D
\end{equation}
where $V_D\subset F(D)$ is generated by the elements of the form $\hat{\phi}^*(F)(\xi)-\hat{\hat{\phi}}^*(F)(\xi)$,
where $\phi\colon D\to D_1$ is a map in $\Theta_2$, and $\hat{\phi},\hat{\hat{\phi}}$ are two its lifts to a morphism in $\hat{\Theta}_2$, $\xi\in F(D_1)$.
\end{lemma}
\begin{proof}
By the classical formula for a point-wise Kan extension, one has
$$
\Lan_{p^\opp}(F)(D)=\colim_{D_1\to D\in p^{\opp}/D}F(D_1)
$$
The colimit is taken over the comma-category whose objects are $D_1\in \hat{\Theta}_2^\opp$ and a map $p(D_1)=D_1\to D$ in $\Theta_2^\opp$, the morphisms $\phi\colon (i_1\colon D_1\to D)\to (i_2\colon D_2\to D)$ are maps $\hat{\phi}\colon D_1\to D_2$ in $\hat{\Theta}_2^\opp$ such that $i_1=i_2\circ p^\opp(\hat{\phi})$. Replacing the opposite categories by the non-opposite ones, the colimit is taken over the category $I_D$ whose objects are morphisms $i_1\colon D\to D_1$ in $\hat{\Theta}_2$, and a morphism $\phi\colon (i_1\colon D\to D_1)\to (i_2\colon D\to D_2)$ is given by a map $\hat{\phi}\colon D_2\to D_1$ in $\hat{\Theta}_2$ such that $i_1=p(\hat{\phi})\circ i_2$. 

Clearly any object $\phi\colon D\to D_1$ of $I_D$ admits a morphism to $\id\colon D\to D$ (one has to take any lift $\hat{\phi}\colon D\to D_1$ of $\phi$). That is, $\id\colon D\to D$ is a generalised final object: any other object admits (a non-unique) morphism to it. It follows that the colimit is a quotient of $F(D)$. One easily gets the description of the quotient given in the statement. 
\end{proof}

Now we want to get a more explicit description of the quotient in the statement of Lemma \ref{lemmakanexttilde}.
It turns out that we easily can restrict ourselves with some ``minimal'' maps $\phi\colon D\to D_1$ in $\Theta_2$.
We call them ``standard maps''. There are two types of them: the associator morphism type and the unit morphism type.

\vspace{3mm}

\noindent{\it The associator type maps.}\\
Let $D=([k];[n_1],\dots,[n_k])$, $D_1=([k+2];[n_1],[n_2],\dots, [n_a],[n_a],[n_a], [n_{a+1}],\dots, [n_k])$. 
Consider the map $\Phi=(\phi,\{\phi_i^j\})\colon D\to D_1$ where 
$\phi(0)=0,\dots,\phi(a-1)=a-1, \phi(a)=a+2,\dots,\phi(k)=k+2$, $\phi_i^j=\id$ for all $i$ and $j$ (clearly one has a unique $j$ for $i\ne a$, and there are $j=a,a+1,a+2$ for $i=a$). 
Consider a lift $\hat{\Phi}$ for which $\phi_a(\ell)$ is either $(\ell_{a}\ell_{a+1})\ell_{a+2}$ or $\ell_a(\ell_{a+1}\ell_{a+2})$ (the choice depends on $\ell\in [n_a]$ (we denote by $\ell_j$ the element $\ell\in [n_j]$). One has to specify the image of the two-morphisms. Clearly the two-morphisms corresponded to 1-morphisms in $[n_i]$ (considered as the interval category $I_{n_i}$) are defined as the corresponding two-morphisms, for $i\ne a$. When $i=a$, the prescription is as follows:
the image of the 2-morphism corresponded to $m_\ell:=\ell\to\ell+1$ in $[n_a]$ is 
$$
\begin{cases}
((m_\ell)_{a}\circ_h( m_\ell)_{a+1})\circ_h (m_\ell)_{a+2}&\text{ if both }\phi_a(\ell)\text{ and }\phi_a(\ell+1)\text{ are of type }(**)*\\
(m_\ell)_{a}\circ_h(( m_\ell)_{a+1}\circ_h (m_\ell)_{a+2})& \text{ if both }\phi_a(\ell)\text{ and }\phi_a(\ell+1)\text{ are of type }*(**)\\
\alpha\circ_v\big(\big((m_\ell)_{a}\circ_h( m_\ell)_{a+1}\big)\circ_h (m_\ell)_{a+2}\big)&\text{ if }\phi_a(\ell)=(\ell_{a}\ell_{a+1})\ell_{a+2},\\ & \phi_a(\ell+1)=(\ell+1)_a\big((\ell+1)_{a+1}(\ell+1)_{a+2}\big)\\
\big(\big((m_\ell)_{a}\circ_h( m_\ell)_{a+1}\big)\circ_h (m_\ell)_{a+2}\big)\circ_v\alpha^{-1}&\text{ if }\phi_a(\ell)=\ell_{a}(\ell_{a+1}\ell_{a+2}),\\ & \phi_a(\ell+1)=\big((\ell+1)_a(\ell+1)_{a+1}\big)(\ell+1)_{a+2}
\end{cases}
$$
(note that we use here non-conventional writing from the left to the right for the composition of 1-morpisms and for the horizontal 
composition of 2-morphisms).

We refer to such morphisms as Type A ones. 

\vspace{3mm}

\noindent{\it The unit type maps.}\\
Here $D=([k]; [n_1],\dots,[n_k])$, $\Phi=\Id\colon D\to D$ in $\Theta_2$, but $\hat{\Phi}\colon D\to D$ in $\hat{\Theta}_2$ is defined by insertion of several marked points to $[k]$ (corresponded to the identity 1-morphisms). Accordingly, each $\phi_i^j$ is the identity modulo the market points, but with their presence it is given by some parenthesising (of words only 1 elements of which is not marked). The images of elementary 2-morphisms are defined accordingly, pre- or post-composing vertically with the unit 2-morphisms or their inverse ones.

We refer to such morphisms as Type U ones.

\vspace{3mm}

\begin{prop}\label{propkanexttilde}
In the notations of Lemma \ref{lemmakanexttilde},
\begin{equation}
\Lan_{p^\opp}(F)(D)=F(D)/W_D
\end{equation}
where $W_D\subset F(D)$ is generated by the elements of the form $\hat{\phi}^*(F)(\xi)-\hat{\hat{\phi}}^*(F)(\xi)$,
where $\phi\colon D\to D_1$ is a map in $\Theta_2$ listed in Type A or in Type U, and $\hat{\phi},\hat{\hat{\phi}}$ are two its lifts to a morphism in $\hat{\Theta}_2$, listed in Type A or Type U, $\xi\in F(D_1)$.
\end{prop}

\begin{proof}
It follows from Lemma \ref{lemmakanexttilde} and (a rather evident) observation that any morphism $\hat{\phi}\colon D\to D_1$ factors as $D\xrightarrow{\hat{\psi}}D^\prime\to D_1$, where $\hat{\psi}$ is of Type A or of Type U.
\end{proof}

\section{\sc The 2-cocellular vector space $A(C,D)(F,G)(\eta,\theta)$}

\subsection{\sc The Davydov-Yetter complex and an attempt of generalisation}\label{sectiondygen}
Let $C,D$ be $\k$-linear monoidal categories (that is, $\k$-linear bicategories with a single object), $F\colon C\to D$ a strong monoidal functor. 
For $n\ge 1$, denote by 
$$
F^{\otimes n}\colon C^{\otimes n}\to D
$$
defined on objects as
$$
F^{\otimes n}(X_1,\dots,X_n)=F(X_1)\otimes_D F(X_2)\otimes_D\dots\otimes_D F(X_n)
$$
(one has to fix any parenthesising of the right-hand side, for example from the left to the right). 

Define
\begin{equation}
A^n(F)=\mathrm{Nat}(F^{\otimes n},F^{\otimes n})=\prod_{X_1,\dots,X_n\in C}D(F(X_1)\otimes\dots\otimes F(X_n), F(X_1)\otimes\dots\otimes F(X_n))_{\mathrm{Nat}}
\end{equation}
where $\Nat$ stands for (linear) natural transformations. 

The assignment $[n]\rightsquigarrow A^n(F)$ gives rise to a cosimplicial object. Its non-normalized dg totalization is called {\it the 
Davydov-Yetter complex} of $F\colon C\to D$. Let us recall this definition, with notations for simplicial coface and codegeneracy maps from the beginning of Section \ref{nord}. Let $\Psi\in \hat{A}_{n}(F)$. 
The elementary coface maps $\partial^i\colon [n-1]\to [n]$, $1\le i\le n-1$, act by plugging $X_i\otimes X_{i+1}$ in place of the $i$-th argument $X_i$ of $\Psi$, followed by the application of the colax-map $F(X_i\otimes X_{i+1})\to F(X_i)\otimes F(X_{i+1})$ and rearranging the parentheses (note that by the MacLane coherence theorem one needn't specify the way by which the parentheses are rearranged, as any two such maps are equal). The extreme coface map $\partial^0$ act by $\Psi\mapsto \id_{F(X_0)}\otimes \Psi(X_1\otimes\dots\otimes X_k)$, and the other extreme coface map $\partial^k$ acts by $\Psi\mapsto \Psi\otimes \id_{F(X_k})$, followed by the necessary reparenthesizing. The codegeneracy map $\epsilon^i$ acts on $k$-cochain $\Psi$ by plugging the monoidal unit $e$ to the $i$-th position of $\Psi$, thus decreasing the number of arguments by 1, followed by the necessary rearrangements. The reader is referred to [EGNO], Ch.7 or [BD] for a more detailed description.

Now the question is: does the construction still give rise to a cosimplicial object in $\mathscr{V}$ when the polynaturality condition of Davydov-Yetter is dropped? The answer is negative, unless the monoidal category is strict (the associator and the unit maps are the identity maps), because the relations in $\Delta$ are not longer respected. Denote the action of elementary cosimplicial operators as above by $\mathcal{O}$. Then, for instance, the actions of $\mathcal{O}(\partial^{i+1})\circ \mathcal{O}(\partial^i)$ differs from $\mathcal{O}(\partial^i)\circ\mathcal{O}(\partial^i)$ only by the $i$-th argument, which is $(X_i\otimes X_{i+1})\otimes X_{i+2}$ for the first composition, and $X_i\otimes(X_{i+1}\otimes X_{i+2})$ for the second one. These two expressions are mapped one to another by the associator map $\alpha$. Therefore, in order the relation $\partial^{i+1}\partial^i=\partial^i\partial^i$ to be respected under the action $\mathcal{O}$, one has to {\it require the naturality with respect to $\alpha$} on $i$-th argument. Similar considerations are applied to the unit map and degeneracy operators. 

It is clear from this reasoning that the naturality under all monoidal {\it structure morphisms} maps (that is, compositions of products the associator, the unit maps, and its inverse, with the identity maps on some factors) acting in each argument is the {\it minimal} naturality condition one has to impose on the cochains to get a cosimplicial object in $\mathscr{V}$. 

A drawback of the Davydov-Yetter complex, from the point of view of deformation theory, is that it controls only deformation of the associator. The full deformation theory of a linear monoidal category should also control deformations of the underlying linear category (which are controlled by the Hochschild cochain complex of this underlying category) and the action of {\it morphisms} on the monoidal product (the monoidal product on objects is a non-linear data and thus is assumed to remain unchanged under the deformation). So our goal is to define a ``bigger'' {\it bi}complex of Davydov-Yetter type, 
whose ``vertical'' differential is of the Hochschild type, and whose ``horizontal'' differential is of Davydov-Yetter type. 
The classical Davydov-Yetter complex is obtained by a sort of truncation. This truncation is the kernel of the vertical (Hochschild type) differential of the 0-th row of the bicimplex (so the truncated complex is a subcomplex of the 0-th row of our bicomplex). Taking this kernel effects in imposing the Davydov-Yetter naturality condition.

\subsection{\sc A functor $\Bar\colon \hat{\Theta}_2\to\Bimod_2$}\label{sectiondefa1}
Let $C$ be small dg bicategory. 
In this subsection, we define a functor
\begin{equation}
\Bar(C)\colon \hat{\Theta}_2^\opp\to\Bimod_2(C)
\end{equation}
(playing the role of bicategorical bar-complex), as follows.

Let $T=([k];[n_1],\dots,[n_k])$ be an object of $\Theta_2$. Let $x^\prime,y^\prime\in C_0, f_0,g_0\in C_1(x^\prime,y^\prime)$. Define
\begin{equation}\label{eqbarcorr1}
(M_T)_2(f_0,g_0)=\bigoplus_{\substack{{x_0,\dots,x_k\in C_0, x_0=x,x_k=y}\\
{f_{ij}\in C_1(x_{i-1},x_i), j=0\dots n_i}\\
{f_{\circ\min}=f_0, f_{\circ\max}=g_0}}}
{\bigotimes_{\substack{{i=1\dots k}\\{j=1\dots n_i}}}} C(f_{i,j-1},f_{i,j})
\end{equation}
where
\begin{equation}\label{eqbarcorr0}
f_{\circ\min}:=f_{k0}\circ (f_{k-1,0}\circ(\dots (f_{20}\circ f_{10})\dots)),\ \ \ 
f_{\circ\max}=f_{kn_k}\circ(f_{k-1,n_{k-1}}\circ(\dots(f_{2n_2}\circ f_{1n_1})\dots))
\end{equation}
It is considered as the second component of the enriched 2-globular set $M_T$ such that $(M_T)_{\le 1}=C_{\le 1}$.

Next define a dg 2-bimodule $\Bar(C)_T$ over $C$ as
\begin{equation}\label{eqbarcorr2}
\Bar(C)_T=L(M_T)
\end{equation}
where $L\colon \mathrm{Glob}_2\to \Bimod_2$ is the left adjoint functor to the forgetful functor, which was discussed in Section \ref{free2bimod}.

Explicitly, for $x,y\in C_0, f,g\in C_1(x,y)$, one has
\begin{equation}
\Bar(C)_T(f,g)=\bigoplus_{\alpha\in C_1(x,x^\prime), \beta\in C_1(y^\prime,y)}
C(f,\beta\circ(f_0\circ \alpha))\otimes_\k (M_T)_2(f_0,g_0)\otimes_\k C(\beta\circ (g_0\circ\alpha),g)/\sim
\end{equation}
where $\sim$ is the relation in \eqref{free2bimodrel}. 

\begin{prop}
In the notations as above, the assignment $T\rightsquigarrow\Bar(C)_T$ is an object part of a functor
$$
\Bar(C)\colon \hat{\Theta}_2\to \Bimod_2(C)
$$
\end{prop}
\begin{remark}{\rm
It will be clear from the proof that the assignment $T\rightsquigarrow M_T$, $\Ob(\hat{\Theta}_2)\to \mathrm{Glob}_2$ can't be extended to a functor $\hat{\Theta}_2\to \mathrm{Glob}_2$. So it becomes a functor only after application of $L(-)$. 
}
\end{remark}

\begin{proof}
If we considered an ordinary (set-enriched) bicategory ${C}$, we would use the fact that the set of strict functors $\Hom_\Bicat^\str(\hat{\omega}_2(T^*),{C})$ is functorial with respect to strict bicategory maps $\hat{\Theta}_2(S,T)=\Hom_\Bicat^\str(\hat{\omega}_2(S^*),\hat{\omega}_2(T^*))$. On the other hand, a strict functor $\hat{\omega}_2(T)\to C$ is the same as a map of 2-globular sets $T^*\to U(C)$, where $U(C)$ denotes the underlying 2-globular object. A map of 2-globular sets $T^*\to U(C)$ is defined by its values on the sets of $i$-cells, $i=0,1,2$. The set of these maps is ``very closed'' to our $M_T$. More precisely, the set of maps of 2-globular sets $T^*\to U(C)$ is 
\begin{equation}
M^\prime_T=\coprod_{\substack{{x_0,\dots,x_k\in C_0}\\
{f_{ij}\in C_1(x_{i-1},x_i), j=0\dots n_i}}}
{\prod_{\substack{{i=1\dots k}\\{j=1\dots n_i}}}} C(f_{i,j-1},f_{i,j})
\end{equation}
(with dropped condition $f_{\circ\min}=f_0, f_{\circ\max}=g_0$). 
\end{proof}
It follows from the discussion just above that the assignment $T\rightsquigarrow M^\prime_T$ gives rise to a functor 
$M^\prime(C)\colon \hat{\Theta}_2^\opp\to \mathscr{V}$. 

Note that if a map $p(\Phi)\in\Theta_2$ is not dominant (that is not necessarily preserves all minima and maxima), the projection along several factors is used to define an action of $\Phi$.

We are mostly interested in the case when the bicategory $C$ is dg $\k$-linear, and in this case the projections of $V\otimes W$ to $V$ and to $W$ are not defined. Moreover, it is not true that, for a $\k$-linear bicategory $C$, 
\begin{equation}
M^\pprime_T=\bigoplus_{\substack{{x_0,\dots,x_k\in C_0}\\
{f_{ij}\in C_1(x_{i-1},x_i), j=0\dots n_i}}}
{\bigotimes_{\substack{{i=1\dots k}\\{j=1\dots n_i}}}} C(f_{i,j-1},f_{i,j})
\end{equation}
is $\Hom_{\Bicat(\k)}^\str(\k\hat{\omega}_2(T^*), C)$, where $\k\hat{\omega}_2(T^*)$ denotes the free $\k$-linear bicategory
generated by $T^*$ made $\k$-linear in dimension 2. Indeed, a homogeneous element of a complex of vector space $V$ is not the same as $C^\udot(\k)(\k,V)$.

On the other hand, the same formulas with the direct product of sets replaced by the tensor product of complexes of $\k$-vector space define an action of $\hat{\Theta}_2^\mathrm{dom}$ on $M^\pprime$, where the upper script dom stands for the subcategory of morphisms $\Phi$ such that $p(\Phi)$ is dominant. 

Now instead of non-existing projections we use the ``from inside'' action in the free 2-bimodule case (it can be considered as a 2-dimensional version of the two-sided bar-construction). It gives an action of general (possibly non-dominant) morphisms. Also, the parenthesising in \eqref{eqbarcorr0} is fixed, so after an application of a morphism in $\hat{\Theta}_2$ one may get a wrong parenthesising in $f_{\circ \min}$ and $f_{\circ\max}$. The same holds when extra identity 1-morphisms are present after
an application of a morphism in $\hat{\Theta}_2$.  In these cases we have to reduce $f_{\circ\min}$ and $f_{\circ\max}$ to the standard parenthesising and without extra identity morphisms, by applying the associator and the unit 2-morphisms ``from inside'' on the (upper and lower) 2-bimodule arguments. The reader easily checks that this description indeed gives rise to a functor $\Bar(C)\colon \hat{\Theta}_2\to \Bimod_2(C)$.

\subsection{\sc The 2-cocellular complex $A(C,D)(F,G)(\eta,\theta)$}\label{sectiondefa2}
Let $C,D$ be a dg bicategories, $F,G\colon C\to D$ be strong functors, $\eta,\theta\colon F\Rightarrow G$ be strong natural transformations. We associate with this data a complex $A(C,D)(F,G)(\eta,\theta)\in C^\udot(\k)$. 

In Section \ref{sectiondefa1}, we associated to a dg bicategory $C$ a functor
$$
\Bar(C)\colon \hat{\Theta}_2^\opp\to \Bimod_2(C)
$$
Let $p\colon \hat{\Theta}_2\to\Theta_2$ be the projection which is the identity on objects. 
The left Kan extension $\Lan_{p^\opp}(\Bar(C))$ is a functor $\Theta_2^\opp\to\Bimod(C)$.

Define
\begin{equation}
A^\Theta(C,D)(F,G)(\eta,\theta)_T=\underline{\Bimod}_2(C)\big((\Lan_{p^\opp}\Bar(C))_T, M(C,D)(F,G)(\eta,\theta)\big)\in C^\udot(\k)
\end{equation}
Here the Hom is taken internally with respect to complexes of vector spaces, so it takes values in $C^\udot(\k)$, the 2-$C$-bimodule $M(C,D)(F,G)(\eta,\theta)$ was defined in Section \ref{section2bimodtheta}. 

For the final object $T_{*}=([0];\varnothing)$, we set $$A^\Theta(C,D)(F,G)(\eta,\theta)_{T_*}=\prod_{X\in \Ob(C)}D_2(\eta_X,
\theta_X)
$$

Finally, define 
\begin{equation}
A(C,D)(F,G)(\eta,\theta)=\Tot_{T\in\Theta_2}A^\Theta(C,D)(F,G)(\eta,\theta)_T
\end{equation}

In Sections \ref{secaexpl1}, \ref{secaexpl2} below we unwind this definition making it more explicit. 

We consider this complex as ``derived modifications'' $\eta\Rrightarrow\theta$ (see Definition \ref{defmodif} for the definition of a classical modification). 

\begin{remark}{\rm
Likewise, for the dg 1-categories $C,D$, and two dg functors $F,G\colon C\to D$, the Hochschild complex 
$$
\Hoch^\udot(C,D(F(-),G(=))=\underline{\Bimod}(C)(\Bar(C), D(F(-),G(=))
$$
is interpreted as ``derived natural transformations'' $F\Rightarrow G$, in the sense that closed degree 0 0-cochains correspond to classical dg natural transformations. Here $\Bar(C)$ is the classical bar-complex of the dg category $C$, taking values in $C$-bimodules.
}
\end{remark}

More precisely, one has:
\begin{prop}
Assume $C,D$ are dg bicategories, $F,G\colon C\to D$ strong functors, $\eta,\theta\colon F\Rightarrow G$ strong dg transformations. Then the vector space $V\subset A^{\Theta}(C,D)(F,G)(\eta,\theta)_{T_*}$ of degree 0 closed elements in $A(C,D)(F,G)(\eta,\theta)$ is isomorphic to the vector space of degree 0 modifications $\eta\Rrightarrow\theta$.
\end{prop}
\begin{proof}
A general element $\Psi$ of $A^{\Theta}(C,D)(F,G)(\eta,\theta)_{T_*}$ belongs to $\prod_{X\in C}D(\eta_X,\theta_X)$. Its boundary belongs to the arity $T_1=([1];[0])$. A general cochain in $A^\Theta(C,D)(F,G)(\eta,\theta)_{T_1}$ is an element of
$$
\prod_{f\colon X\to Y\in C_1}D_2(\eta_Y\circ F(f), G(f)\circ \theta_X)
$$
Let $\Psi$ have the components $\Psi_X\in D_2(\eta_X,\theta_X)$, $X\in C_0$. Then $d\Psi$ has the components $(d\Psi)_f$, $f\in C_1(X,Y)$, which are 
$$
(d\Psi)_f=\Psi_Y\circ F(f)-G(f)\circ \Psi_X
$$
(The two summands come from the two morphisms $T_\varnothing\to T_1$ in $\Theta_2$ corresponded to the two morphisms $[0]\to [1]$ in $\Delta$, they are $D_\min$ and $D_\max$ in notations of Section \ref{facetheta}, (F4)).

Then $(d\Psi)_f=0$ for any $f\in C(X,Y)$ is precisely the condition for $\Psi$ being a modification. 
\end{proof}

\subsection{\sc Some properties of the category of 2-bimodules}
\begin{lemma}\label{lemmabarproj}
Let $C$ be a $\k$-linear bicategory, $\k$ a field. The category $\Bimod_2(C)$ of 2-bimodules over $C$ is abelian $\k$-linear. The 2-bimodules $L(M)$ are projective, where $L$ is the left adjoint to the forgetful functor$U\colon \Bimod_2(C)\to\mathrm{Glob}_2(\k)$, see Section \ref{section2bimod}.
\end{lemma}

\begin{proof}
 Consider the truncation functor $\xi\colon \mathrm{Glob}_2(\k)\to 
\mathrm{Glob}_1$, where $\mathrm{Glob}_2(\k)$ denotes the category of 2-globular objects enriched in $\Vect(\k)$ in degree 2 (so that the degree 0 and the degree 1 components are sets).  It is clear that the comma-category $\xi\setminus Y$  is an abelian $\k$-linear category for any $Y\in\mathrm{Glob}_1$. One can consider the forgetful functor $U$ as a functor $U_\xi\colon\Bimod_2(C)\to\xi\setminus C_{\le 1}$. 

Let $f\colon M\to N$ be a morphism of 2-bimodules over $C$. The kernel and the cokernel of $U_\xi(f)$ have natural structures of $C$-2-bimodules. Then the first statement follows from the fact that the comma-category $\xi\setminus C_{\le 1}$ is abelian.

For the second statement, there is an adjunction
$$
L_\xi\colon \xi\setminus C_{\le 1}\rightleftarrows \Bimod_2(C)\colon U_\xi
$$
where the free 2-bimodule functor $L_\xi$ is the same as the functor $L$ defined in Section \ref{section2bimod}.
The equality
$$
\Bimod_2(C)(L(M), N)=(\xi\setminus C_{\le 1})(M, U_\xi(N))
$$
and the projectivity of $M\in \xi\setminus C_{\le 1}$ proves the second statement.

\end{proof}

\begin{remark}\label{noresolution}{\rm
Consider $B(C)=\mathrm{Real}_{T\in\Theta_2}\Lan_{p^\opp}\Bar(C)_T$, where $\mathrm{Real}$ stands for the $\Theta_2$-realisation. 
It is {\it not} true that $B(C)$ is a resolution of the tautological 2-bimodule $C$ over $C$. Consequently, Lemma \ref{lemmabarproj} does not imply that (*) $A(C,D)(F,G)(\eta,\theta)=\RHom_{\Bimod_2(C)}(C,M(C,D)(F,G)(\eta,\theta)$. As we show now (*) can't be true. 

Indeed, consider the most limit case of a $\k$-linear bicategory, namely, when it has a single object and a single 1-morphism (the identity morphism $\id$ of this object). Such a category is just a commutative algebra $X=C_2(\id,\id)$ over $\k$. Our results of Section \ref{sectionlast} show that $H^3(A(C,C)(\Id,\Id)(\id,\id))$ computes infinitesimal deformations of $X$ as a commutative $\k$-algebra. When $X$ is singular such cohomology may not vanish. On the other hand, the category of 2-bimodules over such $C$ is just the category of left $X$-modules. If (*) was true we would have that this cohomology is $H^3(\RHom_{Mod(X)}(X,X))$, but the latter cohomology vanishes for any $X$. 

It would be interesting to compare, for such $C$, the cohomology of $A(C,C)(\Id,\Id)(\id,\id)$ with the Andre-Quillen cohomology of $X$. 
}
\end{remark}

\subsection{\sc The complex $A(C,D)(F,G)(\eta,\theta)$: an explicit description}\label{secaexpl1}
In this Subsection, we provide a more direct description of $A(C,D)(F,G)(\eta,\theta)$ and of the differential on it, defined in Section \ref{sectiondefa2}, where $C,D$ are $\k$-linear bicategories, $F,G$ strong functors, $\eta,\theta$ strong natural transformations. The reader easily checks that the description given below agrees with the one given in Section \ref{sectiondefa2}. 

We define $A(C,D)(F,G)(\eta,\theta)_T$, for $T=([k]; [n_1],\dots,[n_k])$, as a graded subspace (subcomplex) of $\hat{A}(C,D)(F,G)(\eta,\theta)_T$, the latter is defined as follows. 

For $x,y\in C_0$, $f,g\in C_1(x,y)$, $n\ge 0$, denote
\begin{equation}
C^n_{x,y}(f,g)=\bigoplus_{\substack{{f_0,\dots,f_n\in C_1(x,y)}\\
{f_0=f,f_n=g}}}C_1(f_{n-1},f_n)\otimes_\k C_1(f_{n-2},f_{n-1})\otimes_\k\dots\otimes_\k C_1(f_0,f_1)
\end{equation}
(when $n=0$ we mist have $f=g$). We use notation $\theta$ for an element in $C^n_{x,y}(f,g)$, and then $\theta(i)=f_i$. Also, we use the notation $\theta(\sigma_1,\dots,\sigma_n)$ for such an element, where $\sigma_i\in C_2(f_{i-1},f_i)$.

We will need one more notation. Let $x,y,z,w\in C_0$, $f_i\in C_1(x,y), f_i^\prime\in C_1(y,z), f_i^\pprime\in C_1(z,w)$,
$\sigma_i\in C_2(f_{i-1},f_i), \sigma_i^\prime\in C_1(f_{i-1}^\prime,f_i^\prime), \sigma_i^\pprime\in C_1(f_{i-1}^\pprime,f_i^\pprime)$. Denote $\Sigma_i^-=\sigma_i^\pprime\circ^h( \sigma_i^\prime\circ^h \sigma_i)$,
$\Sigma_i^+=(\sigma_i^\pprime\circ^h \sigma_i^\prime) \circ^h \sigma_i$. Then we denote by $\theta^{(3)}(\Sigma_1,\dots,
\Sigma_k)$ the corresponding element in $C^k_{x,w}$, where each $\Sigma_i$ is either $\Sigma_i^+$ or $\Sigma_i^-$. In particular, we may consider $\theta^{(3)}(\Sigma_1^-,\dots,\alpha\circ^h \Sigma_i^-,\Sigma_{i+1}^+,\dots,\Sigma_k^+)$ where $\alpha$ is the associator.

Then for $T=([k]; [n_1],\dots,[n_k])$, one has
\begin{equation}
\begin{aligned}
\ &\hat{A}(C,D)(F,G)(\eta,\theta)_T=\\
&\prod_{\substack{{x_0,\dots,x_k\in C_0}\\
{f_i,g_i\in C_1(x_{i-1},x_i)}}}\underline{\Hom}_\k\Big(C^{n_1}_{x_0,x_1}(f_1,g_1)\otimes_\k\dots\otimes_\k C^{n_k}_{x_{k-1},x_k}(f_k,g_k),D(\eta_{x_k}\circ f^F_\tot,g_\tot^G\circ\theta_{x_0} )\Big)
\end{aligned}
\end{equation}
where 
$$
f_\tot^F=F(f_k)\circ(F(f_{k-1})\circ (\dots (F(f_2)\circ F(f_1))\dots)), \ g_\tot^G=G(g_k)\circ(G(g_{k-1})\circ (\dots (G(g_2)\circ G(g_1))\dots))
$$

For $T=([0];\varnothing)$ we set
$$
\hat{A}(C,D)(F,G)_{([0];\varnothing)}=\prod_{x\in C_0}D_2(\eta_{x},\theta_{x})
$$

Now $A(C,D)(F,G)(\eta,\theta)_T$ is a subcomplex of $\hat{A}(C,D)(F,G)(\eta,\theta)_T$ formed by the cochains $\Psi\in\hat{A}(C,D)(F,G)(\eta,\theta)$ depicted by the following conditions.

The conditions are divided into two groups: the associator and the unit map conditions. They are direct consequences of Proposition \ref{propkanexttilde}.

The associator conditions read:
\begin{equation}\label{relmain11}
\begin{aligned}
\ & \Psi(\theta_1\otimes\dots\otimes \theta_{j-1}\otimes\theta_j^{(3)}(\Sigma_1^-,\dots,\alpha\circ^v \Sigma_i^-, \Sigma_{i+1}^+,\dots,\Sigma_{n_j}^+)\otimes\theta_{j+1}\otimes\dots\otimes\theta_k)=\\
&\Psi(\theta_1\otimes\dots\otimes \theta_{j-1}\otimes \theta_j^{(3)}(\Sigma_1^-,\dots, \Sigma_i^-,\Sigma_{i+1}^+\circ^v\alpha,  \Sigma_{i+2}^+,\dots,\Sigma_{n_j}^+)\otimes \theta_{j+1}\otimes\dots\otimes \theta_k), 1\le j\le k, 1\le i\le n_j-1\\
&\Psi(\theta_1\otimes\dots\otimes \theta_{j-1}\otimes\theta_j^{(3)}(\Sigma_1^+,\dots,\alpha^{-1}\circ^v \Sigma_i^+, \Sigma_{i+1}^-,\dots,\Sigma_k^-)\otimes\theta_{j+1}\otimes\dots\otimes\theta_k)=\\
&\Psi(\theta_1\otimes\dots\otimes \theta_{j-1}\otimes \theta_j^{(3)}(\Sigma_1^+,\dots, \Sigma_i^+,\Sigma_{i+1}^-\circ^v\alpha^{-1},  \Sigma_{i+2}^-,\dots,\Sigma_k^-)\otimes \theta_{j+1}\otimes\dots\otimes \theta_k), 1\le j\le k, 1\le i\le n_j-1\\
&\Psi(\theta_1\otimes\dots\otimes\theta_j^{(3)}(\Sigma_1^+\circ^v\alpha,\Sigma_2^+,\dots,\Sigma_{n_j}^+)\otimes\dots\otimes\theta_k)=\\
&\Psi(\theta_1\otimes\dots\otimes\theta^{(3)}_j(\Sigma_1^+,\dots,\Sigma_{n_j}^+)\otimes\dots\otimes\theta_k)\circ^v \tilde{\alpha}_j, 1\le j\le k\\
&\Psi(\theta_1\otimes\dots\otimes\theta^{(3)}_j(\Sigma_1^-,\dots, \Sigma_{n_j-1}^-,\alpha\circ^v \Sigma_{n_j}^-)\otimes\dots\otimes\theta_k)=\\
&\tilde{\alpha}_j\circ^v\Psi(\theta_1\otimes\dots\otimes \theta_j^{(3)}(\Sigma_1^-,\dots,\Sigma_{n_j}^-)\otimes\dots\otimes\theta_k), 1\le j\le k\\
&\Psi(\theta_1\otimes\dots\otimes\theta_j^{(3)}(\Sigma_1^-\circ^v\alpha^{-1}, \Sigma_2^-,\dots,\Sigma_{n_j}^-)\otimes\dots\otimes \theta_k)=\\
&\Psi(\theta_1\otimes\dots\otimes\theta_j^{(3)}(\Sigma_1^-, \Sigma_2^-,\dots,\Sigma_{n_j}^-)\otimes\dots\otimes \theta_k)\circ^v\tilde{\alpha}^{-1}_j, 1\le j\le k\\
&\Psi(\theta_1\otimes\dots\otimes\theta_j^{(3)}(\Sigma_1^+,\dots, \alpha^{-1}\circ^v \Sigma_{n_j}^+)\otimes\dots\otimes\theta_k)=\\
&\tilde{\alpha}^{-1}_j\circ^v\Psi(\theta_1\otimes\dots\otimes\theta_j^{(3)}(\Sigma_1^+,\dots,\Sigma_{n_j}^+)\otimes\dots\otimes\theta_k), 1\le j\le k
\end{aligned}
\end{equation}
where $\alpha$ is the associator 2-morphism, and $\tilde{\alpha}_j$ (resp., $\tilde{\alpha}^{-1}_j$) stand for the wiskering of $\alpha$ (resp., of $\alpha^{-1}$) acting on $j$-th output with suitable identity 2-morphisms. 

When $n_j=0$ for some $j$, the list above for this $j$ gives (and is reduced to) the following equations:
\begin{equation}\label{relmain22}
\begin{aligned}
\ & \tilde{\alpha}_j\circ^v\Psi(\theta_1\otimes\dots\otimes \theta^{(3)}(\id_{f^\pprime\circ (f^\prime\circ f)})\otimes\dots\otimes\theta_k)=\Psi(\theta_1\otimes\dots\otimes\theta^{(3)}(\id_{(f^\pprime\circ f\prime)\circ f})\otimes\dots\otimes\theta_k)\circ^v \tilde{\alpha}_j\\
&\tilde{\alpha}_j^{-1}\circ^v \Psi(\theta_1\otimes\dots\otimes \theta_j^{(3)}(\id_{(f^\pprime\circ f^\prime)\circ f})\otimes\dots\otimes\theta_k)=
\Psi(\theta_1\otimes\dots\otimes \theta_j^{(3)}(\id_{f^\pprime\circ (f\prime\circ f)})\otimes\dots\otimes\theta_k)\circ^v\tilde{\alpha}_j^{-1}
\end{aligned}
\end{equation}
(compare with discussion in Section \ref{sectiondygen}).
\\

The unit map conditions read:

\begin{equation}\label{relmain33}
\begin{aligned}
\ &\Psi(\theta_1\otimes\dots\otimes\theta_j(\sigma_1,\dots,\lambda^{-1} \circ^v\sigma_i,\id\circ^h \sigma_{i+1},\dots,\id\circ^h\sigma_{n_j})\otimes\dots\otimes\theta_k)=\\
&\Psi(\theta_1\otimes\dots\otimes\theta_j(\sigma_1,\dots,\sigma_i,(\id\circ^h\sigma_{i+1})\circ^v \lambda^{-1},\id\circ\sigma_{i+1},\dots,\id\circ \sigma_{n_j})\otimes\dots\otimes\theta_k), 1\le j\le k\\
&\Psi(\theta_1\otimes\dots\otimes\theta_j(\id\circ^h\sigma_1,\dots,\id\circ^h\sigma_{i-1}   ,\sigma_i\circ^v\lambda,\sigma_{i+1},\dots,\sigma_{n_j})=\\
&\Psi(\theta_1\otimes\dots\otimes\theta_j(\id\circ^h\sigma_1,\dots, \lambda \circ^v (\id\circ^h\sigma_{i-1}),\sigma_i,\dots,\sigma_{n_j})\otimes\theta_{j+1}\otimes\dots\otimes\theta_k), 1\le j\le k\\
&\Psi(\theta_1\otimes\dots\otimes\theta_j(\sigma_1\circ^v\lambda,\sigma_2,\dots,\sigma_{n_j})\otimes\dots\otimes\theta_k)=\\
&\Psi(\theta_1\otimes \dots\otimes\theta_j(\sigma_1,\dots,\sigma_{n_j})\otimes\theta_{j+1}\otimes\dots\otimes\theta_k)\circ^v\tilde{\lambda}_j, 1\le j\le k\\
&\Psi(\theta_1\otimes\dots\otimes\theta_j((\id\circ^h\sigma_1)\circ^v\lambda^{-1},\id\circ^h\sigma_2,\dots,\id\circ^h\sigma_{n_j})\otimes\dots\otimes\theta_k)=\\
&\Psi(\theta_1\otimes \dots\otimes\theta_j(\id\circ^h \sigma_1,\dots,\id\circ^h \sigma_{n_j})\otimes\theta_{j+1}\otimes\dots\otimes\theta_k)\circ^v\tilde{\lambda}_j^{-1}, 1\le j\le k\\
&\Psi(\theta_1\otimes\dots\otimes\theta_j(\id\circ^h \sigma_1,\dots,\id\circ^h\sigma_{n_j-1}, \lambda\circ^v(\id\circ^h \sigma_{n_j})\otimes\dots\otimes\theta_k)=\\
&\tilde{\lambda}_j\circ^v \Psi(\theta_1\otimes \dots\otimes\theta_j(\id\circ^h \sigma_1,\dots,\id\circ^h \sigma_{n_j})\otimes\theta_{j+1}\otimes\dots\otimes\theta_k), 1\le j\le k\\
&\Psi(\theta_1\otimes\dots\otimes \theta_j(\sigma_1,\dots,\sigma_{n_j-1},\lambda^{-1}\circ^v \sigma_{n_j})\otimes\dots\otimes\theta_k)=\\
&\tilde{\lambda}_j^{-1} \circ^v\Psi(\theta_1\otimes\dots\otimes\theta_j(\sigma_1,\dots,\sigma_{n_j})\otimes \dots\otimes\theta_k),
1\le j\le k
\end{aligned}
\end{equation}
where $\lambda_f\colon\id\circ^h f\to f$ is the left unit map, and $\tilde{\lambda}_j$ (resp., $\tilde{\lambda}^{-1}_j$) stands for  whiskering of $\lambda$ acting on the $j$-th factor of the output with the suitable identity maps. 

As well, one has similar relations for the right unit map $\rho$, which we don't write down here. 
\\

The next point is to extend the assignment $T\rightsquigarrow A(C,D)(F,G)(\eta,\theta)$ to a functor  
$$A(C,D)(F,G)(\eta,\theta)\colon \Theta_2\to C^\udot(\k)$$
Once again, existence of such functor  follows from the construction given in Section \ref{sectiondefa2}, and our task here is to provide explicit formulas for the action of morphisms of $\Theta_2$.

We restrict ourselves to the case when $C$ and $D$ are {\it strict} 2-categories,  $F,G$ {\it strict} functors, and $\eta,\theta$ {\it strict} natural transformations. The reason is that even this case  shows the nature of the aforementioned action, but essentially simplifies the formulas. The formulas for the $\Theta_2$-action in the general case differ by numerous conjunctions with the structure 2-isomorphisms.

\subsection{\sc  An explicit description of the complex $A(C,D)(F,G)(\eta,\theta)$, II}\label{secaexpl2}
\subsubsection{}\label{subsecaexpl1}
Let $C$ be a dg bicategory, $x,y\in C_0$, and let $\phi\colon [m]\to [n]$ be a morphism in $\Delta$. We associate with $\phi$ a map of complexes 
\begin{equation}
C_{x,y}(\phi)\colon C_{x,y}^n(f_0,f_n)\to \underset{A_L(\phi)}{C_2(f_{\phi(m)}, f_n)}\otimes_\k \underset{M(\phi)}{C^m_{x,y}(f_{\phi(0)},f_{\phi(m)})}\otimes_\k \underset{A_R(\phi)}{C_2(f_0,f_{\phi(0)})}
\end{equation}
as follows.

We use notation $\sigma_n\otimes\dots\otimes\sigma_1$ for an element in $C_{x,y}^n(f_0,f_n)$, where $\sigma_i\in C_2(f_{i-1},f_i)$ (a general element in $C^n_{x,y}(f_0,f_n)$ is a linear combination of such elements). 

The two ``extreme'' factors $A_L(\phi)$ and $A_R(\phi)$ are defined as the compositions
$$
A_L(\phi)=\sigma_n\circ^v\dots\circ^v\sigma_{\phi(m)+1},\ \ A_R(\phi)=\sigma_{\phi(0)}\circ^v\dots\circ^v \sigma_1
$$
($A_L(\phi)$ is by definition equal to $\id(f_n)$ if $\phi(m)=n$, and $A_R(\phi)$ is $\id(f_0)$ if $\phi(0)=0$). 

The middle factor $$
M(\phi)(\sigma_{\phi(m)}\otimes\dots\sigma_{\phi(0)+1})=\omega_m\otimes\dots\otimes\omega_1\in C_{x,y}^m(f_{\phi(0)},f_{\phi(m)})
$$ 
is defined by
\begin{equation}
\omega_a=\omega_a(\phi)=
\begin{cases}
\sigma_{c-1}\circ^v\dots\circ^v\sigma_b\colon f_b\to f_c&\text{ if }\phi(a-1)=b,\phi(a)=c, c>b\\
\id(f_b)&\text{ if }\phi(a-1)=\phi(a)=b
\end{cases}
\end{equation}
It completes the definition of $C_{x,y}(\phi)$.
\begin{remark}{\rm
The reader certainly realises that this construction just mimics the classical nerve construction in a non-cartesian monoidal ($\k$-linear) case. In the classical case of enrichment in $\Sets$ one just projects along the two extreme factors $A_L(\phi)$ and $A_R(\phi)$. Our ``bimodule'' version is a way to phrase out the same construction in the situation when the projections (with respect to the monoidal product) do not exist. 
}
\end{remark}

\subsubsection{\sc }\label{subsecaexpl2}
Let $C$ be as above, and let $x_0,\dots,x_k\in C_0$. Assume we are given maps in $\Delta$
$\phi_1\colon [m]\to[n_1],\dots, \phi_k\colon [m]\to [n_k]$. 
For $f_{i0},\dots, f_{in_i}\in C_1(x_{i-1}, x_i)$, $1\le i\le k$, define
$$
f_\min=f_{\otimes 0}=f_{k0}\circ f_{k-1,0}\circ\dots\circ f_{10},\ \ f_\max=f_{\otimes n}=f_{kn_k}\circ\dots\circ f_{1n_1}
$$
and, for $0\le s\le m$,
$$
f_{\otimes \phi(s)}=f_{k\phi_k(s)}\circ\dots\circ f_{1\phi_1(s)}
$$
We define a map generalising the map defined above when we had $k=1$:
\begin{equation}\label{mapeq1}
\begin{aligned}
\ &C_{x_0,\dots,x_k}(\phi_1,\dots,\phi_k)\colon C^{n_1}_{x_0,x_1}(f_{10},f_{1n_1})\otimes_\k\dots\otimes_kC_{x_{k-1},x_k}^{n_k}(f_{k0},f_{kn_k})\to \\
&\underset{A_L(\phi_1,\dots,\phi_k)}{C_2(f_{\otimes\phi(m)},f_{\otimes n})}\otimes_\k \underset{B(\phi_1,\dots,\phi_k)}{C^m_{x_{0},x_{k}}(f_{\otimes \phi(0)},f_{\otimes \phi(m)})}\otimes_\k
\underset{A_R(\phi_1,\dots,\phi_k)}{C_2(f_{\otimes 0},f_{\otimes \phi(0)})}
\end{aligned}
\end{equation}
where

\begin{equation}\label{mapeq2}
\begin{aligned}
\ & A_L(\phi_1,\dots,\phi_k)=A_L(\phi_k)\circ^h A_L(\phi_{k-1})\circ^h\dots \circ^h A_L(\phi_1)\in C_2(f_{\otimes\phi(m)},f_{\otimes n})\\
&A_R(\phi_1,\dots,\phi_k)=A_R(\phi_k)\circ^h\dots\circ^h A_R(\phi_1)\in C_2(f_{\otimes 0},f_{\otimes \phi(0)})\\
&B(\phi_1,\dots,\phi_k)=\Omega_1\otimes_\k\dots\otimes_\k\Omega_m\in C_2(f_{\otimes \phi(0)},f_{\otimes \phi(m)}) \text{  where  }\\
&\Omega_i=\omega_i(\phi_k)\circ^h \omega_i(\phi_{k-1})\circ^h\dots\circ^h\omega_i(\phi_1)\in C_2(f_{\otimes \phi(i-1)},f_{\otimes \phi(i)})
\end{aligned}
\end{equation}

\subsubsection{}\label{subsecaexpl3}
Let $T=([k];[m_1],\dots,[m_k]), S=([\ell]; [n_1],\dots,[n_\ell])$, and let 
$\Phi=(\phi; \{\phi^{i,\ell}\})\colon T\to S$ a morphism in $\Theta_2$. Here we construct a map of complexes
$$
\Phi_*\colon A(C,D)(F,G)(\eta,\theta)_T\to A(C,D)(F,G)(\eta,\theta)_{S}
$$
Let $\Psi\in A(C,D)(F,G)(\eta,\theta)_T$, and let 
$x_0,\dots,x_{\ell}\in C_0$, $\{f_{ij}\in C_1(x_{i-1}, x_i)\}_{i=1\dots \ell, j=0\dots n_i}$,
and
$$
\{\sigma_{ij}\in C_2(f_{i,j-1},f_{i,j})\}_{i=1\dots \ell, j=1\dots n_i}
$$
be a datum ``of shape $S$''. 

We have to define 
$$
\Phi_*(\Psi)(\{\sigma_{ij}\})\in D_2\Big(\eta_{x_{\ell}}\circ F(f_{\ell 0})\circ \dots\circ F(f_{10}), G(f_{\ell n_\ell})\circ\dots\circ G(f_{1n_1})\circ \theta_{x_0}\Big)
$$
Let $\min=\phi(0), \max=\phi(k)$. 

For each $0\le i\le k-1$, one gets a sequence of maps $\{\phi^{i,j}\colon [m_i]\to [n_j]\}_{\phi(i-1)+1\le j\le \phi(i)}$.
Assume $\phi(i-1)<\phi(i)$, then the construction of Section \ref{subsecaexpl2} (\eqref{mapeq1} and \eqref{mapeq2}) gives a map 
\begin{equation}
\begin{aligned}
\ &C^{n_{\phi(i-1)+1}}_{x_{\phi(i-1)},x_{\phi(i-1)+1}}(f_{\phi(i-1)+1,0},f_{\phi(i-1)+1,n_{\phi(i-1)+1}})\otimes_\k\dots
\otimes_\k C^{n_{\phi(i)}}_{x_{\phi(i)-1},x_{\phi(i)}}(f_{\phi(i),0}, f_{\phi(i),n_{\phi(i)}})\to\\
&\underset{A_L^i}{C_2(f_{\otimes \phi_i(m_i)}, f_{\otimes i, \max}) }\otimes_\k \underset{B^i}{C^{m_i}_{x_{\phi(i-1)},x_{\phi(i)}}(f_{\otimes \phi_i(0)}, f_{\otimes\phi_i(m_i)})}\otimes_\k \underset{A_R^i}{C_2(f_{\otimes i, 0},f_{\otimes \phi_i(0)})}
\end{aligned}
\end{equation}
where
\begin{equation}
f_{\otimes \phi_i(s)}=f_{\phi(i),\phi_i^{\phi(i)}(s)}\circ \dots\circ f_{\phi(i-1)+1,\phi_i^{\phi(i-1)+1}(s)}
\end{equation}
where $0\le s\le m_i$,
and 
\begin{equation}
\begin{aligned}
\ & f_{\otimes i,0}=f_{\phi(i),0}\circ f_{\phi(i)-1,0}\circ\dots f_{\phi(i-1)+1,0}\\
&f_{\otimes i,\max}=f_{\phi(i), n_{\phi(i)}}\circ f_{\phi(i)-1, n_{\phi(i)-1}}\circ\dots\circ f_{\phi(i-1)+1,n_{\phi(i-1)+1}}
\end{aligned}
\end{equation}

For the case when $\phi(i-1)=\phi(i)$, we set 
$$
A_L^i=\id\in C_2(\id_{x_{\phi(i)}} ,\id_{x_{\phi(i)}}), A_R^i=\id\in C_2(\id_{x_{\phi(i)}} ,\id_{x_{\phi(i)}}), B^i=\id_{x_{\phi(i)}}\xrightarrow{\id}\id_{x_{\phi(i)}}\xrightarrow{\id}\dots\xrightarrow{\id}\id_{x_{\phi(i)}}
$$
where $B^i$ is the chain with $m_i$ arrows. 

Define
$$
G(A_L^{\otimes})=G(A_L^k)\circ^h G(A_L^{k-1})\circ^h \dots \circ^h G(A_1^1),\ \ F(A_R^{\otimes})=F(A_R^k)\circ^h F(A_R^{k-1})\circ^h\dots\circ^h F(A_R^1)
$$
For a string
$$
f_{j0}\xrightarrow{\sigma{j1}}f_{j1}\xrightarrow{\sigma{j2}}f_{j2}\dots\xrightarrow{\sigma_{jn_j}}f_{jn_j}
$$
($f_{js}\in C_1(x_{j-1},x_j)$ for $0\le s\le n_j$) denote by 
$$
\sigma_{j,\tot}=\sigma_{jn_j}\circ^v\dots \circ^v\sigma_{j0}
$$
Let $\min=\phi(0), \max=\phi(k)$,
denote
$$
\sigma_{\otimes\min}=\sigma_{\min,\tot}\circ^h\sigma_{\min-1,\tot}\circ^h\dots\circ^h\sigma_{1,\tot}\circ^h\sigma_{0,\tot}
$$
$$
\sigma_{\otimes (\max+1)}=\sigma_{\ell,\tot}\circ^h\sigma_{\ell-1,\tot}\circ^h\dots\circ^h\sigma_{\max+1,\tot}
$$

Finally, we have:
\begin{equation}\label{eqaexplfinal}
\Phi_*(\Psi)=G(\sigma_{\otimes (\max+1)})\circ^h \Big((G(A_L^\otimes))\circ^v \Psi(B^1\otimes_\k\dots\otimes_\k B^k)\circ^v (F(A_R^\otimes))\Big)\circ^h F(\sigma_{\otimes\min})
\end{equation}

Formula \eqref{eqaexplfinal} is an explicit expression for the definition of $A(C,D)(F,G)(\eta,\theta)$, in particular case when $C,D,F,G,\eta,\theta$ are strict. The general case differs by numerous insertions of structure 2-isomorphisms,which makes them more complicated and tedious, but can be written down in a similar way.

Note that it follows from the discussion in Sections \ref{sectionthetahat}, \ref{sectiondefa1}, \ref{sectiondefa2} that the prescription \eqref{eqaexplfinal} gives rise to a functor $\Theta_2\to C^\udot(\k)$.

\subsection{\sc An example: a shuffle permutation}\label{sectionshuffle}
Consider in more detail the action of an inner face map $D_{j,\sigma}$ (F2) (see Section \ref{facetheta}) on $A(C,D)(F,G)(\eta,\theta)$, corresponded to an $(m_j,m_{j+1})$-shuffle permutation $\sigma$. 

Let $t\in \Sigma_{\ell_j}$ be an $(m_j,m_{j+1})$-shuffle, $\ell_j=m_j+m_{j+1}$. Let $p^*\colon [\ell_j]\to [m_j]$ and $q^*\colon [\ell_j]\to [m_{j+1}]$ be the two maps Joyal dual to the natural embeddings $[m_j-1]\to [\ell_j-1]$ and $[m_{j+1}-1]\to [\ell_j-1]$ (see Section \ref{facetheta}(F2)). 
Let $T=([n]; [\ell_1],\dots,[\ell_n])$, $S=([n+1]; [\ell_1],\dots, [\ell_{j-1}], [m_j], [m_{j+1}],[\ell_{j+1}], \dots, [\ell_{n+1}])$. 
Consider the morphism $\Phi=D_{j,t}\colon T\to S$ in $\Theta_2$, corresponded to the shuffle $t$.

Then the morphism $D_{j,t}$ acts on $A(C,D)(F,G)(\eta,\theta)$ as follows.

Use notation
$\underline{f}_s$ for a chain of 2-morphisms $\{\sigma_{si}\colon f_{s,i-1}\to f_{si}\}$ in $C$:
$$
f_{s0}\xrightarrow{\sigma_{s1}}f_{s2}\xrightarrow{\sigma_{s2}}\dots\xrightarrow{\sigma_{s,m_s}}f_{s,m_s}
$$
For a cochain $\Psi\in A(C,D)(F,G)(\eta,\theta)_T$, one has:

\begin{equation}
({D_{j,t}}_*(\Psi))(\underline{f}_1,\dots, \underline{f}_{n+1})=
\Psi_T(\underline{f}_1,\dots,\underline{f}_{j-1}, \underline{g}_j, \underline{f}_{j+2},\dots,\underline{f}_{n+1})
\end{equation}
where $\underline{g}_j$ is the chain 
$$
 f_{j+1,0}\circ f_{j0}\xrightarrow{\hspace{3mm}\omega_1\hspace{3mm}}\dots\dots\xrightarrow{\omega_{m_j+m_{j+1}}} f_{j+1,m_{j+1}}\circ  f_{j,m_j}
$$
and 
\begin{equation}
\omega_i=\begin{cases}
\id \circ^h \sigma_{ja} & \text{if  }t^{-1}(i)=a, \ 0\le a\le m_j\\
\sigma_{j+1,b}\circ^h \id& \text{if  }t^{-1}(i)=b, \ m_j+1\le b\le m_j+m_{j+1}
\end{cases}
\end{equation}

\subsection{\sc Normalized vs non-normalized chain complexes of a 2-cellular object in $C^\udot(\k)$}
In Section \label{sectionlast}, in the proof Theorem \ref{infdefcat} we use that the $\Theta_2$-cochain complex of $A(C,D)(F,G)(\eta,\theta)$ is quasi-isomorphic to its {\it normalized} subcomplex $A(C,D)(F,G)(\eta,\theta)_\norm(C,D)(F,G)(\eta,\theta)$. The latter is, by definition, the sub-complex which consists of all cochains $\Psi$ which are equal to 0 if some of its 2-morphism arguments $\sigma_{i,j}$ is the identity morphism of a 1-morphism. 

In Proposition \ref{propnormtheta}, we prove that the complexes $A(C,D)(F,G)(\eta,\theta)$ and $A_\norm(C,D)(F,G)(\eta,\theta)$ are quasi-isomorphic. Theorem \ref{infdefcat} is a statement about cohomology. Therefore, due to Proposition \ref{propnormtheta}, one can assume in its proof that we work with the normalized complex.

Recall that for a simplicial object in an abelian category $\mathscr{A}$ its normalized Moore complex $N(X)$ is defined as the quotient-complex of the ordinary Moore complex $C(X)$ by the subcomplex $DC(X)$ spanned by elements of the form $s_iy$ (here $s_i$ stands for the simplicial version of the degeneracy morphisms  $\varepsilon_i\in \Delta$, see Section \ref{nord}).

Recall the following classical result, in a slightly more general version:
\begin{prop}\label{propml}
Let $X\colon \Delta^\opp\to C^\udot(\k)$ be a simplicial object in $C^\udot(\k)$. Then the total sum complex $\Tot^\oplus(C(X))$ of the Moore complex of $X$ is quasi-isomorphic to the total sum complex $\Tot^\oplus(N(X))$ of the normalized Moore complex.
\end{prop}
\begin{proof}
The proof given in [ML1], Section VIII.6, can be easily adopted to this case. Indeed, MacLane constructs a map $g\colon C(Y)/DC(Y)\to C(Y)$, for $Y$ a simplicial object in an abelian category, such that $g$ is a ``quasi-inverse'' to the natural projection $\pi\colon C(Y)\to C(Y)/DC(Y)$ in the sense that $\pi\circ g=\id$, and $g\circ \pi$ is chain homotopic to the identity. 
The chain homotopy constructed in loc. cit. clearly commutes with ``inner'' differentials on $X_i$s. 
Consequently, if one defines $\pi^\prime \colon \Tot^{\oplus}(C(X))\to \Tot^\oplus(N(X))$ and $g^\prime\colon \Tot^\oplus(N(X))\to \Tot^\oplus(C(X))$ one still has $\pi^\prime g^\prime=\id$ and $g^\prime \pi^\prime$ chain homotopic to the identity. 
\end{proof}

The next step is to generalise Proposition \ref{propml} to the case of 2-cellular objects in $C^\udot(\k)$, that is, to the case of functors $X\colon \Theta_2^\opp\to C^\udot(\k)$. 

For $Y\colon \Theta_2^\opp\to \Vect(\k)$, its chain complex is defined as the complex $C(Y)$, with 
$$
C_{-\ell}(Y)=\oplus_{T, \dim T=\ell} Y_T
$$
with the differential dual to \eqref{totd}, and its normalized complex is defined as the quotient-complex of $C(Y)$ by the subcomplex $DC(Y)$ generated by the elements $\varepsilon_p^j(y)$ of type (D1) (see Section \ref{facetheta}), $y\in Y_D$:
$$
N(Y)=C(Y)/DC(Y)
$$
That is, we use only ``vertical'' degeneracy morphisms of type (D1), {\it not} ``horizontal'' degeneracy morphisms of type (D2), in the definition of $DC(Y)$. 

For the case of a functor $X\colon \Theta_2^\opp\to C^\udot(\k)$ as above, $C(X), DC(X), N(X)$ are defined as $\Tot^\oplus(C(X)), \Tot^\oplus(DC(X)), \Tot^\oplus(N(X))$, correspondingly. 

\begin{prop}\label{propnormtheta}
Let $X\colon \Theta_2^\opp\to C^\udot(\k)$ be a 2-cellular complex. Then the natural projection 
$\pi\colon \Tot^\oplus(C(X))\to \Tot^\oplus(N(X))$ is a quasi-isomorphism of complexes.
\end{prop}
\begin{proof}
One can not follow directly the same line as in the proof of [ML1], Ch. VIII, Th.6.1, by the following reason. The subspaces $D_iC(X)$, $i\ge 0$ (or rather their direct analogues) are {\it not} subcomplexes of $C(X)$, because the components $D_{j,\sigma}$ of type (F2) (see Section \ref{facetheta}) in the differential \eqref{totd} may {\it increase} $i$. Indeed, these components act as ``deshuffling'' of two neighbour columns, resulting in a column of a greater length, so this operation may send $\varepsilon_p^iy$
to $\varepsilon_q^{i^\prime}(y^\prime)$ with $i^\prime>i$ (here $q=p$ or $p-1$). 

To overcome this obstacle, we employ the following spectral sequence argument.

Denote by $F_N\subset C(X)$ the subspace spanned by $X_T$, $T=([n];[\ell_1],\dots,[\ell_n])$ with $n\le N$. Then $F_N$ is a subcomplex: the boundary operators of type (F1) and (F3) preserve $n$, and the boundary operators of types (F2) and (F4) decrease $n$ by 1, see Section \eqref{facetheta}. 

We get an exhausting ascending filtration of $C(X)$ by subcomplexes:
$$
F_0\subset F_1\subset F_2\subset \dots
$$

A similar filtration exists for $N(X)$ as well, denote the corresponding subspaces by $F_N^\prime$. The natural projection $\pi\colon C(X)\to N(X)$ sends $F_N$ to $F_N^\prime$, hence $\pi$ induces a map of the corresponding spectral sequences. Denote these spectral sequences by $\{E_n^{pq}\}$ and $\{E_n^{\prime pq}\}$, so that $\pi$ induces a map $\pi_*\colon (E_n^{pq}, d_n)\to (E_n^{\prime pq}, d_n^\prime)$. 

The spectral sequences at the term $E_0$ (resp., $E_0^\prime$) are non-zero at the lower half plane $y\le 0$, the differential $d_0$ is horizontal. So the spectral sequences converge by dimensional reasons. 

\begin{lemma}
The map $\pi_*\colon (E_0^{\udot,\ell}, d_0)\to (E_0^{\prime \udot,\ell}, d_0^\prime)$ is a quasi-isomorphism, for any $\ell\le 0$. In particular, $\pi_*$ defines an isomorphism $\pi_*\colon E_1^{pq}\to E_1^{\prime pq}$, for all $p,q$. 
\end{lemma}
\begin{proof}
For any fixed $\ell$, the complex $(E_0^{\udot,\ell}, d_0)$ is $C^{(\ell)}(X)$, whose degree $-n$ component is equal to the direct sum $\oplus_T X_T$ over $T=([\ell]; [n_1],\dots,[n_\ell])$ with $\dim T=n$, and with the differential components given only by (F1) and (F3) types, see \eqref{facetheta}. That is, the contribution of types (F2) and (F4) components in \eqref{totd} becomes 0 in the associated graded complex $C^{(\ell}(X)=F_{\ell}/F_{\ell-1}$. The complex $(E_0^{\prime \udot,\ell}, d_0^\prime)$ has a similar description.

It makes us possible to employ the construction of the proof of  [ML1], Ch.VIII, Th.6.1. Namely, we define {\it subcomplexes} $D_iC^{(\ell)}(X)$, for any $i\ge 0$, such that $D_{i+1} C^{(\ell)}(X)\supset D_i C^{(\ell)}(X)$ and $DC^{(\ell)}(X)=\cup_{i\ge 0}D_iC^{(\ell)}(X)$. As in loc.cit., we construct a map $h_i\colon C^{(\ell)}(X)\to C^{(\ell)}(X)$ chain homotopic to $\id$ and mapping $D_i$ to $D_{i-1}$. The composition of these maps is well-defined, is chain homotopic to $\id$, and sends $DC^{(\ell)}(X)$ to 0. It gives a map $g\colon N^{(\ell)}(X)\to C^{(\ell)}(X)$ such that $\pi_* g=\id$ and $g\pi_*$ is chain homotopic to $\id$, which completes the proof. 

\end{proof}

It follows from this Lemma that $\pi_*$ defines an isomorphism at $E_\infty$ sheet, hence $\pi$ is a quasi-isomorphism.

\end{proof}

We can prove 
\begin{theorem}\label{theornorm}
The natural embedding
$$
i\colon A_\norm(C,D)(F,G)(\eta,\theta)\to A(C,D)(F,G)(\eta,\theta)
$$
is a quasi-isomorphism of complexes. 
\end{theorem}
\begin{proof}
One can apply the arguments ``dual'' to the ones provided above, for the case of a 2-cocellular complex. 
Define $\Phi^p\subset A(C,D)(F,G)(\eta,\theta)=A$ be the subcomplex formed by cochains which vanish on $F_p$ defined above. Then $\{\Phi^p\}$ form a descending filtration of $A$. This filtration is complete in the sense that 
$$
A=\underset{\leftarrow}{\lim}A/\Phi^p A
$$
The term $E_0$ lives in $x\ge 0$ half of the plane, thus it is bounded below. It follows from the complete convergence theorem [W], Th.5.5.10 (in its version when $d$ has degree +1) that the sequence converges. 
A similar convergent spectral sequence exists on $A_\norm$, given by the filtration $\Phi^{p^\prime}A_\norm =(\Phi^p A)\cap A_\norm$. The argument given in the 2-cellular case gives a quasi-isomorphism of complexes $\Phi^{p^\prime}A_\norm/\Phi^{(p+1)\prime}A_\norm\to \Phi^p A/\Phi^{p+1}A$ induced by $i$. Then the result follows from  follows from the complete comparison theorem [W], Th. 5.5.11.
\end{proof}

\section{\sc The $p$-relative totalization $Rp_*(A(C,D)(F,G)(\eta,\theta))$ and higher structures via Davydov-Batanin}
We know from Propositions \ref{propdeltaaction} and \ref{prop!} that the relative totalization of the cosimplicial vector space 
 $Rp_*(A(C,D)(F,G)(\eta,\theta))$ is a cosimplicial vector space, 
and its $\Delta$-Moore complex is equal to the (absolute) $\Theta_2$-totalization of 
such that its non-normalized  Moore complex is isomorphic to the (absolute) $\Theta_2$-totalization of $A(C,D)(F,G)(\eta,\theta)$.

In this section we apply some results of [BD] for studying the higher structures on the complexes $C^\udot(C,D)(F,G)(\eta,\theta)$.  We show that $(R p_*)(X)$ enjoys, for the case $X=A(C,D)(F,F)(\id,\id)$, the property of being a {\it 1-commutative} cosimplicial monoid, in the sense of [BD]. Consequently, $C^\udot(C,D)(F,F)(\id,\id)$ is a homotopy 2-algebra, for any $\k$-linear strict 2-functor $F\colon C\to D$. 

At the same time, for the case $X=A(C,C)(\Id,\Id)(\id,\id)$, the cosimplicial monoid $(R p_*)(X)$ is {\it not} 2-commutative (unlike the case of the Davydov-Yetter complex)). At the moment, we don't know the correct homotopy refinement of $2$-commutativity, which would imply that $C^\udot(C,D)(\Id,\Id)(\id,\id)$ is a homotopy 3-algebra. 

\comment
Proposition \ref{propnormtheta} we prove that for $A(C,D)(F,G)(\eta,\theta)$ the non-normalized Moore complex for $\Theta_2$ is quasi-isomorphic to its normalized subcomplex, by a refinement of the classical argument for $\Delta$. The normalized complex is the one we use in Section 4 in the proof Theorem \ref{infdefcat}, and Proposition \ref{propnormtheta} guarantees that both normalized and non-normalized complexes have isomorphic cohomology.
\endcomment

\subsection{\sc The totalization $\Tot_{\Theta_2}A(C,D)(F,F)(\id,\id)$ is a homotopy 2-algebra}\label{sectionn=2}
Recall a {\it cosimplicial monoid} $X$ (in a symmetric monoidal category $\mathscr{C}$) is a cosimplicial object in the category of monoids $\Mon(\mathscr{C})$. The question raised in [BD] is the following: 

{\it Which condition on $X$ implies that the totalization $\Tot(X)$ admits an action of an operad (homotopy equivalent to) $E_n$}?

It follows immediately that the condition that $X$ is a cosimplicial monoid implies that $X^\udot$ is a monoid with respect to the Batanin $\square$-product [Ba2]. Thus, it follows from loc.cit. that for a cosimplicial monoid $X^\udot$, the totalization $\Tot(X)$ is an $A_\infty$ monoid, that is, a $E_1$-algebra. 

In [BD, Section 2.2], the following definition is given:
\begin{defn}\label{linkingdef1}
{\rm
Let $\tau\colon [p]\to[m]$ and $\pi\colon [q]\to [m]$ be two maps in $\Delta$. A {\it shuffling of lenth $n$} of $\tau,\pi$ is a decomposition of the images of $\tau$ and $\pi$ into disjoint union of {\it connected intervals}
\begin{equation}
\begin{aligned}
\ &\mathrm{Im}(\tau)=A_1\cup A_2\cup\dots \cup A_s,\ \ \ A_1<A_2<\dots<A_s\\
&\mathrm{Im}(\pi)=B_1\cup B_2\cup\dots\cup B_t,\ \ \ B_1<B_2<\dots<B_t\\
&s+t=n+1
\end{aligned}
\end{equation}
which satisfy either
\begin{equation}
A_1\le B_1\le A_2\le B_2\le\dots
\end{equation}
or
\begin{equation}
B_1\le A_1\le B_2\le A_2\le \dots
\end{equation}
(that is, the rightmost end-point of $A_i$ may coincide with the leftmost end-point of the sequel $B$). \\
The linking number $\mathbf{lk}(\tau,\pi)$ is defined as $n$ is the minimal possible shuffling of $\tau,\pi$ has length $n$.
}
\end{defn}
See [BD, Section 2.2], for examples.

\begin{defn}\label{linkingdef2}
{\rm
Let $X$ be a cosimplicial monoid, $n\ge 0$. $X$ is called {\it $n$-commutative} if for any $\tau\colon [p]\to[m]$, $\pi\colon [q]\to [m]$ in $\Delta$ with $\mathbf{lk}(\tau,\pi)\le n$, the diagram below commutes:
\begin{equation}
\xymatrix{
X(p)\otimes X(q)\ar[rrrr]^{X(\tau)\otimes X(\pi)}\ar[d]&&&&X(m)\otimes X(m)\ar[d]^{\mu}\\
X(q)\otimes X(p)\ar[rr]^{X(\pi)\otimes X(\tau)}&&X(m)\otimes X(m)\ar[rr]^{\mu}&&X(m)
}
\end{equation}
}
\end{defn}

The following result is proven in [BD, Th. 2.45, Cor. 2.46]:
\begin{theorem}\label{theorembd}
Let $X$ be an $n$-commutative cosimplicial monoid in $C(\k)$. Then there is an action of the operad homotopy equivalent to $C_\ldot(E_{n+1},\k)$ on the totalization $\Tot(X)\in C(\k)$. 
\end{theorem}
In [BD], some explicit formulas for the degree $-n$ Lie bracket are provided, see [BD, Sections 2.9, 2.10]. 

\vspace{0.5cm}

W easily prove:
\begin{prop}\label{propneq2}
Let $C,D$ be $\k$-linear bicategories, $F\colon C\to D$ a strong bicategorical functor. Then the cosimplicial vector space $Rp_*(A(C,D)(F,F)(\id,\id))$ is a 1-commutative cosimplicial monoid.
\end{prop}
\begin{proof}
Let $\tau_{m,n}\colon [n]\to [m+n]$ and $\pi_{m,n}\colon [m]\to [m+n]$ are defined as $\tau_{m,n}(i)=i$ and $\pi_{m,n}(j)=n+j$. It is clear that $\lk(\tau_{m,n},\pi_{m,n})=1$. Moreover, the general case of linking number 1 is reduced to this particular case, due to the following simple observation ([BD, Lemma 2.1]):\\
Let $\tau\colon [p]\to [m], \pi\colon [q]\to [m]$ be morphisms in $\Delta$, and
$$
[p]\to [p^\prime] \xrightarrow{\tau^\prime}[m],\ \ \ [q]\to[q^\prime]\xrightarrow{\pi^\prime}[m]
$$
be their epi-mono factorisations. Then $\lk(\tau,\pi)=\lk(\tau^\prime,\pi^\prime)$.\\
\\
We check the 1-commutativity of $Rp_*(A(F,F))$. We firstly assume that $C$ is strict and $F$ is strict. Let $\Phi\in Rp_*(A(C,D)(F,F)(\id,\id))^n$, $\Psi\in Rp_*(A(C,D)(F,F)(\id,\id))^m$ be represented by cochains 
$\Phi\in A(C,D)(F,F)(\id,\id)_D$,\\ $\Psi\in A(C,D)(F,F)(\id,\id)_{D^\prime}$, with $p(D)=[n], p(D^\prime)=[m]$. Assume $D=([n]; [k_1],\dots,[k_n])$ and $D^\prime=([m]; [\ell_1],\dots, [\ell_m])$. Then $\tau_{m,n}(\Phi)$ takes a non-zero value on the object $\hat{D}=([n+m]; [k_1],\dots,[k_n],[0],\dots,[0])$, and is equal to
$$
\tau_{m,n}(\Phi)(-,X_{n+1},\dots,X_{m+n})=\Phi(-)\circ^h (\id_{F(X_{n+1}\circ\dots \circ X_{n+m})})
$$
Analogously, $\pi_{m,n}(\Psi)$ takes a non-zero value on $\hat{D^\prime}=([m+n]; [0],\dots,[0], [\ell_1],\dots,[\ell_m])$, and
$$
\pi_{m,n}(\Psi)(Y_1,\dots,Y_m,-)=\id_{F(Y_1\circ\dots\circ Y_m)}\circ^h \Psi(-)
$$
Finally, for their product in the monoid $Rp_*(A(C,D)(F,F)(\id,\id))^{m+n}$, one has
\begin{equation}\label{linking1}
\begin{aligned}
\ &\tau_{m,n}(\Phi)* \pi_{m,n}(\Psi)(T_1,\dots,T_{m+n})=\\
&\Big(\Phi(T_1,\dots, T_n)\circ^h \id_{F(X_{n+1}\circ\dots \circ X_{m+n})}\Big)\circ^v \Big(\id_{F(Y_1\circ\dots\circ Y_n)}\\circ^h \Psi(T_{n+1},\dots,T_{m+n}\Big)=\\
&\Phi(T_1,\dots,T_n)\circ^h \Psi(T_{n+1},\dots,T_{m+n})
\end{aligned}
\end{equation}
where $T_i$ is a string of morphisms of length $k_i$ for $1\le i\le n$ and $\ell_{j-n}$ for $j=n+1,\dots,n+m$, starting at $X_i$ and ending at $Y_i$. 

We clearly get the same expression when computing $\pi_{m,n}(\Psi)*\tau_{m,n}(\Phi)(T_1,\dots,T_{m+n})$, and 1-commutativity for $Rp_*(A(C,D)(F,F)(\id,\id))$ follows. It completes the proof for $F$ strict.

Now, when $C$ is a bicategory and $F$ is strong bicategorical functor, we argue as follows. 

There are two sorts of the structure isomorphisms which figure in (the strong counterpart of) \eqref{linking1}. These two sorts are (a) the structure constraints of the bicategory, and (b) the structure constraints of the functor $F$. It follows from Definition \ref{deflaxfunctor}, the structure constraints of type (a) commute with the structure constraints of type (b); on the other hand, the elements 
$Rp_*(A(C,D)(F,F)(\id,\id))$ {\it commute} with the constraints of type (a),  in the sense of \eqref{relmain11}-\eqref{relmain33}. These two properties imply that the presence of these constraints does not affect the previous speculation in the strict case. 
\end{proof}

\begin{remark}\label{remfailure3}{\rm
The fulfilment of the 1-commutativity condition for $Rp_*(A(C,D)(F,F)(\id,\id))$ is a lucky situation, which is easily generalised from the corresponding proof for the classical Davydov-Yetter complex in [BD, Theorem 3.4]. Namely, (although our cochains are not natural transformations) one does not use the naturality of cochains for general morphisms in this proof. One does use the naturality with respect to the identity morphisms, which automatically holds. \\
The case of 2-commutativity of $Rp_*(\Id,\Id)$ is not that lucky, because the corresponding proof for the classical counterpart given in [BD, Theorem 3.8], essentially uses the naturality for non-identity morphisms. Our cochains are not natural transformations, which results in the failure of 2-commutativity for $Rp_*(A(C,D)(F,F)(\id,\id))$. However, a sort of ``homotopy 2-commutativity'' still holds.
}
\end{remark}

\begin{theorem}\label{defftheorem}
Let $C,D$ be $\k$-linear bicategories, $F\colon C\to D$ a strong bicategorical functor. Then the 2-cocellular totalization $\Tot_{\Theta_2}(A(C,D)(F,F)(\id,\id))$ has a structure of an algebra over an operad homotopically equivalent to $C_\ldot(E_2;\k)$.
\end{theorem}
\begin{proof}
By Proposition \ref{prop!}, $\Tot_{\Theta_2}(A(C,D)(F,F)(\id,\id))\sim \Tot_\Delta(Rp_*(A(C,D)(F,F)(\id,\id))$. By Proposition \ref{propneq2}, $Rp_*(A(C,D)(F,F)(\id,\id))$ is a 1-commutative cosimplicial monoid. Then the result follows from Theorem \ref{theorembd}.
\end{proof}

\section{\sc The totalizations $\Tot_{\Theta_2}(A(C,D)(F,F)(\id,\id))$ and \\ $\Tot_{\Theta_2}(A(C,C)(\Id,\Id)(\id,\id))$ as deformation complexes}\label{sectionlast}

\subsection{\sc Infinitesimal deformation theory of a monoidal dg category $C$}
We start with deformations of a monoidal category $C$. 

Let $C$ be a $\k$-linear monoidal category (or a monoidal dg category over $\k$). The deformations we consider are {\it formal} deformations, that is, $C_t$ may not make sense unless $t=0$. That is, the category $C_t$ is a monoidal category over the formal power series $\k[[t]]$. 

We consider {\it flat} deformations $C_t$ of $C$ in the following sense: the set of objects, the vector spaces (complexes) $C_t(x,y)$ of morphisms, and the monoidal product {\it on objects}, remain undeformed. 

Then the data which is being deformed is:
\begin{itemize}
\item[(A1)] the composition of morphisms $m_{X,Y,Z}\colon C(Y,Z)\otimes C(X,Y)\to C(X,Z)$, $X,Y,Z\in C$,
\item[(A2)] for $f\in C(X,X^\prime), g\in C(Y,Y^\prime)$, the monoidal products of {\it morphisms} $m_{X,g}= \id_X\otimes g\colon C(X,Y)\to C(X,Y^\prime)$ and $m_{f,Y}=f\otimes \id_Y\colon C(X,Y)\to C(X^\prime,Y)$,
\item[(A3)] the associator $\alpha_{X,Y,Z}\colon X\otimes (Y\otimes Z)\to (X\otimes Y)\otimes Z$, $X,Y,Z\in C$,
\item[(A4)] the left and right unit maps $\lambda_X\colon e\otimes X\to X$  and $\rho_X\colon X\otimes e\to X$.
\end{itemize}
{\it It is assumed that (a) the identity morphism $\id_X$, $X\in C$, (b) the monoidal unit $e$, (c) the maps $\lambda_Y$, $\rho_X$ are stable under the deformations, (d) $m_{f,e}$ and $m_{e,g}$ are stable under the deformations, and (e) $m_{X,\id_Y}=m_{\id_X,Y}=\id_{X\otimes Y}$}.

The following example shows that this set-up is realistic:

\begin{example}{\rm
Let $A$ be a bialgebra over $\k$, $C=\Mod(A)$ the category of left $A$-modules over the underlying algebra.
It is a monoidal category in a standard way: for two modules $M,N$, the tensor product of the underlying vector spaces $M\otimes_\k N$ is naturally an $A\otimes A$-module, and the precomposition with $\Delta\colon A\to A\otimes A$ makes it an $A$-module. 

Assume that $A$ is a Hopf algebra. Then the monoidal product $A\otimes A$ of two {\it free} modules of rank 1 is a free module again, whose underlying vector space is {\it canonically} isomorphic to $A\otimes A_u$, where $A_u$ is the underlying vector space of $A$. 

Indeed, define the maps $\alpha\colon A\otimes A\to A\otimes A_u$ and $\beta\colon A\otimes A_u\to A\otimes A$ as
$$
\alpha(a\otimes b)=\sum a_1\otimes S(a_2)b
$$
$$
\beta(a\otimes b)=\sum a_1\otimes a_2b
$$
where $S\colon A\to A$ is the antipode, and we use the Swindler notations $\Delta(a)=\sum a_1\otimes a_2$.

One that $\alpha$ and $\beta$  are maps of $A$-modules, and that $\alpha\circ \beta=\id$, $\beta\circ \alpha=\id$. It proves the claim.

Thus, if we consider a deformation $A_t$ of a Hopf algebra $A$, the $\k$-linear subcategory $C_\free(A_t)$ is a deformation of a monoidal $\k$-linear category $C_\free(A)$, for which the conditions $(A1)-(A3)$ are fulfilled. 

}
\end{example}

The data listed in (A1)-(A4) is subject to the following axioms:
\begin{itemize}
\item[(R1)] the composition $m_{X,Y,Z}$ in (A1) is associative,
\item[(R2)] for maps in (A2) one has $(f\otimes\id_y)\circ (\id_x\otimes g)=(\id_x\otimes g)\circ (f\otimes \id_y)$,
(both sides are equal to $f\otimes g\colon X\otimes Y\to X^\prime\otimes Y^\prime$, therefore, the deformation of $f\otimes g$ are determined by deformations of $f\otimes\id_y$ and $\id_x\otimes g$),
\item[(R3)] for any two composable morphisms $X\xrightarrow{f}X^\prime\xrightarrow{f^\prime}X^{\pprime}$, and any $Y\in C$, one has $m_{f^\prime,Y}\circ m_{f,Y}=m_{f^\prime\circ f,Y}$; similarly, for any two composable morphisms $Y\xrightarrow{g}Y^\prime\xrightarrow{g^\prime}Y^\pprime$, and any $X\in C$, one has
$m_{X,g^\prime}\circ m_{X,g}=m_{X,g^\prime\circ g}$,
\item[(R4)] this and the next two axioms express naturality of the associator. 
Let $f\colon X\to X^\prime$ be a morphism in $C$, and $Y,Z$ objects. The following diagram commutes:
$$
\xymatrix{
X\otimes(Y\otimes Z)\ar[rr]^{\alpha_{X,Y,Z}}\ar[d]_{m_{f,Y\otimes Z}}&&(X\otimes Y)\otimes Z\ar[d]^{m_{m_{f,Y},Z}}\\
X^\prime\otimes (Y\otimes Z)\ar[rr]^{\alpha_{X^\prime,Y,Z}}&&(X^\prime\otimes Y)\otimes Z
}
$$

\item[(R5)] let $g\colon Y\to Y^\prime$ be a morphism in $C$, $X,Z$ objects. The following diagram commutes:
$$
\xymatrix{
X\otimes (Y\otimes Z)\ar[rr]^{\alpha_{X,Y,Z}}\ar[d]_{m_{X,m_{g,Z}}}&&(X\otimes Y)\otimes Z\ar[d]^{m_{m_{X,g},Z}}\\
X\otimes (Y^\prime \otimes Z)\ar[rr]^{\alpha_{X,Y^\prime, Z}}&&(X\otimes Y^\prime)\otimes Z
}
$$

\item[(R6)] let $h\colon Z\to Z^\prime$ be a morphism in $C$, $X,Y$ objects. Then the following diagram commutes:
$$
\xymatrix{
X\otimes(Y\otimes Z)\ar[rr]^{\alpha_{X,Y,Z}}\ar[d]_{m_{X,m_{Y,h}}}&&(X\otimes Y)\otimes Z\ar[d]^{m_{X\otimes Y,h}}\\
X\otimes (Y\otimes Z^\prime)\ar[rr]^{\alpha_{X,Y,Z^\prime}}&&(X\otimes Y)\otimes Z^\prime
}
$$

\item[(R7)] the pentagon equation for the associator:
\begin{equation}\label{r7assoc}
    \begin{tikzpicture}[commutative diagrams/every diagram]
  \node (P0) at (90:2.3cm) {$(X\otimes Y)\otimes (Z\otimes T)$};
  \node (P1) at (90+72:2cm) {$X\otimes (Y\otimes (Z\otimes T))$} ;
  \node (P2) at (90+2*72:2cm) {\makebox[5ex][r]{$X\otimes ((Y\otimes Z)\otimes T)$}};
  \node (P3) at (90+3*72:2cm) {\makebox[5ex][l]{$(X\otimes (Y\otimes Z))\otimes T$}};
  \node (P4) at (90+4*72:2cm) {$((X\otimes Y)\otimes Z)\otimes T$};
  \path[commutative diagrams/.cd, every arrow, every label]
    (P0) edge node {$\alpha_{X\otimes Y,Z,T}$} (P4)
    (P1) edge node {$\alpha_{X,Y,Z\otimes T}$} (P0)
    (P2) edge node[swap] {$\alpha_{X,Y\otimes Z,T}$} (P3)
    (P1) edge node[swap] {$m_{X,\alpha_{Y,Z,T}}$} (P2)
    (P3) edge node[swap] {$m_{\alpha_{X,Y,Z},T}$} (P4);
\end{tikzpicture}
\end{equation}
\item[(R8)] left unit functionality: for any map $f\colon X\to X^\prime$ the diagram
$$
\xymatrix{
X\otimes e\ar[r]^{\rho_X}\ar[d]_{m_{f,e}}&X\ar[d]^{f}\\
X^\prime\otimes e\ar[r]_{\rho_{X^\prime}}&X^\prime
}
$$
commutes,
\item[(R9)] right unit functionality: for any $g\colon Y\to Y^\prime$ the diagram
$$
\xymatrix{
e\otimes Y\ar[r]^{\lambda_Y}\ar[d]_{m_{e,g}}&Y\ar[d]^{g}\\
e\otimes Y^\prime\ar[r]_{\lambda_{Y^\prime}}&Y^\prime
}
$$
commutes,
\item[(R10)] left right unit compatibility: the two possible maps $\lambda_e,\rho_e\colon e\otimes e\to e$ are equal.

\end{itemize}

Among the deformations $C_t$ there are ones which we consider as ``trivial''. It appears in the literature under the name ``twist'', however, here we consider ``upgraded'' twists acting not only on the associator, but as well on the underlying category structure and on the action of morphisms on the monoidal product. 

In deformation theory, one identifies two deformations if one is obtained from another by a twist, and interests in the ``quotient-space''. 

\begin{lemma}\label{twist}
Let $C$ be a $\k$-linear (or dg over $\k$) monoidal category, denote by $C_u$ the underlying $\k$-linear quiver of $C$. Assume that, for any $X,Y\in C$, we are given an isomorphism $\varphi_{X,Y}\colon C(X,Y)\to C(X,Y)$, and an isomorphism $\psi_{X,Y}\in C(X\otimes Y, X\otimes Y)$. Then these data gives rise to a monoidal equivalence functor $F$ from $C$ to another monoidal $\k$-linear (resp., dg over $\k$) category $\tilde{C}$ on the quiver ${C}_u$, such that $F$ is the identity map on any object of $C$. 
\end{lemma}

\begin{proof}
It is standard. We define $F$ on morphisms by $F(f)=\varphi_{X,Y}(f)$ if $f\in C(X,Y)$, and define monoidal constrains 
$F(X\otimes Y)\to F(X)\otimes F(Y)$ as the isomorphisms $\psi_{X,Y}$. Then the monoidal category structure on $\tilde{C}$ is uniquely determined by the requirement that $F$ is a monoidal functor. 

We assume that $\varphi_{X,X}(\id_X)=\id_X$, $\psi_{e,Y}=\id_{e\otimes Y}$, $\psi_{X,e}=\id_{X\otimes e}$. As well, we assume that the constrain $F(e)\to e$ is the identity map.

For convenience of the reader, we provide explicit formulas for the new composition of morphisms, for the action of morphisms on the tensor product, and for the associator. We use the same notations decorated by $\sim$ for the corresponding data (A1)-(A3) of the new category on $\tilde{C}$. We use the same notations as in (A1)-(A3). One has:
\begin{equation}\label{exttwist1}
\tilde{m}_{X,Y,Z}(f,g)=\varphi_{X,Z}(m_{X,Y,Z}(\varphi^{-1}_{Y,Z}(g),\varphi_{X,Y}^{-1}(f)))
\end{equation}

\comment
\begin{equation}\label{exttwist2}
\begin{aligned}
\ &\tilde{m}_{X,g}=\psi_{X,Y^\prime}\circ m_{X,\varphi_{Y,Y^\prime}^{-1}(g)}\circ \psi_{X,Y}^{-1}\\
&\tilde{m}_{f,Y}=\psi_{X^\prime,Y}\circ m_{\varphi_{X,X^\prime}^{-1}(f),Y}\circ \psi^{-1}_{X,Y}
\end{aligned}
\end{equation}
\endcomment

\begin{equation}\label{exttwist2}
\begin{aligned}
\ &\tilde{m}_{X,g}\tilde{\circ} \psi_{X,Y}=\psi_{X,Y^\prime}\tilde{\circ}\varphi_{X\otimes Y,X\otimes Y^\prime}(m_{X,\varphi^{-1}_{Y,Y^\prime}g})\\
&\tilde{m}_{f,Y}\tilde{\circ}\psi_{X,Y}=\psi_{X^\prime,Y}\tilde{\circ}\varphi_{X\otimes Y,X^\prime\otimes Y}(m_{\varphi^{-1}_{X,X^\prime}f,Y})
\end{aligned}
\end{equation}
where $\tilde{\circ}$ denotes the composition in $\tilde{C}$ (given by \eqref{exttwist1}).

The two last equations follow from the commutative diagram
$$
\xymatrix{
F(X\otimes Y)\ar[r]\ar[d]_{F(f\otimes g)}&F(X)\otimes F(Y)\ar[d]^{F(f)\otimes F(g)}\\
F(X^\prime\otimes Y^\prime)\ar[r]&F(X^\prime)\otimes F(Y^\prime)
}
$$
where $\tilde{\circ}$ is the composition in $\tilde{C}$, and $F(?)=?$ for any object $?\in C$. 

\begin{equation}\label{exttwist3}
\tilde{\alpha}_{X,Y,Z}= \tilde{m}_{\psi_{X,Y},Z}\tilde{\circ}\psi_{X\otimes Y, Z}\tilde{\circ} \varphi(\alpha_{X,Y,Z})\tilde{\circ }\psi^{-1}_{X,Y\otimes Z}\tilde{\circ} \tilde{m}_{X, \psi^{-1}_{Y,Z}}
\end{equation}

It comes from the commutative diagram:
$$
\xymatrix{
F(X\otimes (Y\otimes Z))\ar[r]^{F(\alpha)}\ar[d]_{\psi_{X,Y\otimes Z}}&F((X\otimes Y)\otimes Z)\ar[d]^{\psi_{X\otimes Y,Z}}\\
F(X)\otimes F(Y\otimes Z)\ar[d]_{\id_X\otimes \psi_{Y,Z}}&F(X\otimes Y)\otimes F(Z)\ar[d]^{\psi_{X,Y}\otimes\id_Z}\\
F(X)\otimes (F(Y)\otimes F(Z))\ar[r]^{\tilde{\alpha}}&(F(X)\otimes F(Y))\otimes F(Z)
}
$$
where $\tilde{\circ}$ is the composition in $\tilde{C}$, and $F(?)=?$ for any object $?\in C$.

One can check directly that $\tilde{m}_{X,Y,Z}, \tilde{m}_{X,g}, \tilde{m}_{f,Y}, \tilde{\alpha}_{X,Y,Z}$ satisfy (R1)-(R7), and thus define a monoidal category $\tilde{C}$, such that the functor $F\colon C\to \tilde{C}$ is a monoidal equivalence. 
\end{proof}

Note that in the assumption of Lemma $\varphi_{X,Y}$ and $\psi_{X,Y}$ are arbitrary isomorphisms. Now we switch back to formal deformation theory. 

By definition, a {\it trivial deformation} depends on the following data:
\begin{itemize}
\item[(T1)]
a formal power series $\varphi_{X,Y}\colon C(X,Y)\to C(X,Y)$, for any $X,Y\in C$ of the form
\begin{equation}\label{inftw1}
\varphi_{X,Y}(t)=\Id_{C(X,Y)}+t\cdot \varphi_{X,Y}^1+t^2\cdot \varphi_{X,Y}^2+\dots
\end{equation}
where $\varphi_{X,Y}^i\in \Hom_\k(C(X,Y), C(X,Y)), i\ge 1$,
\item[(T2)] a formal power series $\psi_{X,Y}\colon C(X\otimes Y, X\otimes Y)$, for any $X,Y\in C$, of the form
\begin{equation}\label{inftw2}
\psi_{X,Y}=\Id_{X\otimes Y}+t\cdot \psi_{X,Y}^1+t^2\cdot \psi_{X,Y}^2+\dots
\end{equation}
where $\psi^i_{X,Y}\in C(X\otimes Y, X\otimes Y), i\ge 1$.
\end{itemize}
Out of this data, a formal deformation of $C$ is constructed as in \eqref{exttwist1}-\eqref{exttwist3}.

One defines the concepts of an {\it infinitesimal deformation} and of a {\it trivial infinitesimal deformation} of a monoidal (linear or dg) category by replacing in the previous definitions the ring of formal power series $\k[[t]]$ by the dual numbers $\k[t]/(t^2)$. We say that two infinitesimal deformations belong to the same equivalence class if the corresponding monoidal categories are equivalent by an (extended) twist, as in Lemma \ref{twist} but over $\k[t]/(t^2)$. 

One has:

\begin{theorem}\label{infdefcat}
Let $C$ be a $\k$-linear (or a dg over $\k$) monoidal category. 
The third cohomology $H^3(\Tot_{\Theta_2}A((C,C)(\Id,\Id)(\id,\id)))$ is isomorphic to the equivalence classes of infinitesimal deformations (in the sense specified above) of the monoidal (dg) category $C$.
\end{theorem}
\begin{proof}
The proof is a rather long but standard computation, for which we refer the reader to [PS], Section 5. 

For simplicity, we assume that all complexes $C(X,Y)$ are concentrated in cohomological degree 0, that is, are ordinary vector spaces over $\k$. The general case is similar but technically more involved, and we leave it to the reader. 

By theorem \ref{theornorm}, we can establish the isomorphism in Theorem for the {\it normalized} complex $A_\norm(C,C)(\Id,\Id)(\id,\id)$. Below we assume that the  cochains are normalized. 

A general element $\pi$ of $\Tot_{\Theta_2}(A(C,C)(\Id,\Id))$ of degree 3 is a sum of cochains of types (1;2), (2;1,0), (2;0,1), and (3;0,0,0). Denote them by $\kappa,\beta^\ell, \beta^r,\gamma$, correspondingly:
$$
\pi=\kappa+\beta^{\ell}+\beta^r+\gamma
$$
We identify them with infinitesimal deformations of $m_{X,Y,Z}$ for $\kappa$, of $m_{f,Y}$ for $\beta^\ell$, of $m_{X,g}$ for $\beta^r$, and of $\alpha_{X,Y,Z}$ for $\gamma$, see (A1)-(A4):
\begin{equation}\label{eqcycle2}
\begin{aligned}
\ &\tilde{m}_{X,Y,Z}=m_{X,Y,Z}+t\cdot \kappa_{X,Y,Z}\\
&\tilde{m}_{f,Y}=m_{f,Y}+t\cdot \beta^\ell_{f,Y}\\
&\tilde{m}_{X,g}=m_{X,g}+t\cdot \beta^r_{X,g}\\
&\tilde{\alpha}_{X,Y,Z}=\alpha_{X,Y,Z}(1+t\cdot \gamma_{X,Y,Z})
\end{aligned}
\end{equation}
It follows from our conditions on deformations (listed below (A1)-(A4)) that $\kappa,\beta^\ell,\beta^r$  are reduced cochains. Also it is clear that $\kappa,\beta^\ell,\beta^r,\gamma$ belong to $A(\Id,\Id)$, that is, \eqref{exttwist1}-\eqref{exttwist3} are fulfilled for them. Note also that by our assumptions made just below the list (A1)-(A4), one has $\beta^\ell_{f,e}=\beta^r_{e,g}=0$, where $e$ is the monoidal unit.  

 Assume $\pi$ is a cycle:
\begin{equation}\label{eqcycle3}
d\pi=d(\kappa+\beta^{\ell}+\beta^r+\gamma)=0
\end{equation}
This equation is a system of several equations, one equation for any diagram of dimension 4. We write schematically:
\begin{equation}
\begin{aligned}
\ &d\kappa=(1;3)+(2;0,2)+(2;1,1)+(2;2,0)\\
&d\beta^\ell=(3;0,1,0)+(3;1,0,0)+(2;2,0)+(2;1,1)\\
&d\beta^r=(3;0,1,0)+(3;0,0,1)+(2;0,2)+(2;1,1)\\
&d\gamma=(4;0,0,0,0)+(3;0,0,1)+(3;0,1,0)+(3;1,0,0)
\end{aligned}
\end{equation}
We see that there are many cross-terms in \eqref{eqcycle3}. 

Now consider relations (R1)-(R10) for tilde-data \eqref{eqcycle2}, taking to the account $t^2=0$. We get system of linear in $\kappa,\beta^\ell,\beta^r,\gamma$ equations. The claim is that these equations are exactly the homogeneous components (for any given degree, e.g. (2;1,1)) of the equation $d\pi=0$.

\vspace{1mm}
Note that for the cases (R4)-(R9) we have to use relations \eqref{relmain11}-\eqref{relmain33}. 

The case of (R1) is standard, the computation here is basically the same as the classical computation with Hochschild complex. 

\vspace{1mm}

The infinitesimal version of (R2) gives:

$$
(m_{f,Y^\prime}+t\beta^\ell_{f,Y^\prime})\tilde{\circ }(m_{X,g}+t\beta^r_{X,g})=
(m_{X^\prime,g}+t\beta^r_{X^\prime,g})\tilde{\circ} (m_{f,Y}+t\beta^\ell_{f,Y}) \mod t^2
$$
where $\tilde{\circ}=\circ +t\cdot\kappa$,
gives
$$
\beta^\ell_{f,Y^\prime}\circ m_{X,g}+m_{F,Y^\prime}\circ\beta^r_{X,g}+\kappa(m_{f,Y^\prime}, m_{X,g})=
\beta^r_{X^\prime,g}\circ m_{f,Y}+m_{X^\prime,g}\circ \beta^\ell_{f,Y}+\kappa(m_{X^\prime,g}, m_{f,Y})
$$

It is the vanishing of type $(2;1,1)$ cross-terms in $d\beta^\ell+d\beta^r+d\kappa$. (The other summands of $\pi$ do not contain components of type (2;1,1) in their boundary).

For $d\kappa$ these are two possible shuffle maps. 

\vspace{1mm}

The case (R3) comprises two sub-cases, which are analogous. We consider one of them. It is, in the infinitesimal version
$$
(m_{f^\prime,Y}+t\cdot \beta^\ell_{f^\prime,Y})\tilde{\circ}(m_{f,Y}+t\cdot \beta^\ell_{f,Y})=
m_{f^\prime\tilde{\circ} f, Y}+t\cdot \beta^{\ell}_{f^\prime\tilde{\circ}f,Y}\mod t^2
$$
It gives
$$
\beta^\ell_{f^\prime,Y}\circ m_{f,Y}+m_{f^\prime,Y}\circ \beta^\ell_{f,Y}+\kappa(m_{f^\prime,Y}, m_{f,Y})=\\
m_{\kappa(f^\prime,f),Y}+\beta^\ell_{f^\prime\circ f,Y}
$$

It is the vanishing of type $(2;2,0)$ cross-terms in $d(\beta^\ell+\kappa)$. The other summands of $\pi$ do not contain $(2; 2,0)$ type elements in their boundary. 

\vspace{1mm}

The cases (R4), (R5), (R6) are similar, we consider one of them, (R5). 
This is the first case where we essentially use the relations \eqref{relmain11}-\eqref{relmain33}.

The case (R5) in the infinitesimal version reads:

\begin{equation}
(m_{m_{X,g},Z}+t\beta^\ell_{m_{X,g},Z}+t m_{\beta^r_{X,g},Z})  \tilde{\circ}(\alpha+t\gamma\circ \alpha)=(\alpha+t\gamma\circ\alpha)\tilde{\circ}(m_{X,m_{g,Z}}+t\beta^r_{X,m_{g,Z}}+t m_{X,\beta^\ell_{g,Z}})\ \mod t^2
\end{equation}
The terms in $t$ give the identity, all terms are of type $(3; 0,1,0)$:
\begin{equation}\label{r5comp}
\begin{aligned}
\ &m_{m_{X,g},Z}\circ (\gamma\circ \alpha)+\kappa(\alpha, m_{m_{X,g},Z})+\beta^\ell_{m_{X,g},Z}\circ \alpha+m_{\beta^r_{X,g},Z}\circ\alpha=\\
&(\gamma\circ\alpha)\circ m_{X,m_{g,Z}}+\kappa(m_{X,m_{g,Z}},\alpha)+\alpha\circ \beta^r_{X,m_{g,Z}}+\alpha\circ m_{X,\beta^\ell_{g,Z}})
\end{aligned}
\end{equation}
\begin{lemma}\label{r5l}
One has $\kappa(\alpha, m_{m_{X,g},Z})=0$ and $\kappa(m_{X,m_{g,Z}},\alpha)=0$.
\end{lemma}
\begin{proof}
We prove the first equation, the second one is analogous.

One has by \eqref{relmain11}: $\kappa(\alpha, m_{m_{X,g},Z})=\kappa(\id, m_{m_{X,g},Z}\circ\alpha)=0$, where the second equality follows from the requirement that the identity maps are preserved under the deformation. 
\end{proof}

The 6 non-vanishing terms in \eqref{r5comp} are interpreted as the degree (3;0,1,0) terms in $d\pi$, as follows. We rewrite \eqref{r5comp} as:

\begin{equation}\label{eqr5lx}
m_{m_{X,g},Z}\circ \gamma+\beta^\ell_{m_{X,g},Z}+m_{\beta^r_{X,g},Z}=
\gamma\circ(\alpha\circ m_{X,m_{g,Z}}\circ \alpha^{-1})+\alpha\circ \beta^r_{X,m_{g,Z}}\circ \alpha^{-1}+\alpha\circ m_{X,\beta^\ell_{g,Z}}\circ\alpha^{-1}
\end{equation}
(here we made use of Lemma \ref{r5l}). 

Next, $\alpha\circ m_{X,m_{g,Z}}\circ \alpha^{-1}=m_{m_{X,g},Z}$, so the first summand in the r.h.s. of \eqref{eqr5lx} is equal to $\gamma\circ m_{m_{X,g},Z}$. 

We have:
\begin{equation}
\begin{aligned}
\ &(d\gamma)_{(3;0,1,0)}=m_{m_{X,g},Z}\circ \gamma-\gamma\circ m_{m_{X,g},Z}\\
&(d\beta^\ell)_{(3;0,1,0)}=\beta^\ell_{m_{X,g},Z}-\alpha\circ m_{X,\beta^\ell_{g,Z}}\circ\alpha^{-1}\\
&(d\beta^r)_{(3;0,1,0)}=m_{\beta^r_{X,g},Z}-\alpha\circ \beta^r_{X,m_{g,Z}}\circ \alpha^{-1}
\end{aligned}
\end{equation}

\vspace{1mm}

The infinitesimal version of (R7) gives rise to the following equation (see \eqref{r7assoc}):
\begin{equation}\label{r71}
\begin{aligned}
\ &(\alpha_{XY,Z,W}+t\gamma_{XY,Z,W}\circ\alpha_{XY,Z,W})\tilde{\circ}
(\alpha_{X,Y,ZW}+t\gamma_{X,Y,ZW}\circ \alpha_{X,Y,ZW})=\\
&(m_{\alpha_{X,Y,Z},W}+t\beta^\ell_{\alpha_{X,Y,Z},W}+tm_{\gamma_{X,Y,Z}\circ\alpha_{X,Y,Z},W})
\tilde{\circ}(\alpha_{X,YZ,W}+t\gamma_{X,YZ,W}\circ \alpha_{X,YZ,W})\\
&\tilde{\circ}(m_{X,\alpha_{Y,Z,W}}+t\beta^r_{X,\alpha_{Y,Z,W}}+tm_{X,\gamma_{Y,Z,W}\circ\alpha_{Y,Z,W}})\mod t^2
\end{aligned}
\end{equation}
where, as above, $-\tilde{\circ}-=m(-,-)+t\kappa(-,-)$. 
One has:
\begin{lemma}\label{r7l}
The following identities hold:
\begin{itemize}
\item[(1)] $\kappa(\alpha,-)=\kappa(-,\alpha)=0$, where $\alpha$ has any arguments such as $\alpha_{XY,Z,W}$ etc.,
\item[(2)] $\beta^\ell_{\alpha_{X,Y,Z},W}=\beta^{r}_{X,\alpha_{Y,Z,W}}=0$.
\end{itemize}
\end{lemma}
\begin{proof}
(1) is proven as in Lemma \ref{r5l}. For (2), prove the first assertion (the second one is analogous):
$\beta^\ell_{\alpha_{X,Y,Z},W}\overset{\eqref{relmain22}}{=}\beta^\ell_{\id_{X\otimes (Y\otimes Z),W}}\circ m_{\alpha_{X,Y,Z},W}=0$, where the second equality follows from a more general $\beta^\ell_{\id,-}=0$.
\end{proof}
Removing the summands which vanish by Lemma \ref{r7l}, \eqref{r71} in ordet $t$ becomes:
\begin{equation}\label{r72}
\begin{aligned}
\ &\alpha_{XY,Z,W}\circ \gamma_{X,Y,ZW}\circ \alpha_{X,Y,ZW}+\gamma_{XY,Z,W}\circ \alpha_{XY,Z,W}\circ\alpha_{X,Y,ZW}=\\
&m_{\alpha_{X,Y,Z},W}\circ \alpha_{X,YZ,W}\circ m_{X,\gamma_{Y,Z,W}\circ\alpha_{Y,Z,W}}+m_{\alpha_{X,Y,Z},W}\circ(\gamma_{X,YZ,W}\circ\alpha_{X,YZ,W})\circ m_{X,\alpha_{Y,Z,W}}+\\
&m_{\gamma_{X,Y,Z}\circ\alpha_{X,Y,Z},W}\circ\alpha_{X,YZ,W}\circ m_{X,\alpha_{Y,Z,W}}
\end{aligned}
\end{equation}
Next, by \eqref{relmain22} or (R3),
$m_{X,\gamma_{Y,Z,W}\circ\alpha_{Y,Z,W}}=m_{X,\gamma_{Y,Z,W}}\circ m_{X,\alpha_{Y,Z,W}}$ and
$m_{\gamma_{X,Y,Z}\circ\alpha_{X,Y,Z},W}=m_{\gamma_{X,Y,Z},W}\circ m_{\alpha_{X,Y,Z},W}$. 

Now we recognize in the 5 summands of \eqref{r72} the 5 terms in $(d\gamma)_{(4;0,0,0,0)}$.
\vspace{3mm}

It remains to consider (R8)-(R10). The identity (R10) is fulfilled automatically for the deformed category.
The identities (R8)-(R9) are analogous, we consider (R8). Here all morphisms in the diagram are not deformed by our assumptions, but the composition does. So one has to prove that the terms, coming from the deformation of the composition, vanish. These terms are $\rho_X^{-1}\circ \kappa(\rho_X,f)$ and $\kappa(f,\rho_X^{-1})\circ \rho_X$. These terms vanish by \eqref{relmain22}. For example, $\kappa(\rho_X,f)=\kappa(\id,f)\circ \rho_X=0$. (To be precise, one {\it derives} that $\kappa$ obeys \eqref{relmain22} from this speculation).

We have identified 3-cycles $\pi=\kappa+\beta^\ell+\beta^r+\gamma$ in $\Tot_{\Theta_2}(A(\Id,\Id))$ with infinitesimal deformations of the monoidal category $C$. Now we show that the infinitesimal deformations of type \ref{exttwist1}-\eqref{exttwist3} are corresponded to coboundaries $\pi=d\omega$, for a 2-cochain $\omega\in\Tot_{\Theta_2}(A(\Id,\Id))$.

A general 2-cochain $\omega$ is a linear combination of components of the types $(1; 1)$ and $(2;0)$. 
On the other hand, our infinitesimal twists $\varphi^1_{X,Y}$ and $\psi^1_{X,Y}$ (see \eqref{inftw1} and \eqref{inftw2}) are of the same type. We identify their cobounaries with the infinitesimal versions of \eqref{exttwist1}-\eqref{exttwist3}. That means, that we have to identify the components of $d(\varphi^1+\psi^1)$ of degrees $(1;2), (2; 0,1), (2;1,0), (3;0,0,0)$ with the infinitesimal versions of \eqref{exttwist1}, \eqref{exttwist2}(1), \eqref{exttwist2}(2), \eqref{exttwist3} correspondingly. 

For \eqref{exttwist1}, it is clear that $d(\varphi^1+\psi^1)_{(1;2)}=d(\varphi^1)_{(1;2)}$. The computation with \eqref{exttwist1} is standard, it is the same as for the Hochschild cochains, and we leave it to the reader.

For \eqref{exttwist2}(1), rewrite it as
\begin{equation}
\tilde{m}_{X,g}=\varphi(\varphi^{-1}\psi_{X,Y^\prime}{\circ}m_{X,\varphi^{-1}_{Y,Y^\prime}g}{\circ}\varphi^{-1}\psi_{X,Y}^{-1})
\end{equation}
Note that
\begin{equation}
\varphi^{-1}\psi=(\id-t\varphi^1+\dots)(\id+t\psi^1+\dots)=\id+t\psi^1 \mod t^2
\end{equation}
because $\phi^1(\id)=0$, similarly for $\varphi^{-1}\psi^{-1}$. 

Then
\begin{equation}
{m}_{X,g}+t m^1_{X,g}=(\id+t\varphi^1)\Big( (\id+t\psi^1_{X,Y^\prime})\circ (m_{X,g}-tm^1_{X,\varphi^1 g})\circ (\id-t\psi^1_{X,Y})\Big) \mod t^2
\end{equation}
from which 
\begin{equation}\label{exttwist2bis}
m^1_{X,g}=\varphi^1 m_{X,g}+\psi^1_{X,Y^\prime}\circ m_{X,g}-m^1_{X,\varphi^1 g}-m_{X,g}\circ \psi^1_{X,Y}
\end{equation}
One easily recognises $d(\varphi^1+\psi^	1)_{(2;0,1)}$ in the r.h.s. of \eqref{exttwist2bis}. 

The computation for \eqref{exttwist2}(2) is similar. 

The remaining case is to identify $d(\varphi^1+\psi^1)_{(3;0,0,0)}=d(\psi^1)_{(3;0,0,0)}$ with (the r.h.s. of) the infinitesimal version of \eqref{exttwist3}, written in the form
\begin{equation}\label{exttwist3_1}
\tilde{\alpha}=\tilde{m}_{\psi_{X,Y},Z}\tilde{\circ}\psi_{X\otimes Y,Z}\tilde{\circ}\varphi(\alpha)\tilde{\circ}\psi^{-1}_{X,Y\otimes Z}\tilde{\circ}\tilde{m}^{-1}_{X,\psi_{Y,Z}}
\end{equation}
The order $t$ term of the r.h.s. of \eqref{exttwist3_1} is:
\begin{equation}\label{exttwist3_2}
\begin{aligned}
\ &(\id_{(X\otimes Y)\otimes Z}+tm_{\psi^1_{X,Y},Z})\tilde{\circ} (\id_{(X\otimes Y)\otimes Z}+t\psi^1_{X\otimes Y,Z})\tilde{\circ}
((\id+t\varphi^1)\circ \alpha)\tilde{\circ}\\
&\tilde{\circ}(\id_{X\otimes (Y\otimes Z)}-t\psi^1_{X,Y\otimes Z})\tilde{\circ}
(\id_{X\otimes (Y\otimes Z)}-tm_{X,\psi^1_{Y,Z}})=\\
&(\id+t\varphi^1)\Big((\id-t\varphi^1)(\id+tm_{\psi^1_{X,Y},Z})\circ (\id+t\varphi^1)(\id+t\psi^1_{X\otimes Y, Z})\circ \alpha\circ\\
&\circ (\id-t\varphi^1)(\id-t\psi^1_{X,Y\otimes Z})\circ (\id-t\varphi^1)(\id-tm_{X,\psi^1_{Y,Z}})\Big)
\end{aligned}
\end{equation}
(we made use that $\tilde{m}_{X,\id}=\tilde{m}_{\id,Y}=\id$, which results in $m_{X,\id}=\id$ and $m^1_{X,\id}=0$ etc).
We compute the r.h.s. of \eqref{exttwist3_2} modulo $t^2$ (using $\phi^1(\id)=0$):
\begin{equation}\label{exttwist3_3}
\begin{aligned}
\ &t\gamma_{X,Y,Z}\circ\alpha=\\
&(\id+t\varphi^1)\Big((\id+tm_{\psi^1_{X,Y},Z})\circ (\id+t\psi^1_{X\otimes Y,Z})\circ \alpha\circ (\id-t\psi^1_{X,Y\otimes Z})\circ(\id-tm_{X,\psi^1_{Y,Z}})\Big)
\end{aligned}
\end{equation}
Thus,
\begin{equation}\label{exttwist3_4}
\begin{aligned}
\ &\gamma_{X,Y,Z}=\Big(\varphi^1(\alpha)+m_{\psi^1_{X,Y},Z}\circ \alpha+\psi^1_{X\otimes Y,Z}\circ \alpha-\alpha\circ \psi^1_{X,Y\otimes Z}-\alpha\circ m_{X,\psi^1_{Y,Z}}\Big)\circ \alpha^{-1}=\\
&\varphi^1(\alpha)\circ\alpha^{-1}+m_{\psi^1_{X,Y},Z}+\psi^1_{X\otimes Y,Z}-\alpha\circ \psi^1_{X,Y\otimes Z}\circ\alpha^{-1}-\alpha\circ m_{X,\psi^1_{Y,Z}}\circ \alpha^{-1}
\end{aligned}
\end{equation}
The first summand of the r.h.s. of \eqref{exttwist3_4} vanishes, because 
$$
\varphi^1(\alpha)\overset{\eqref{relmain22}}{=}\varphi^1(\id)\circ \alpha =0
$$
The 4 remaining summands are the 4 summands in $d(\psi^1)_{(3;0,0,0)}$.

\end{proof}

\subsection{\sc Infinitesimal deformation theory of a strict monoidal functor}
The following Theorem is proven analogously to but easier than Theorem \ref{infdefcat}, and we leave details to the reader.

\begin{theorem}\label{definffun}
Let $C,D$ be $\k$-linear (or dg over $\k$) monoidal categories, $F\colon C\to D$ a monoidal functor. 
The second cohomology $H^2(\Tot_{\Theta_2}A(C,D)(F,F)(\id,\id))$ is isomorphic to the equivalence classes of infinitesimal deformations of the functor $F$.
\end{theorem}

By Theorem \ref{defftheorem}, $\Tot_{\Theta_2}A(C,D)(F,F)$ is a homotopy 2-algebra. In fact, one can construct a dg Lie algebra on $\Tot_{\Theta_2}A(C,D)(F,F)[1]$ directly (without any use of loc.cit.), and to develop, via the Maurer-Cartan equation and the deformation functor associated to dg Lie algebra formalism, the ``global'' deformation theory for $F\colon C\to D$ over $\k[[t]]$.

\appendix

\section{\sc Relations in $\Theta_2$}\label{relationstheta}\label{appendixa}
One has the following relations between the elementary face and degeneracy maps in $\Theta_2$, which are checked straightforwardly.

\begin{equation}\label{reltheta1}
D_{q,\sigma^\prime}D_{p,\sigma}=D_{p,\sigma}D_{q-1,\sigma^\prime}\text{   if   }p<q-1
\end{equation}

\begin{equation}\label{reltheta2}
D_{q,\sigma_2}D_{q-1,\sigma_1}=D_{q-1,\eta_2}D_{q-1,\eta_1}
\end{equation}
Here is an explanation of the notations: any $(a,b)$-shuffle $\sigma_1$ and $(a+b,c)$-shuffle $\sigma_2$ define uniquely a $(b,c)$-shuffle $\eta_1$ and an $(a,b+c)$-shuffle $\eta_2$ such that $\sigma_2\circ (\sigma_1,\id_c)=\eta_2\circ (\id_a,\eta_1)$ (the latter is an $(a,b,c)$-shuffle).

\begin{equation}\label{reltheta3}
\partial_p^j\partial_q^i=\partial_q^i\partial_p^j\text{   if   }p\ne q
\end{equation}

\begin{equation}\label{reltheta4}
\partial_p^j\partial_p^i=\partial_p^i\partial_p^{j-1}\text{   if   }i<j
\end{equation}

\begin{equation}\label{reltheta5}
\begin{aligned}
\ &D_{q,\sigma}\partial_p^j=\partial_{p+1}^j D_{q,\sigma} \text{   if   }p>q\\
&D_{q,\sigma}\partial_p^j=\partial_{p}^j D_{q,\sigma}\text{   if   }p<q
\end{aligned}
\end{equation}

\begin{equation}\label{reltheta6}
D_{p,\sigma}\partial_p^i=\begin{cases}
\partial^a_p D_{p,\bar{\sigma}} &\text{ if   }\sigma^{-1}(\overrightarrow{i,i+1})=\overrightarrow{a,a+1}\in [0,k_p]\\
\partial^b_{p+1}D_{p,\bar{\sigma}}&\text{   if   }\sigma^{-1}(\overrightarrow{i,i+1})=\overrightarrow{b,b+1}\in [k_p,k_p+k_{p+1}]
\end{cases}
\end{equation}
where $\bar{\sigma}$ is the shuffle obtained from $\sigma$ by collapsing $\sigma^{-1}(\overrightarrow{i,i+1})$, and $\{k_s\}$ is used as in Section \ref{facetheta}, (F2).

\begin{equation}\label{reltheta7}
\begin{aligned}
\ & \partial_{p}^iD_\min=D_\min \partial_{p-1}^i\text{   if   }p\ge 1\\
&D_{p,\sigma} D_\min=D_\min D_{p-1,\sigma}\text{   if   }p\ge 1
\end{aligned}
\end{equation}
and similarly for $D_\max$.

\begin{equation}
    \epsilon^j_p \circ \epsilon^i_q = \epsilon^i_q \circ \epsilon^j_p \text{  if  } p\neq q 
\end{equation}
\begin{equation}
    \epsilon^j_p \circ \epsilon^i_p = \epsilon^i_p \circ \epsilon^{j-1}_p \text{  if  } i \leq j 
\end{equation}

\begin{equation}\label{reltheta10}
\Upsilon_0^q\circ \Upsilon_0^p=\Upsilon_0^p\circ \Upsilon_0^{q+1}\text{  if  }p\le q
\end{equation}

\begin{equation}
 \Upsilon^q_\ell \circ \epsilon^j_p = \begin{cases}\epsilon^j_{p-1}\circ \Upsilon^q_\ell &\text{  if  } p > q+1 \\
  \epsilon^j_{p}\circ \Upsilon^q_\ell &\text{  if  } p \le q\\
  \Upsilon^q_{\ell+1}&\text{   if   }p=q+1
  \end{cases}
\end{equation}

\begin{equation}
    \partial^i_p \circ \epsilon^j_q = \epsilon^j_q \circ \partial^i_p \text{  if  } p\neq q
\end{equation}

\begin{equation}
    \epsilon^j_p \circ \partial^i_p = \begin{cases}
       \partial^i_p\circ \epsilon^{j-1}_p &\text{  if  } i < j; \\
       id & \text{  if  } i = j, j+1; \\
       \partial^{i-1}_p \circ \epsilon^j_p & \text{  if  } i > j+1.
    \end{cases}
\end{equation}

\begin{equation}\label{reltheta14}
    \Upsilon^q_\ell \circ \partial^j_p = \begin{cases}
       \partial^{j}_{p-1}\circ \Upsilon^q_\ell &\text{  if  } p > q+1; \\
       \partial^j_p \circ \Upsilon^q_\ell & \text{  if  } p \le q\\
       \Upsilon^q_{\ell-1}&\text{   if   }p=q+1
    \end{cases}
\end{equation}

\begin{equation}
D_{q,\sigma}\circ \epsilon^i_p=\begin{cases}\epsilon^i_{p+1}\circ D_{q,\sigma}& \text{   if   }q<p\\
\epsilon^i_{p}\circ D_{q,\sigma}& \text{   if   }q>p\\
       \epsilon^a_p \circ D_{q,\sigma'}&\text{   if   } q=p, \sigma^{-1}(\overrightarrow{i,i+1})=\overrightarrow{a,a+1}\in [0,k_p]\\
       \epsilon^b_{p+1}\circ D_{q,\sigma'}&\text{   if   }q=p, \sigma^{-1}(\overrightarrow{i,i+1})=\overrightarrow{b,b+1}\in [k_p,k_p+k_{p+1}]
    \end{cases}
\end{equation}
where $\sigma^\prime$ is obtained from $\sigma$ by adding a new element (blowing up) at $\sigma^{-1}(\overrightarrow{i,i+1})$.

\begin{equation}\label{reltheta16}
    \Upsilon^q_{0} \circ D_{p,\sigma} = \begin{cases}
       D_{p,\sigma}\circ \Upsilon^{q-1}_{0} &\text{  if  } p<q \\
       D_{p-1,\sigma}\circ \Upsilon^q_0 & \text{  if  } p > q+1\\
       \id&\text{   if   }p=q, \sigma=(0,k_p+k_{p+1})\\
              \id&\text{   if   }p=q+1, \sigma=(k_p+k_{p+1},0)
    \end{cases}
\end{equation}

\begin{equation}
    D_{\min}\circ \epsilon^i_p = \epsilon^i_{p+1}\circ D_{\min}
\end{equation}

\begin{equation}
    D_{\max}\circ \epsilon^i_p = \epsilon^i_{p}\circ D_{\max}
\end{equation}
\comment
\begin{equation}
    D_{\min}\circ \Upsilon^q_{0} = \begin{cases}
       \Upsilon^{q+1}_0\circ D_{\min} & \text{  if  } q > 0; \\
       \id & \text{  if  } q = 0.
    \end{cases}
\end{equation}
\endcomment
\begin{equation}
\Upsilon^q_0\circ D_\min=\begin{cases}
D_\min\circ \Upsilon^{q-1}_0&\text{  if  }q>0;\\
\id&\text{  if  }q=0
\end{cases}
\end{equation}
\comment
\begin{equation}
    D_{\max}\circ \Upsilon^q_0 = \begin{cases}
       \Upsilon^{q}\circ D_{\max} & \text{  if  } q < n+1; \\
       \id & \text{  if  } q = n+1.
    \end{cases}
\end{equation}
\endcomment
\begin{equation}
\Upsilon^q_0\circ D_\max=\begin{cases}
D_\max\circ \Upsilon^q_0&\text{  if  }q<n+1;\\
\id&\text{  if  }q=n+1
\end{cases}
\end{equation}

\section{\sc A proof of Proposition \ref{propdeltaaction}}\label{sec.proofpropactiondelta}\label{appendixb}
Here we give a proof of Proposition \ref{propdeltaaction}:
\begin{proof}
Let $\phi \in \mathfrak{R}^{\ldot}_{[n]}$. We first want to show that $\Upsilon^q_\Delta(d\phi) = d\Upsilon^q_\Delta(\phi)$:
\begin{equation*}
\begin{split}
\Upsilon^q_\Delta\left( \sum_{s=1}^n\sum_{i=0}^{\kappa_s}(-1)^{\kappa_1 + \dots + \kappa_{s-1}+s-1+i}\phi_{s,i} \right) = \left( \sum_{s=1}^n\sum_{i=0}^{\kappa_s}(-1)^{\kappa_1 + \dots + \kappa_{s-1}+s-1+i}\phi\circ \partial^i_s \right) \circ \Upsilon^q_0 = \\
 = \sum_{s=1, s\neq q+1}^{n+1}\sum_{i=0}^{\kappa_s}(-1)^{\kappa_1 + \dots + \kappa_{s-1}+s-1+i}(\phi\circ \Upsilon^q_0) \circ \partial^i_s = d\Upsilon^q_\Delta(\phi)
\end{split}
\end{equation*}
where the second equality follows from \eqref{reltheta14}. The reason for excluding $s=q+1$ by the summation is due to the fact that there is no face map with codomain $[0]$, which is the $q^{th}$ interval.

Now we would like to show that $\Omega^0_{\Delta}(d\phi) = - d\Omega^0_\Delta(\phi)$: 
\begin{equation*}
    \begin{split}
        &\Omega^0_\Delta \left( \sum_{s=1}^n\sum_{i=0}^{\kappa_s}(-1)^{\kappa_1 + \dots + \kappa_{s-1}+s-1+i}\phi_{s,i} \right) = \left( \sum_{s=2}^n\sum_{i=0}^{\kappa_s}(-1)^{\kappa_2 + \dots + \kappa_{s-1}+s-1+i}\phi\circ \partial^i_s \right) \circ D_{min} = \\
        &- \sum_{s=2}^{n}\sum_{i=0}^{\kappa_s}(-1)^{\kappa_2 + \dots + \kappa_{s-1}+s-2+i}(\phi\circ D_{\min}) \circ \partial^i_{s-1} = \\
        &- \sum_{s=1}^{n-1}\sum_{i=0}^{\kappa'_s}(-1)^{\kappa'_1 + \dots + \kappa'_{s-1}+s-1+i}(\phi\circ D_{\min}) \circ \partial^i_{s} = -d\Omega^0_\Delta(\phi)
    \end{split}
\end{equation*}
the first equality comes from the fact that $k_1 = 0,$ so that there is no $\partial^i_1\colon \dots \to [0]$; the second equality follows from \eqref{reltheta7}, and we define $\kappa'_i \coloneqq \kappa_{i+1}$, for each $i = 1,\dots, n-1$.

One can similarly prove that $\Omega^{n}_{\Delta}(d\phi) = - d\Omega^n_\Delta(\phi)$.

Now we would like to show that $\Omega^p_{\Delta}(d\phi) = - d\Omega^p_\Delta(\phi)$:

\begin{equation*}
    \begin{split}
        &\Omega^p_\Delta \left( \sum_{s=1}^n\sum_{i=0}^{\kappa_s}(-1)^{\kappa_1 + \dots + \kappa_{s-1}+s-1+i}\phi_{s,i} \right) = \\
        &= \left( \sum_{s=1}^n\sum_{i=0}^{\kappa_s}(-1)^{\kappa_1 + \dots + \kappa_{s-1}+s-1+i}\phi\circ \partial^i_s \right) \circ \left( \sum_\sigma (-1)^{\ell_1 + \dots + \ell_{p-1}+p-1 +\sharp(\sigma)} D_{p,\sigma} \right) =  \\
        &(a) + (b) + (c)
    \end{split}
\end{equation*}
where: \[ (a) = \left( \sum_{s=1}^{p-1}\sum_{i=0}^{\kappa_s}(-1)^{\kappa_1 + \dots + \kappa_{s-1}+s-1+i} \phi\circ \partial^i_s \right) \circ \left( \sum_{\sigma}(-1)^{\ell_1 + \dots + \ell_{p-1}+p-1 +\sharp(\sigma)}D_{p,\sigma} \right) \]
\[ (b) = \left( \sum_{s=p,p+1}\sum_{i=0}^{\kappa_s}(-1)^{\kappa_1 + \dots + \kappa_{s-1}+s-1+i} \phi\circ \partial^i_s \right) \circ \left( \sum_{\sigma}(-1)^{\ell_1 + \dots + \ell_{p-1}+p-1 +\sharp(\sigma)}D_{p,\sigma} \right) \]
\[ (c) = \left( \sum_{s=p+2}^{n}\sum_{i=0}^{\kappa_s}(-1)^{\kappa_1 + \dots + \kappa_{s-1}+s-1+i} \phi\circ \partial^i_s \right) \circ \left( \sum_{\sigma}(-1)^{\ell_1 + \dots + \ell_{p-1}+p-1 +\sharp(\sigma)}D_{p,\sigma} \right) \]
Using \eqref{reltheta5} we easily get:
\begin{equation*}
    \begin{split}
    (a) = \left( \sum_{s=1}^{p-1}\sum_{i=0}^{\kappa_s}(-1)^{\kappa_1 + \dots + \kappa_{s-1}+s-1+i} \phi\circ \partial^i_s \right) \circ \left( \sum_{\sigma}(-1)^{\ell_1 + \dots + \ell_{p-1}+p-1 +\sharp(\sigma)}D_{p,\sigma} \right) = \\
    -\sum_{s=1}^{p-1}\sum_{i=0}^{\kappa_s}\left(\sum_\sigma (-1)^{\kappa_1 + \dots + \kappa_{s-1}+s-1+i+\ell'_1 + \dots + \ell'_{p-1}+p-1 +\sharp(\sigma) } \phi\circ D_{p,\sigma}\circ \partial^i_s \right) = \\
    - \sum_{s=1}^{p-1}\sum_{i=0}^{\kappa_s} (-1)^{\kappa_1 + \dots + \kappa_{s-1}+s-1+i} \left(\sum_\sigma (-1)^{\ell'_1 + \dots + \ell'_{p-1}+p-1 +\sharp(\sigma) } \phi\circ D_{p,\sigma} \right) \circ \partial^i_s = \\
    - \left(\sum_\sigma (-1)^{\ell'_1 + \dots + \ell'_{p-1}+p-1 +\sharp(\sigma) } \phi\circ D_{p,\sigma} \right) \circ \left( \sum_{s=1}^{p-1}\sum_{i=0}^{\kappa_s} (-1)^{\kappa_1 + \dots + \kappa_{s-1}+s-1+i} \partial^i_s \right) = - (a')
    \end{split}
\end{equation*}
where: \[ \ell'_j = \begin{cases}
   \ell_j &\text{  if  } j \neq s \\
   \ell_s+1 &\text{otherwise}
\end{cases}\] 
Similarly:
\begin{equation*}
    \begin{split}
        (c) = \left( \sum_{s=p+2}^{n}\sum_{i=0}^{\kappa_s}(-1)^{\kappa_1 + \dots + \kappa_{s-1}+s-1+i} \phi\circ \partial^i_s \right) \circ \left( \sum_{\sigma}(-1)^{\ell_1 + \dots + \ell_{p-1}+p-1 +\sharp(\sigma)}D_{p,\sigma} \right) = \\
        - \sum_{s=p+2}^{n}\sum_{i=0}^{\kappa'_s} \left( \sum_\sigma (-1)^{\kappa'_1 + \dots + \kappa'_{s-2} +s-2+i+\ell_1+ \dots +\ell_{p-1}+p-1+\sharp(\sigma)} \phi\circ D_{p,\sigma}\circ \partial^i_{s-1} \right) = \\
        - \sum_{s=p+1}^{n-1}\sum_{i=0}^{\kappa'_s} (-1)^{\kappa'_1 + \dots + \kappa'_{s-1}+s-1+i} \left( \sum_\sigma (-1)^{\ell_1+ \dots +\ell_{p-1}+p-1+\sharp(\sigma)} \phi\circ D_{p,\sigma} \right) \circ \partial^i_s = \\
        - \left( \sum_\sigma (-1)^{\ell_1+ \dots +\ell_{p-1}+p-1+\sharp(\sigma)} \phi\circ D_{p,\sigma} \right) \circ \left( \sum_{s=p+1}^{n-1}\sum_{i=0}^{\kappa'_s} (-1)^{\kappa'_1 + \dots + \kappa'_{s-1}+s-1+i}\partial^i_s \right) = -(c')
    \end{split}
\end{equation*}
where: \[ \kappa'_i = \begin{cases}
   \kappa_i &\text{  if  } i < p ; \\
   \kappa_p + \kappa_{p+1} &\text{  if  } i = p ; \\
   \kappa_{i+1} &\text{  if  } i > p ;
\end{cases}\]

The last and more tricky summand is the following, which follows from \eqref{reltheta6}:
\begin{equation*}
    \begin{split}
        (b) = \left( \sum_{s=p,p+1}\sum_{a=0}^{\kappa_s}(-1)^{\kappa_1 + \dots + \kappa_{s-1}+s-1+a} \phi\circ \partial^a_s \right) \circ \left( \sum_{\sigma}(-1)^{\ell_1 + \dots + \ell_{p-1}+p-1 +\sharp(\sigma)}D_{p,\sigma} \right) = \\
        - \sum_{i=0}^{\kappa'_p}\left(\sum_{\tilde{\sigma}} (-1)^{\kappa'_1+ \dots +\kappa'_{p-1}+p-1+j+\ell_1+ \dots +\ell_{p-1}+p-1+\sharp(\tilde{\sigma})} \phi\circ D_{p,\tilde{\sigma}} \circ \partial^i_p \right) = \\
        - \left( \sum_{\sigma}(-1)^{\ell_1+ \dots +\ell_{p-1}+p-1+\sharp(\sigma)} \phi\circ D_{p,\sigma} \right) \circ \left( \sum_{i=0}^{\kappa'_p}(-1)^{\kappa'_1+ \dots +\kappa'_{p-1}+p-1+i}\partial^i_p\right) = -(b')
    \end{split}
\end{equation*}
where: \[ \kappa'_i = \begin{cases}
   \kappa_i &\text{  if  } i < p ; \\
   \kappa_p + \kappa_{p+1} &\text{  if  } i = p.
\end{cases}\]
and $\tilde{\sigma}$ is the ``extended" shuffle, which is explained in the following lemma:
\begin{lemma}
Let $\partial^a_p$ and $D_{p,\sigma}=(\alpha,\beta)\colon [\kappa'_p]\to [\kappa_p]\times[\kappa_{p+1}]$. Let $i\coloneqq \min\{j \in [\kappa'_p]\; |\; \alpha(j)= a \}$. 
Then we set $D_{p,\tilde{\sigma}} \coloneqq (\tilde{\alpha},$ $\tilde{\beta})\colon [\kappa'_p+1]\to [\kappa_p+1]\times[\kappa_{p+1}]$, where:
\[ \tilde{\alpha}(j) = \begin{cases}
   \alpha(j) &\text{  if  } j<i; \\
   a &\text{  if  } j=i; \\
   \alpha(j-1)+1 &\text{  if  } j>i.
\end{cases} \;\;\; \text{    and    }\;\;\;  \tilde{\beta}(j) = \begin{cases}
   \beta(j) &\text{  if  } j\leq i; \\
   \beta(j-1) &\text{  if  } j>i.
\end{cases} \]
Then we have: 
\begin{equation*}
    \begin{aligned}
        &(-1)^{\kappa_1 + \dots + \kappa_{p-1}+p-1+a+\ell_1 + \dots + \ell_{p-1}+p-1 +\sharp(\sigma)} \partial^a_p\circ D_{p,\sigma} = \\
        -&(-1)^{ \kappa'_1+ \dots +\kappa'_{p-1}+p-1+j+\ell_1+ \dots +\ell_{p-1}+p-1+\sharp(\tilde{\sigma})} D_{p,\tilde{\sigma}}\circ \partial^i_p
    \end{aligned}
\end{equation*}
and similarly for $\partial^b_{p+1}$ and $D_{p,\sigma}$.
\end{lemma}
\begin{proof}
Straightforward.
\end{proof}
Now the desired equality: $\Omega^p_{\Delta}(d\phi) = - d\Omega^p_\Delta(\phi)$ follows since \[ d\Omega^p_\Delta(\phi) = (a')+(b')+(c'). \]

(2) Now we need to prove that $\Omega^p_\Delta$ and $\Upsilon^q_\Delta$ satisfy the simplicial identities.
The first
\[ \Omega^p_\Delta\circ\Omega^q_\Delta = \Omega^{q-1}_\Delta\circ\Omega^p_\Delta \text{  if  } p < q; \]
follows directly from \eqref{reltheta1} for $p < q-1$; for $p=q-1$, it follows from \eqref{reltheta2}.

The second identity
\[ \Upsilon^p_\Delta\circ\Upsilon^q_\Delta = \Upsilon^{q+1}_\Delta\circ\Upsilon^{p}_\Delta \text{  if  } p \leq q; \]
follows directly from \eqref{reltheta10}.

Now the last identity: \[ \Omega^p_\Delta \circ \Upsilon^q_\Delta = \begin{cases}
   \Upsilon^{q-1}_\Delta\circ\Omega^p_\Delta &\text{  if  } p < q; \\
   \id &\text{  if  } p=q,q+1; \\
   \Upsilon^{q}_\Delta\circ\Omega^{p-1}_\Delta &\text{  if  } p >q+1.
\end{cases} \]
easily follows from \eqref{reltheta16}.
\end{proof}

\bigskip

\noindent{\small P.P.:
 {\sc Universiteit Antwerpen, Campus Middelheim, Wiskunde en Informatica, Gebouw G\\
Middelheimlaan 1, 2020 Antwerpen, Belgi\"{e}}}

\bigskip

\noindent{{\it e-mail}: {\tt panero1994@gmail.com}}

\bigskip

\noindent{\small B.S.:
 {\sc Euler International Mathematical Institute\\
10 Pesochnaya Embankment, St. Petersburg, 197376 Russia }}

\bigskip

\noindent{{\it e-mail}: {\tt shoikhet@pdmi.ras.ru}}


\begin{thebibliography}{999}
{\footnotesize

\bibitem[Ar]{Ar} D.~Ara, Higher quasi-categories vs higher Rezk spaces, {\it Journal of K-theory}, {\bf 14}(3) (2014), 701-749

\bibitem[B1]{B1} C.~Berger, A Cellular nerve for higher categories, {\it Advances in Math.} {\bf 169} (2002), 118-175

\bibitem[B2]{B2} C.~Berger, Iterated wreath product of the simplex category and iterated loop spaces, {\it Advances in Math.} {\bf 213} (2007), 230-270

\bibitem[Ba1]{Ba1} M.A.~Batanin, Monoidal globular categories as a natural environment for the theory of weak $n$-categories,
{\it Adv. Math.} {\bf 136}(1) (1998), 39-103


\bibitem[Ba2]{Ba2} M.A.~Batanin, The Eckmann-Hilton argument and higher operads, {\it Advances in Math.}, {\bf 217}(2008), 334-385


\bibitem[Ba3]{Ba3} M.A.~Batanin, The symmetrisation of $n$-operads and compactification of real configuration spaces, {\it Advances in Math.} {\bf 211}(2007), 684-725

\bibitem[BB]{BB} M.A.~Batanin, C.~Berger, The lattice path operad and Hochschild cochains, {\it Contemp. Math.} {\bf 504}(2009), 23-52

\bibitem[BD]{BD} M.~Batanin, A.~Davydov, Cosimplicial monoids and deformation theory of tensor categories, preprint 	arXiv:2003.13039



\bibitem[BM1]{BM1} M.A.~Batanin, M.~Markl, Centers and homotopy centers in enriched monoidal categories, {\it Adv. Math.} {\bf 230} (2012), no. 4-6, 1811-1858

\bibitem[BM2]{BM2} M.A.~Batanin, M.~Markl, Operadic categories and Duoidal Deligne's conjecture, {\it Advances in Math.} {\bf 285}(2015), 1630-1687


\bibitem[Be]{Be} J.Bénabou, Introduction to Bicategories, {\it Lecture Notes in Mathematics} {\bf 47} Springer (1967), 1-77

\bibitem[BR]{BR} J.Bergner, C.Rezk, Reedy categories and the $\Theta$-construction, {\it Math. Z.} {\bf 274}, 499–514 (2013)


\bibitem[D]{D} A.Davydov, Monoidal categories, {\it J.Math.Sci.}, {\bf 88} (1998), no.4, 457-519

\bibitem[EGNO]{EGNO} P.Etingof, S.Gelaki, D.Nikshych, V.Ostrik, {\it Tensor categories}, AMS Math Surveys and Monographs Volume 205


\bibitem[GPS]{GPS} R.~Gordon, A.J.~Power, R.~Street, Coherence for tricategories, {\it Mem. Amer. Math Soc.} {\bf 117} (1995) no. 558

\bibitem[GS]{GS} M.~Gerstenhaber, S.D.~Schack, Bialgebra cohomology, deformations, and quantum groups, {\it Proc. Natl. Acad. Sci. USA}, {\bf 87}(1990), 478-481

\bibitem[GLST]{GLST} A. Guan, A. Lazarev, Y. Sheng, R. Tong, Review of Deformation theory II: a homotopical approach, {\it Adv. Math. (China)} {\bf 49} (3) (2020), 278-298


\bibitem[Hir]{Hir} Ph. Hirschhorn, {\it Model Categories and Their Localizations}, AMS Math Surveys and Monographs Volume 99

\bibitem[J]{J} A.Joyal, Disks, duality, and $\Theta$-categories, preprint September 1997

\bibitem[K]{K} M.Kontsevich, Deformation quantization of Poisson manifolds, I,  {\it Lett. Math. Phys.} {\bf 66} (2003), 157-216

\bibitem[Ke]{Ke} G.M.~Kelly, {\it Basic concepts of enriched category theory}, Cambridge University Press, Lecture Notes in Mathematics 64, 1982

\bibitem[KS]{KS} M. Kontsevich, Y. Soibelman, Deformations of algebras over operads and Deligne's conjecture, in {\it Conference Moshe Flato 1999, Quantization, Deformations, and Symmetries}, vol. I, Ed. G.Dito and D.Sternheimer, Kluwer Academic Publishers, 2000, 255-307

\bibitem[L]{L} T.Leinster, Basic Bicategories, arXiv:math/9810017v1

\bibitem[L-DAGX]{L-DAGX} J.Lurie, Formal moduli problems, available at www.math.ias.edu/~lurie/



\bibitem[ML]{ML} S.~MacLane, {\it Categories for working mathematician}, Graduate Text in Mathematics 5, 2nd edition, Springer-Verlag, 1998

\bibitem[ML1]{ML1} S.~MacLane, {\it Homology}, Springer-Verlag, 1975 (reprinted edition 1995)


\bibitem[MS1]{MS1} J.~McClure, J.~Smith, A solution of Deligne's conjecture, In {\it Recent progress in homotopy theory (Baltimore, MD, 2000)}, {\it Contemp. Math.}, {\bf 293}(2002), 153-193 

\bibitem[MS2]{MS2} J.~McClure, J.~Smith, Multivariable cochain operations and little n-cubes, {\it J. Amer. Math. Soc.}, {\bf 16}(3)(2003), 681-704 

\bibitem[P]{P} J.P.Pridham, {\it Unifying derived deformation theories}, {\it Advances in Math.} {\bf 224} (2010), 772-826

\comment
\bibitem[PS]{PS} P.Panero, B.Shoikhet, The category $\Theta_2$, derived modifications, and deformation theory of monoidal categories, (the previous version of this paper, contains some computations omitted here), arxiv:2210.01664v1
\endcomment

\bibitem[R]{R} E.Riehl, {\it Categorical homotopy theory}, New Mathematical Monographs 28, Cambridge University Press, 2014




\bibitem[Sh1]{Sh1} B.~Shoikhet, Tetramodules over a bialgebra form a 2-fold monoidal category, {\it Applied Category Theory}, {\bf 21}(3), 2013, 291-309

\bibitem[Sh2]{Sh2} B.~Shoikhet, Differential graded categories and Deligne conjecture, {\it Advances in Math.}, {\bf 289}(2016), 797-843

\bibitem[Sh3]{Sh3} B.~Shoikhet, 2-shuffles, 3-operads, and Deligne conjecture in dimension 3, {\it in progress}


\bibitem[Tam1]{Tam1} D.Tamarkin, What do dg categories form?, {\it Compos. Math.} {\bf 143} (2007), no. 5, 1335–1358

\bibitem[Tam2]{Tam2} D.Tamarkin, Deformation complex of a $d$-algebra is a $(d+1)$-algebra, arXiv:math/0010072

\bibitem[Tam3]{Tam3} D.Tamarkin, Another proof of M. Kontsevich formality theorem, arXiv:math/9803025



\bibitem[W]{W} C.~Weibel, {\it An introduction to homological algebra}, Cambridge University Press, 1994

\bibitem[Y]{Y} D.Yetter, Braided deformations of monoidal categories and Vassiliev invariants, in: M.Kapranov, E.Getzler (Eds.), Higher Category Theory, AMS Contemporary Mathematics, Vol. 320, AMS, Providence, RI, 1998, 117-134.

}
\end{thebibliography}
\end{document}